\documentclass[article,onefignum,onetabnum]{siamonline220329}

\usepackage{algorithmic}
\usepackage{amsfonts}
\usepackage{amsmath}
\usepackage{amsopn}
\usepackage{amssymb}
\usepackage{booktabs}
\usepackage{url}
\usepackage{color}
\usepackage{array}
\usepackage{colortbl}
\usepackage{blkarray}
\usepackage{colortbl}
\usepackage{enumitem}
\setlist[enumerate]{leftmargin=.5in}
\setlist[itemize]{leftmargin=.5in}
\usepackage{epstopdf}
\usepackage{fancyhdr}
\usepackage{fancyvrb}
\usepackage{float}
\usepackage[T1]{fontenc}
\usepackage{framed}
\usepackage{geometry}
\usepackage{graphicx}
\usepackage{hyperref}
\usepackage[latin1]{inputenc}
\usepackage{lipsum}
\usepackage{multirow}
\usepackage{scalerel}
\usepackage{stmaryrd}
\usepackage{subfigure}
\usepackage{todonotes}
\usepackage{upquote}
\usepackage{xspace}

\definecolor{lightgray}{gray}{0.5}
\definecolor{darkgreen}{rgb}{0,0.6,0.13}

\newcommand{\nc}{\newcommand}
\nc{\dsp}{\displaystyle}
\nc{\txt}{\textstyle}
\nc{\creff}[1]{(\cref{#1})}
\nc{\mrm}[1]{\mathrm{#1}}
\nc{\udl}[1]{\underline{#1}}
\nc{\ovl}[1]{\overline{#1}}
\nc{\al}{\underline{\boldsymbol{\alpha}}}
\nc{\la}{\underline{\boldsymbol{\lambda}}}
\nc{\llbr}{\llbracket}
\nc{\rrbr}{\rrbracket}
\nc{\lbr}{\lbrack}
\nc{\rbr}{\rbrack}
\nc{\N}{\mathbb{N}}
\nc{\Z}{\mathbb{Z}}
\nc{\D}{\mathbb{D}}
\nc{\Q}{\mathbb{Q}}
\nc{\R}{\mathbb{R}}
\nc{\C}{\mathbb{C}}
\nc{\T}{\mathbb{T}}
\nc{\Stwo}{\mathbb{S}^2}
\nc{\tld}[1]{\tilde{#1}}
\nc{\wtld}[1]{\widetilde{#1}}
\nc{\hu}{\hat{u}}
\nc{\wh}[1]{\widehat{#1}}
\nc{\Fbf}{\textbf{F}}
\nc{\Gbf}{\textbf{G}}
\nc{\Lbf}{\textbf{L}}
\nc{\Nbf}{\textbf{N}}
\nc{\Ibf}{\textbf{I}}
\nc{\Dbf}{\textbf{D}} 
\nc{\Tbf}{\textbf{T}}
\nc{\Jbf}{\textbf{J}} 
\nc{\Rbf}{\textbf{R}}   
\nc{\ph}{\varphi}
\nc{\NN}{\mathcal{NN}}
\nc{\OO}{\mathcal{O}}
\nc{\cqfd}{~\hbox{\vrule width 2.5pt depth 2.5 pt height 3.5 pt}}
\nc{\ra}[1]{\renewcommand{\arraystretch}{#1}}
\nc{\bs}[1]{\boldsymbol{#1}}
\nc{\rr}[1]{\textcolor{red}{#1}}
\nc{\bb}[1]{\textcolor{blue}{#1}}
\nc{\eqvsp}{\\[0.25em]}
\nc{\eqvspp}{\\[0.5em]}
\nc{\eqVsp}{\\[0.75em]}
\nc{\eqVspp}{\\[1em]}
\nc{\dmu}{\delta\hspace{-0.025cm}\mu\hspace{0.02cm}}
\nc{\dnu}{\delta\hspace{-0.010cm}\nu\hspace{0.01cm}}
\nc{\dep}{\delta\hspace{-0.015cm}\epsilon\hspace{0.02cm}}
\newsiamremark{remark}{Remark}
\newsiamremark{hypothesis}{Hypothesis}
\crefname{hypothesis}{Hypothesis}{Hypotheses}
\newsiamthm{claim}{Claim}
\ifpdf
  \DeclareGraphicsExtensions{.eps,.pdf,.png,.jpg}
\else
  \DeclareGraphicsExtensions{.eps}
\fi

\usetikzlibrary{arrows.meta,bending,shapes,shapes.misc}
\tikzset{cross/.style={cross out, ultra thick, draw, fill, minimum size=2*(#1-\pgflinewidth), inner sep=0pt, outer sep=0pt}, cross/.default={7pt}}

\title{The linear sampling method for random sources\thanks{Submitted to the editors 27 October 2022.
\funding{This work was supported by the Interdisciplinary Centre for Defence and Security (CIEDS).}}}

\headers{The linear sampling method for random sources}{J. Garnier, H. Haddar, and H. Montanelli}

\ifpdf
\hypersetup{
  pdftitle={The linear sampling method for random sources},
  pdfauthor={J. Garnier, H. Haddar, and H. Montanelli}
}
\fi

\author{Josselin Garnier\thanks{CMAP, \'{E}cole Polytechnique, IP Paris, 91120 Palaiseau, France.}
\and Houssem Haddar\thanks{Inria, UMA, Ensta Paris, IP Paris, 91120 Palaiseau, France.}
\and Hadrien Montanelli\footnotemark[3]}

\begin{document}

\maketitle

\begin{abstract}
We present an extension of the linear sampling method for solving the sound-soft inverse acoustic scattering problem with randomly distributed point sources. The theoretical justification of our sampling method is based on the Helmholtz--Kirchhoff identity, the cross-correlation between measurements, and the volume and imaginary near-field operators, which we introduce and analyze. Implementations in MATLAB using boundary elements, the SVD, Tikhonov regularization, and Morozov's discrepancy principle are also discussed. We demonstrate the robustness and accuracy of our algorithms with several numerical experiments in two dimensions.
\end{abstract}



\section{Introduction}

A typical inverse scattering problem is the identification of the shape of a defect inside a medium by sending waves that propagate within. In the data acquisition step, several receivers record the medium's response that forms the data of the inverse problem; in the data processing step, numerical algorithms are used to recover an approximation of the shape from the measurements. What we have just described is an example of \textit{active} imaging, where both the sources and the receivers are controlled. In \textit{passive} imaging, only receivers are employed and the illumination comes from uncontrolled, random sources. In this setup, it is the cross-correlations between the recorded signals that convey information about the medium through which the waves propagate \cite{garnier2016, gouedard2008}. This information can be exploited, e.g., for velocity estimation \cite{curtis2006} or reflector imaging \cite{garnier2009, garnier2010}. Particularly important applications include crystal tomography and volcano monitoring by seismic interferometry \cite{koulakov2014, shapiro2005}. In seismology applications, the sources are typically microseisms and ocean swells. Passive imaging has also been successful in other domains, such as structural health monitoring \cite{duroux2010, sabra2011}, oceanography \cite{godin2010, siderius2010, woolfe2015}, or medical elastography \cite{gallot2011}.

We deal, in this paper, with the data processing step for the sound-soft inverse acoustic scattering problem, around resonance, in passive imaging. While there are many techniques available for this problem in active imaging, including PDE-constrained optimization \cite{bourgeois2012}, Newton \cite{hohage1998, kirsch1993}, level-set \cite{dorn2006}, factorization \cite{kirsch2007} and sampling methods \cite{colton2003, colton1996}, only linearization approaches have been developed so far in passive imaging \cite{ammari2012b, ammari2012a}. However, around resonance, inverse scattering problems are highly nonlinear---linearization strategies cannot be employed. On top of their long reconstruction times, PDE-constrained optimization, and Newton and level-set methods rely on some a priori information to initialize the corresponding iterative procedures. This is not the case for factorization and sampling methods, which have notable computational speed and require very little a priori information on the scatterer. The linear sampling method (LSM) goes back to Colton and Kirsch in 1996 \cite{colton1996}; regularization was introduced the following year \cite{colton1997}, and significant numerical validation was reported in 2003 \cite{colton2003}. Several extensions to the LSM have been proposed in the last two decades, including the generalized LSM, based on an exact characterization of the target's shape in terms of the far-field operator, for full- \cite{audibert2014} and limited-aperture measurements \cite{audibert2017}. Details about the history and the evolution of the LSM can be found in the 2018 SIAM review article of Colton and Kress \cite{colton2018}; details about the mathematics in the books \cite{cakoni2014, cakoni2016, colton2019}. For factorization methods, we refer to the book of Kirsch and Grinberg \cite{kirsch2007}. Links between sampling and factorization methods can be found in \cite{cakoni2016}. To conclude this paragraph, it is worth highlighting that when it comes to high-definition reconstructions, iterative methods may prove to be the most suitable approach. Combining sampling and iterative methods can be advantageous in such cases, e.g., using the former to initialize the latter, as it allows for the benefits of both methods to be leveraged \cite{bao2007}.

We propose, in this paper, an extension of the LSM for solving the sound-soft inverse acoustic scattering problem with randomly distributed point sources. The theoretical justification of our method is based on the Helmholtz--Kirchhoff identity and the cross-correlation between measurements (\cref{sec:HK-id}), and on the mathematical properties of the volume (\cref{sec:vol-near-field}) and imaginary near-field operators  (\cref{sec:imag-near-field}). Implementations in MATLAB using boundary elements, the SVD, Tikhonov regularization, and Morozov's discrepancy principle are also discussed, together with numerical experiments in two dimensions (\cref{sec:numerics}). 

\section{The Helmholtz--Kirchhoff identity and cross-correlation}\label{sec:HK-id}

Acoustic scattering is governed by the Helmholtz equation $\Delta u+k^2u=0$, whose Green's function is\footnote{The Green's function is the solution to $\Delta u(\bs{x},\bs{y})+k^2u(\bs{x},\bs{y})=\delta(\bs{x}-\bs{y})$ in $\R^d$.}
\begin{align}\label{eq:phi}
\phi(\bs{x},\bs{y}) = \frac{i}{4}H_0^{(1)}(k\vert\bs{x}-\bs{y}\vert) \quad (d=2), \quad \quad \phi(\bs{x},\bs{y}) = \frac{e^{ik\vert\bs{x}-\bs{y}\vert}}{4\pi \vert\bs{x}-\bs{y}\vert} \quad (d=3).
\end{align}
The number $k>0$ is the wavenumber and $H^{(1)}_0$ denotes the Hankel function of the first kind of order $0$. The Helmholtz--Kirchhoff identity for the Green's function reads \cite[Thm.~2.2]{garnier2016} 
\begin{align}\label{eq:HK}
\phi(\bs{x},\bs{y}) - \overline{\phi(\bs{x},\bs{y})} = 2ik\int_{\Sigma}\overline{\phi(\bs{x},\bs{z})}\phi(\bs{y},\bs{z})dS(\bs{z}),
\end{align}
where the surface $\Sigma$ encloses $\bs{x}$ and $\bs{y}$ and is far from them. The identity \cref{eq:HK} connects the imaginary part of the measurements at $\bs{x}$ of point sources located at $\bs{y}$ (left) to the cross-correlation between the measurements at $\bs{x}$ and $\bs{y}$ of point sources located on $\Sigma$ (right).

The Helmholtz--Kirchhoff identity applies to total fields, too. Let $D\subset\R^d$ be a bounded domain whose complement is connected and $\bs{y}\in\R^d$. A point source located at $\bs{y} \in \R^d\setminus\overline{D}$ transmits a unit-amplitude time-harmonic signal that generates the incident field $\phi(\cdot,\bs{y})$. The scattered field $u^s(\cdot,\bs{y})\in H^1_{\mrm{loc}}(\R^d\setminus D)$ for a sound-soft defect $D$ is the solution to
\begin{align}\label{eq:ext-Dirichlet}
\left\{
\begin{array}{ll}
\Delta u^s(\cdot,\bs{y}) + k^2 u^s(\cdot,\bs{y}) = 0 \quad \text{in $\R^d\setminus\overline{D}$}, \\[0.4em]
\phi(\cdot,\bs{y}) + u^s(\cdot,\bs{y}) = 0 \quad \text{on $\partial D$}, \\[0.4em]
\text{$u^s(\cdot,\bs{y})$ is radiating}.
\end{array}
\right.
\end{align}
Note that the (Sommerfeld) radiation condition in \cref{eq:ext-Dirichlet} reads
\begin{align}\label{eq:Sommerfeld}
\lim_{r\to\infty} r^{\frac{d-1}{2}}\left(\frac{\partial u^s}{\partial r} - iku^s\right) = 0, \quad r = \vert\bs{x}\vert \quad \text{(uniformly in $\bs{x}/\vert\bs{x}\vert$)}.
\end{align}
The total field is then defined as $u(\bs{x},\bs{y})=\phi(\bs{x},\bs{y})+u^s(\bs{x},\bs{y})$. Similar arguments to those used in the proof of \cite[Thm.~2.2]{garnier2016} lead to the Helmholtz--Kirchhoff identity for total fields,
\begin{align}\label{eq:HK-total}
u(\bs{x},\bs{y}) - \overline{u(\bs{x},\bs{y})} = 2ik\int_{\Sigma}\overline{u(\bs{x},\bs{z})}u(\bs{y},\bs{z})dS(\bs{z}).
\end{align}
This immediately implies the following relationship for the scattered fields,
\begin{align}\label{eq:HK-scattered}
u^s(\bs{x},\bs{y}) - \overline{u^s(\bs{x},\bs{y})} = 2ik\int_{\Sigma}\overline{u(\bs{x},\bs{z})}u(\bs{y},\bs{z})dS(\bs{z}) - \left[\phi(\bs{x},\bs{y})-\overline{\phi(\bs{x},\bs{y})}\right].
\end{align}

Let $\bs{x}_j$, $1\leq j\leq J$, and $\bs{y}_m$, $1\leq m\leq M$, be measurement points and point sources both located on a surface $\partial B$ that encloses $D$. The standard LSM for near-field measurements, in active imaging, relies on the construction of the near-field matrix $N^\mrm{ac}$ with entries
\begin{align}\label{eq:near-field-mat-ac}
N^\mrm{ac}_{jm} = u^s(\bs{x}_j,\bs{y}_m), \quad 1\leq j\leq J, \quad 1\leq m\leq M.
\end{align}
The matrix $N^\mrm{ac}$ corresponds to the medium's response, measured at $\bs{x}_j$, to the illumination by point sources located at $\bs{y}_m$; see \cref{fig:active-LSM}. Here, the positions of both measurement points and point sources are known and controlled. This matrix is filled out in the data acquisition step, which entails the solution of \cref{eq:ext-Dirichlet} for each $\bs{y}_m$. In the processing step, we numerically probe the medium by solving the linear system $N^\mrm{ac}g_{\bs{z}}=\phi_{\bs{z}}$ for various $\bs{z}\in\R^d$. (The right-hand side is defined by $(\phi_{\bs{z}})_j = \phi(\bs{x}_j, \bs{z})$, $1\leq j\leq J$.) The boundary $\partial D$ of the defect $D$ coincides with those points $\bs{z}$ for which the value of $\Vert g_{\bs{z}}\Vert_2$ is large \cite[sect.~5.6]{colton2019}.

\begin{figure}
\centering
\begin{tikzpicture}[scale=1.0]

\def\R{2.5};
\def\s{0.55};
\def\t{7};
\def\x{-0.9238795325112867};
\def\y{0.38268343236508984};
\def\xx{-0.7071067811865475};
\def\yy{0.7071067811865476};
\def\theta{10};

\draw[black, thick, fill=black!10] plot[smooth cycle] coordinates {(-2*\R/\t,-1.5*\R/\t) (-2*\R/\t,1.5*\R/\t) (2*\R/\t,\R/\t) (2*\R/\t,-\R/\t)};
\node[anchor=north] at (0,0.5*\R/\t) {$D$};

\draw[black, dashed] (0,0) circle (\R);
\node[anchor=south] at (0,1.05*\R) {$\partial B$};

\node[draw, fill, circle, scale=\s] at (\R*-1, \R*0) {};
\node[anchor=east] at (\R*-1.1,\R*0) {$\bs{y}_m$};
\node[draw, fill, circle, scale=\s] at (\R*\x, \R*\y) {};
\node[draw, fill, circle, scale=\s] at (\R*\x, -\R*\y) {};
\node[draw, fill, circle, scale=\s] at (\R*\xx, \R*\yy) {};
\node[draw, fill, circle, scale=\s] at (\R*\xx, -\R*\yy) {};

\node[cross, rotate=\theta, scale=\s] at (\R*1, \R*0) {};
\node[anchor=west] at (\R*1.1,\R*0) {$\bs{x}_j$};
\node[cross, rotate=\theta, scale=\s] at (-\R*\x, \R*\y) {};
\node[cross, rotate=\theta, scale=\s] at (-\R*\x, -\R*\y) {};
\node[cross, rotate=\theta, scale=\s] at (-\R*\xx, \R*\yy) {};
\node[cross, rotate=\theta, scale=\s] at (-\R*\xx, -\R*\yy) {};

\end{tikzpicture}
\caption{In active imaging, the defect $D$ is illuminated by controlled, deterministic point sources located at $\bs{y}_m$; the medium's response is recorded at $\bs{x}_j$. Here, the $\bs{y}_m$'s and the $\bs{x}_j$'s are located on a surface $\partial B$ that encloses $D$---more general configurations, including limited-aperture measurements, are allowed \normalfont{\cite{audibert2017}}.}
\label{fig:active-LSM}
\end{figure}
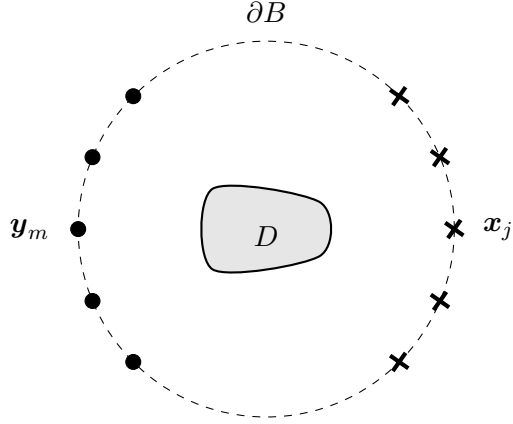

We now turn our attention to passive imaging. In the first setup, we assume that the measurement points $\bs{x}_j$ are located in some bounded volume $B\subset\R^d\setminus\overline{D}$ (as opposed to a surface $\partial B$; we shall justify this later). Let $\Sigma$ be a surface that encloses $B$ and $D$, and let us assume that there are $L>0$ point sources $\bs{z}_\ell$ randomly distributed on $\Sigma$. These sources can transmit a unit-amplitude time-harmonic signal, one by one, so that it is possible to measure the total fields $u(\bs{x}_j,\bs{z}_\ell)$. Moreover, it is possible to compute $\phi(\bs{x}_j,\bs{x}_m)$, so we can evaluate the cross-correlation matrix $C$ with entries
\begin{align}\label{eq:cross-cor-mat}
C_{jm} = \frac{2ik\vert\Sigma\vert}{L}\sum_{\ell=1}^L\overline{u(\bs{x}_j,\bs{z}_\ell)}u(\bs{x}_m,\bs{z}_\ell) - \left[\phi(\bs{x}_j,\bs{x}_m) - \overline{\phi(\bs{x}_j,\bs{x}_m)}\right], \quad 1\leq j,m\leq J,
\end{align}
where $\vert\Sigma\vert$ is the area of $\Sigma$. The matrix \cref{eq:cross-cor-mat} corresponds to the discretization of the right-hand side of \cref{eq:HK-scattered} at points $\bs{z}_\ell$ with uniform weights and evaluated at $\bs{x}_j$ and $\bs{x}_m$, i.e.,
\begin{align}\label{eq:trap-rule-approx}
C_{jm} \approx 2ik\int_{\Sigma}\overline{u(\bs{x}_j,\bs{z})}u(\bs{x}_m,\bs{z})dS(\bs{z}) - \left[\phi(\bs{x}_j,\bs{x}_m)-\overline{\phi(\bs{x}_j,\bs{x}_m)}\right].
\end{align}
The quadrature error in \cref{eq:trap-rule-approx} is $\OO(1/\sqrt{L})$ in general, in both two and three dimensions. In two dimensions, if the $\bs{z}_\ell$'s correspond to a $\beta$-perturbed trapezoidal rule, then the error improves to $\OO(1/L^{\sigma - 4\beta})$ whenever the integrand has $\sigma>4\beta+1/2$ derivatives \cite[Thm.~1]{austin2017b}. We will come back to this in \cref{sec:numerics}. Combining \cref{eq:trap-rule-approx} with \cref{eq:HK-scattered} we obtain
\begin{align}\label{eq:HK-MC}
C_{jm} \approx N_{jm} - \overline{N_{jm}}, \quad 1\leq j,m\leq J,
\end{align}
where the matrix $N$ is the near-field matrix \cref{eq:near-field-mat-ac} for co-located receivers and sources:
\begin{align}\label{eq:near-field-mat}
N_{jm} = u^s(\bs{x}_j,\bs{x}_m), \quad 1\leq j,m\leq J.
\end{align}
Note that the relationship \cref{eq:HK-MC} relates the imaginary part of the near-field matrix (right) to the cross-correlation between the total fields generated by the random sources $\bs{z}_\ell$ (left). This motivates the introduction of the imaginary near-field matrix $I$,
\begin{align}\label{eq:imag-near-field-mat}
I_{jm} = N_{jm} - \overline{N_{jm}} = u^s(\bs{x}_j,\bs{x}_m) - \overline{u^s(\bs{x}_j,\bs{x}_m)}, \quad 1\leq j,m\leq J.
\end{align}
In the rest of the paper, we will justify the utilization of the matrix \cref{eq:imag-near-field-mat} for the LSM in active imaging. If the LSM works for \cref{eq:imag-near-field-mat} in active imaging, then, via \cref{eq:HK-MC}, it will work for \cref{eq:cross-cor-mat} in passive imaging. The setup we have just described is illustrated in \cref{fig:passive-LSM}.

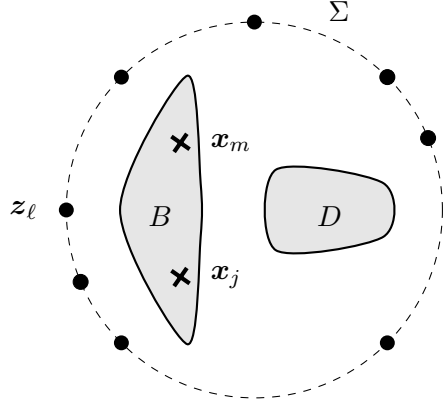
\begin{figure}
\centering
\begin{tikzpicture}[scale=1.0]

\def\R{2.5};
\def\s{0.55};
\def\t{7};
\def\x{-0.9238795325112867};
\def\y{0.38268343236508984};
\def\xx{-0.7071067811865475};
\def\yy{0.7071067811865476};
\def\z{0.7071067811865476};
\def\theta{10};

\draw[black, thick, fill=black!10, xshift=2/5*\R cm] plot[smooth cycle] coordinates {(-2*\R/\t,-1.5*\R/\t) (-2*\R/\t,1.5*\R/\t) (2*\R/\t,\R/\t) (2*\R/\t,-\R/\t)};
\node[anchor=north] at (2/5*\R,0.5*\R/\t) {$D$};

\draw[black, thick, fill=black!10] plot[smooth cycle] coordinates {(-2.5*\R/\t,5*\R/\t) (-5*\R/\t,0*\R/\t) (-2.5*\R/\t,-5*\R/\t) (-2*\R/\t,-\R/\t) (-2*\R/\t,\R/\t)};
\node[anchor=north] at (-3.5*\R/\t,0.5*\R/\t) {$B$};

\draw[black, dashed] (0,0) circle (\R);
\node[anchor=south] at (0.45*\R,0.95*\R) {$\Sigma$};

\node[fill, circle, scale=\s] at (\R*-1, \R*0) {};
\node[anchor=east] at (\R*-1.1,\R*0) {$\bs{z}_\ell$};
\node[fill, circle, scale=\s] at (\R*0, \R*1) {};
\node[draw, fill, circle, scale=\s] at (\R*\x, -\R*\y) {};
\node[draw, fill, circle, scale=\s] at (-\R*\x, \R*\y) {};
\node[draw, fill, circle, scale=\s] at (-\R*\xx, \R*\yy) {};
\node[fill, circle, scale=\s] at (\R*\z, \R*\z) {};
\node[fill, circle, scale=\s] at (\R*\z, -\R*\z) {};
\node[fill, circle, scale=\s] at (-\R*\z, \R*\z) {};
\node[fill, circle, scale=\s] at (-\R*\z, -\R*\z) {};

\node[cross, rotate=\theta, scale=\s] at (-2.75*\R/\t,2.5*\R/\t) {};
\node[anchor=west] at (-2*\R/\t,2.5*\R/\t) {$\bs{x}_m$};
\node[cross, rotate=\theta, scale=\s] at (-2.75*\R/\t,-2.5*\R/\t) {};
\node[anchor=west] at (-2*\R/\t,-2.5*\R/\t) {$\bs{x}_j$};

\end{tikzpicture}
\caption{In passive imaging, the defect $D$ is illuminated by uncontrolled, random sources located at points $\bs{z}_\ell$, which we assume to be distributed on a surface $\Sigma$. The medium's response is recorded at points $\bs{x}_j$ and $\bs{x}_m$, contained in some volume $B$; cross-correlation between these measurements is then carried out. This is equivalent, via the relationship \cref{eq:HK-MC}, to recording at $\bs{x}_j$ the imaginary part of the illumination by controlled, deterministic sources located at $\bs{x}_m$.}
\label{fig:passive-LSM}
\end{figure}

It is possible to consider another setup. We still assume that the measurement points $\bs{x}_j$  are located in some bounded volume $B\subset\R^d\setminus\overline{D}$. We assume that a noise source distribution localized on a surface $\Sigma$ enclosing $B$ and $D$ transmits random signals $(n(\bs{z}))_{\bs{z}\in\Sigma}$ such that
\begin{align}\label{eq:noise-distribution}
\left<n(\bs{z}) n(\bs{z}')\right> = \delta(\bs{z}-\bs{z}') \delta_{\Sigma}(\bs{z}),
\end{align}
where $\left<\cdot\right>$ stands for a statistical average. In other words, the noise source distribution is delta-correlated in space and uniformly distributed on the surface $\Sigma$. This random source generates the random incident field
\begin{align}
{\cal U}^{i}(\bs{x}) = \int_\Sigma \phi(\bs{x},\bs{z}) n(\bs{z}) dS(\bs{z}).
\end{align}
The random scattered field ${\cal U}^s $ for a sound-soft defect $D$ is the solution to
\begin{align} 
\left\{
\begin{array}{ll}
\Delta{\cal U}^s  + k^2 {\cal U}^s = 0 \quad \text{in $\R^d\setminus\overline{D}$}, \\[0.4em]
{\cal U}^{i} + {\cal U}^s  = 0 \quad \text{on $\partial D$}, \\[0.4em]
\text{${\cal U}^s$ is radiating}.
\end{array}
\right.
\end{align}
The total field is ${\cal U}={\cal U}^{i} + {\cal U}^s$ and it is of the form
\begin{align}
{\cal U}(\bs{x}) = \int_\Sigma u(\bs{x},\bs{z}) n(\bs{z}) dS(\bs{z}).
\end{align}
It is a random field with mean zero and covariance
\begin{align}\label{eq:covariance}
\left<\mathcal{U}(\bs{x}) \overline{\mathcal{U}(\bs{x}')}\right> = \int_\Sigma u(\bs{x},\bs{z})\overline{u(\bs{x}',\bs{z})} dS(\bs{z}). 
\end{align}
From the measurements ${\cal U}(\bs{x}_j)$ of the total field ${\cal U}(\bs{x})$ and the computed $\phi(\bs{x}_j,\bs{x}_m)$, we can compute the covariance matrix
\begin{align}
C_{jm} = 2ik\left<\mathcal{U}(\bs{x}_j) \overline{\mathcal{U}(\bs{x}_m)}\right> - \left[\phi(\bs{x}_j,\bs{x}_m) - \overline{\phi(\bs{x}_j,\bs{x}_m)}\right], \quad 1\leq j,m\leq J,
\end{align}
where the statistical average can be estimated by an empirical over repeated measurements with independent realizations of the source term $(n(\bs{z}))_{\bs{z}\in \Sigma}$. By \cref{eq:HK-scattered} we have that
\begin{align} 
C_{jm}  = N_{jm} - \overline{N_{jm}} = I_{jm},
\end{align}
where $I$ is the imaginary near-field matrix \cref{eq:imag-near-field-mat}.

Let us conclude this section with a few comments. The reason why we need to consider a volume $B$, as opposed to a surface $\partial B$, is because the scattered fields $u^s-\overline{u^s}$ in the right-hand side of \cref{eq:HK-MC} do not satisfy the radiation condition \cref{eq:Sommerfeld}. We will study the properties of the volume near-field operator in a volume $B$ in the next section (the near-field operator is the continuous analogue of the near-field matrix \cref{eq:near-field-mat}); the imaginary volume near-field operator will be studied in the following section (the continuous analogue of \cref{eq:imag-near-field-mat}).

\section{The volume near-field operator}\label{sec:vol-near-field}

Let $B\subset\R^d\setminus\overline{D}$ be a bounded domain and define
\begin{align}\label{eq:W(B)}
W(B) = \{g\in L^2(B)\,:\,\Delta g + k^2g = 0\;\text{in $B$}\},
\end{align}
which is a Hilbert space equipped with the $L^2(B)$-scalar product. We assume that $\partial D$ and $\partial B$ are smooth enough to allow the forming of Dirichlet and Neumann traces and the application of partial integration formulas (Lipschitz continuity is a sufficient condition). We start by introducing so-called volume potentials, which play the role of Herglotz wave functions in the standard LSM \cite[Def.~3.26]{colton2019}. 

\begin{definition}[Volume potential]\label{def:vol-pot} 
A volume potential is a function
\begin{align}
v(\bs{x}) = \int_{B} \phi(\bs{x},\bs{y})g(\bs{y})d\bs{y}, \quad \bs{x}\in\R^d,
\end{align}
for some $g\in W(B)$. It satisfies the inhomogeneous Helmholtz equation $\Delta v + k^2 v = -g\chi_{B}$ in $\R^d$, where $\chi_{B}$ denotes the characteristic function of the domain $B$.
\end{definition}

In the rest of the paper, we will denote by $u^s(\cdot,\bs{y})\in H^1_{\mrm{loc}}(\R^d\setminus D)$ the solution to the scattering problem \cref{eq:ext-Dirichlet} for the incident wave $\phi(\cdot,\bs{y})$. By linearity with respect to $\phi(\cdot,\bs{y})$, for a given kernel $g\in W(B)$, the solution $v^s\in H^1_{\mrm{loc}}(\R^d\setminus D)$ to the scattering problem
\begin{align}
\left\{
\begin{array}{ll}
\Delta v^s + k^2 v^s = 0 \quad \text{in $\R^d\setminus\overline{D}$}, \\[0.4em]
v^i + v^s = 0 \quad \text{on $\partial D$}, \\[0.4em]
\text{$v^s$ is radiating},
\end{array}
\right.
\end{align}
for the incident wave
\begin{align}
v^i(\bs{x}) = \int_{B} \phi(\bs{x},\bs{y})g(\bs{y})d\bs{y}, \quad \bs{x}\in\R^d,
\end{align}
is given by
\begin{align}
v^s(\bs{x}) = \int_{B} u^s(\bs{x},\bs{y})g(\bs{y})d\bs{y}, \quad \bs{x}\in\R^d\setminus\overline{D},
\end{align}
and has near-field pattern $v^s|_{B}$.

We are now ready to introduce near-field operators in a volume $B$, which we call volume near-field operators (near-field operators are usually defined on surfaces $\partial B$; see \cite[sect.~5]{audibert2017}). The following theorem mirrors that of the far-field operator $F$ \cite[Thm.~3.30]{colton2019}.

\begin{theorem}[Volume near-field operator]\label{thm:near-field-op} 
The volume near-field operator $N$,
\begin{align}
& N:W(B)\to W(B), \\
& (Ng)(\bs{x}) = \int_{B}u^s(\bs{x},\bs{y})g(\bs{y})d\bs{y}, \quad \bs{x}\in B, \nonumber
\end{align}
is injective and has dense range if $k^2$ is not a Dirichlet eigenvalue of $-\Delta$ in $D$.
\end{theorem}

\begin{proof} 
We first show that $N$ is injective. Let $g\in W(B)$ and consider the scattered field
\begin{align}
v^s(\bs{x}) = \int_{B} u^s(\bs{x},\bs{y})g(\bs{y})d\bs{y}, \quad \bs{x}\in\R^d\setminus\overline{D}.
\end{align}
It is a radiating solution to the Helmholtz equation in $\R^d\setminus\overline{D}$ for the incident wave
\begin{align}
v^i(\bs{x}) = \int_{B} \phi(\bs{x},\bs{y})g(\bs{y})d\bs{y}, \quad \bs{x}\in\R^d,
\end{align}
with near-field pattern $v^s|_{B}=Ng$. Suppose that $Ng=0$. This leads to $v^s=0$ in $\R^d\setminus\overline{D}$  by the unique continuation principle. Moreover, the regularity of $v^s\in H^1_{\mrm{loc}}(\R^d\setminus D)$ at the boundary implies that $v^s=0$ on $\partial D$, and the boundary condition $v^i+v^s=0$ on $\partial D$ leads to $v^i=0$ on $\partial D$. Since $v^i$ is a solution to the Helmholtz equation in $D$, $v_i|_{\partial D}=0$ gives $v_i=0$ in~$D$ (since $k^2$ is not a Dirichlet eigenvalue of $-\Delta$ in $D$). Furthermore, since $v^i$ is a solution to the Helmholtz equation in $\R^d\setminus\overline{B}$, this means that $v^i=0$ in $\R^d\setminus\overline{B}$ (by the unique continuation principle). Besides, the regularity of $v^i\in H^2_{\mrm{loc}}(\R^d)$ yields $v^i=\partial v^i/\partial n=0$ on $\partial B$, i.e., $v^i\in H^2_0(B)$. Now, taking the $L^2(B)$-scalar product of
\begin{align}
\Delta v^i + k^2 v^i = -g \quad \text{in $B$}
\end{align}
with $g\in W(B)$, we obtain
\begin{align}
\int_{B} (\Delta v^i(\bs{x}) + k^2v^i(\bs{x}))\overline{g}(\bs{x})d\bs{x} = -\Vert g\Vert_{L^2(B)}^2.
\end{align}
Since $v^i\in H_0^2(B)$ and $\Delta g + k^2g=0$ in $B$, the left-hand side vanishes---and so does $g$.

To show that $N$ has dense range, we note that the dual operator $N^*$,
\begin{align}
& N^*:W(B)\to W(B), \\
& (N^*g)(\bs{x}) = \int_{B}\overline{u^s(\bs{x},\bs{y})}g(\bs{y})d\bs{y}, \quad \bs{x}\in B, \nonumber
\end{align}
is injective since $N^*g=\overline{N\overline{g}}$.
\end{proof}

The conjugate operator $\overline{N}$,
\begin{align}
& \overline{N}:W(B)\to W(B), \\
& (\overline{N}g)(\bs{x}) = \int_{B}\overline{u^s(\bs{x},\bs{y})}g(\bs{y})d\bs{y}, \quad \bs{x}\in B, \nonumber
\end{align}
coincides with the dual operator $N^*$ and indeed shares the same properties as $N$.

A key step in the analysis of the LSM is to factorize $N$. We shall prove in \cref{thm:factorization} that $N$ admits the factorization $N=-AV$, with $V$ and $A$ described in \cref{thm:vol-op} and \cref{thm:bnd-near-field-op} below. These operators are the analogues of the Herglotz operator $H$ and the boundary-to-far-field operator $A$ characterized in \cite[Cor.~5.32]{colton2019} and \cite[Cor.~5.33]{colton2019} for the far-field operator $F=-AH$. Their domain and range are illustrated in \cref{fig:operators}.

\begin{figure}
\centering
\begin{tikzpicture}[scale=1.0]

\def\R{2.5};
\def\s{0.55};
\def\t{7};
\def\x{-0.9238795325112867};
\def\y{0.38268343236508984};
\def\xx{-0.7071067811865475};
\def\yy{0.7071067811865476};
\def\z{0.7071067811865476};
\def\theta{10};

\draw[black, thick, fill=black!10, xshift=5/5*\R cm] plot[smooth cycle] coordinates {(-2*\R/\t,-1.5*\R/\t) (-2*\R/\t,1.5*\R/\t) (2*\R/\t,\R/\t) (2*\R/\t,-\R/\t)};
\node[anchor=north] at (5/5*\R,0.5*\R/\t) {$D$};
\node[anchor=west] at (2*\R/\t+5.25/5*\R,0.8*\R/\t) {$\partial D$};

\draw[black, thick, fill=black!10] plot[smooth cycle] coordinates {(-2.5*\R/\t,5*\R/\t) (-5*\R/\t,0*\R/\t) (-2.5*\R/\t,-5*\R/\t) (-2*\R/\t,-\R/\t) (-2*\R/\t,\R/\t)};
\node[anchor=north] at (-3.5*\R/\t,0.5*\R/\t) {$B$};

\node (A) at (-3*\R/\t,2*\R/\t) {};
\node (B) at (2*\R/\t+4/5*\R,1.05*\R/\t) {};
\draw[thick, black, -{Stealth[bend]}] (A) to [bend left=60] (B);
\node[anchor=south] at (4.5*\R/\t,5*\R/\t) {$V:W(B)\to H^{1/2}(\partial D)$};
\node (C) at (2*\R/\t+4/5*\R,-1.05*\R/\t) {};
\node (D) at (-3*\R/\t,-2*\R/\t) {};
\draw[thick, black, -{Stealth[bend]}] (C) to [bend left=60] (D);
\node[anchor=north] at (4.5*\R/\t,-5.25*\R/\t) {$A:H^{1/2}(\partial D)\to W(B)$};
\node (E) at (-7.5*\R/\t,0) {};
\draw[thick, black, -{Stealth[bend]}] (A) to [bend right=60] (E) to [bend right=60] (D);
\node (F) at (-5.5*\R/\t,-3.5*\R/\t) {};
\node[anchor=east] at (F) {$N:W(B)\to W(B)$};

\end{tikzpicture}
\caption{The near-field operator $N$ is related to the operators $V$ and $A$ via the factorization $N=-AV$. The operator $V$ corresponds to a superposition of point sources located in $B$ and evaluated on $\partial D$, while the operator $A$ maps the boundary values on $\partial D$ of radiating solutions to the Helmholtz equation onto near-field measurements in $B$.}
\label{fig:operators}
\end{figure}
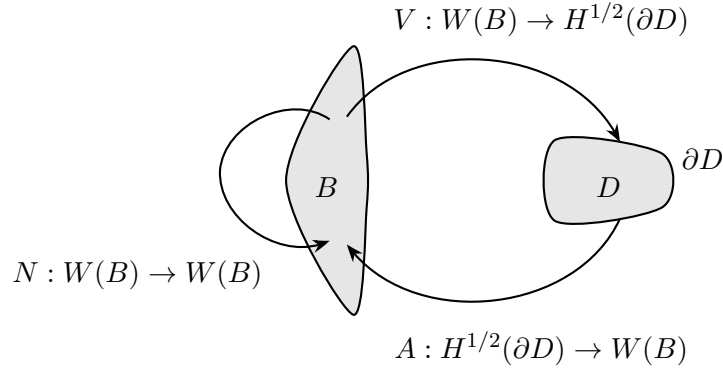

\begin{theorem}[Volume operator]\label{thm:vol-op}
The volume operator $V$,
\begin{align}
& V:W(B)\to H^{1/2}(\partial D), \\
& (Vg)(\bs{x}) = \int_{B}\phi(\bs{x},\bs{y})g(\bs{y})d\bs{y}, \quad \bs{x}\in\partial D, \nonumber
\end{align}
is injective and has dense range if $k^2$ is not a Dirichlet eigenvalue of $-\Delta$ in $D$. 
\end{theorem}

\begin{proof}
We start with the injectivity. Let $g\in W(B)$ and consider the volume potential
\begin{align}
v^i(\bs{x}) = \int_{B}\phi(\bs{x},\bs{y})g(\bs{y})d\bs{y}, \quad \bs{x}\in\R^d.
\end{align}
It is a solution to the Helmholtz equation in $D$ with boundary data $v^i|_{\partial D}=Vg$. Suppose that $Vg=0$. This implies that $v^i=0$ in $D$ since $k^2$ is not a Dirichlet eigenvalue of $-\Delta$ in~$D$. We conclude that $g=0$ with the same arguments as those used to prove \cref{thm:near-field-op}.

We now show that $V$ has dense range by showing that the dual operator $V^*$,
\begin{align}
& V^*:H^{-1/2}(\partial D)\to W(B), \\
& (V^*\varphi)(\bs{x}) = \int_{\partial D}\overline{\phi(\bs{x},\bs{y})}\varphi(\bs{y})d\bs{y}, \quad \bs{x}\in B, \nonumber
\end{align}
is injective. Let $\varphi\in H^{-1/2}(\partial D)$ and consider the conjugate single-layer potential,
\begin{align}
v^i(\bs{y}) = \int_{\partial D}\overline{\phi(\bs{x},\bs{y})}\varphi(\bs{x})d\bs{x}, \quad \bs{y}\in\R^d.
\end{align}
It is a solution to the Helmholtz equation in both $\R^d\setminus\overline{D}$ and $D$, which satisfies the absorption condition,\footnote{The absorption condition is the conjugate condition of the Sommerfeld radiation condition \cref{eq:Sommerfeld}, and reads $\lim_{r\to\infty} r^{\frac{d-1}{2}}\left(\frac{\partial w}{\partial r} + ikw\right) = 0$ for $r = \vert\bs{x}\vert$ (uniformly in $\bs{x}/\vert\bs{x}\vert$).} and with near-field pattern $v^i|_{B}=V^*\varphi$. Suppose that $V^*\varphi=0$. By the unique continuation principle, one has $v^i=0$ in $\R^d\setminus\overline{D}$, and the regularity of $v^i\in H^1_{\mrm{loc}}(\R^d\setminus D)$ yields $v^i=0$ on $\partial D$, and hence $v^i=0$ in all of $D$ (since $k^2$ is not a Dirichlet eigenvalue of $-\Delta$ in $D$). Therefore, we have that $v^i=0$ in $\R^d$, and $\varphi=0$ follows from the jump relations of the conormal derivative of the single-layer potential \cite[Eqn.~(3.2)]{colton2019}.
\end{proof}

Again, it follows immediately that the conjugate operator $\overline{V}$,
\begin{align}
& \overline{V}:W(B)\to H^{1/2}(\partial D), \\
& (\overline{V}g)(\bs{x}) = \int_{B}\overline{\phi(\bs{x},\bs{y})}g(\bs{y})d\bs{y}, \quad \bs{x}\in\partial D, \nonumber
\end{align}
shares the same properties as $V$.

We now introduce the operator $A$, which maps the boundary values on $\partial D$ of radiating solutions onto the near-field measurements in $B$; see \cref{fig:operators}.

\begin{theorem}[Boundary-values-to-near-field operator]\label{thm:bnd-near-field-op}
Let $A:H^{1/2}(\partial D)\to W(B)$ be the operator that maps the boundary values $f=w|_{\partial D}$ of radiating solutions $w\in H^1_{\mrm{loc}}(\R^d\setminus D)$ to the Helmholtz equation onto the near-field pattern $Af=w|_{B}$. It is bounded, injective, and has dense range. 
\end{theorem}

\begin{proof}
It is bounded because the exterior Helmholtz Dirichlet problem is well-posed. To prove that it is injective, let $f\in H^{1/2}(\partial D)$ and suppose that $Af=0$. This implies that the radiating solution $w$ to the Helmholtz equation with near-field pattern $w|_{B}=Af$ vanishes in $\R^d\setminus\overline{D}$ by the unique continuation principle. The regularity of $w\in H^1_{\mrm{loc}}(\R^d\setminus D)$ leads to $w|_{\partial D}=f=0.$

To show that $A$ has dense range, we write it as an integral operator using Green's formula \cite[Eqn.~(2.9)]{colton2019}, Green's second theorem \cite[Eqn.~(2.3)]{colton2019}, and the radiation condition, 
\begin{align}
(Af)(\bs{x}) = \int_{\partial D}\frac{\partial u(\bs{x},\bs{y})}{\partial n(\bs{y})}f(\bs{y})dS(\bs{y}), \quad \bs{x}\in B,
\end{align}
where $u(\cdot,\bs{y})$ denotes the total field associated with the Helmholtz scattering problem \cref{eq:ext-Dirichlet} for the incident wave $\phi(\cdot,\bs{y})$. Let $g\in W(B)$ and suppose that
\begin{align}
\int_{B}(Af)(\bs{x})g(\bs{x})d\bs{x} = 0, \quad \forall f\in H^{1/2}(\partial D).
\end{align}
Interchanging the orders of integration gives
\begin{align}\label{eq:density-A}
\int_{\partial D}\int_{B}\frac{\partial u(\bs{x},\bs{y})}{\partial n(\bs{y})}g(\bs{x})d\bs{x}f(\bs{y})dS(\bs{y}) = 0, \quad \forall f\in H^{1/2}(\partial D).
\end{align}
Consider the volume potential
\begin{align}
v^i(\bs{y}) = \int_{B} \phi(\bs{x},\bs{y})g(\bs{x})d\bs{x}, \quad \bs{y}\in\R^d,
\end{align}
whose corresponding scattered and total fields are
\begin{align}
& v^s(\bs{y}) = \int_{B} u^s(\bs{x},\bs{y})g(\bs{x})d\bs{x}, \quad \bs{y}\in\R^d\setminus\overline{D}, \\
& v(\bs{y}) = \int_{B} u(\bs{x},\bs{y})g(\bs{x})d\bs{x}, \quad \bs{y}\in\R^d\setminus\overline{D}.
\end{align}
We rewrite the density relation \cref{eq:density-A} as
\begin{align}
\int_{\partial D}\frac{\partial v(\bs{x},\bs{y})}{\partial n(\bs{y})}f(\bs{y})dS(\bs{y}) = 0, \quad \forall f\in H^{1/2}(\partial D).
\end{align}
This implies that $\partial v/\partial n=0$ on $\partial D$, and since $v=0$ on $\partial D$, one has $v=0$ in $\R^d\setminus\{\overline{D}\cup\overline{B}\}$ via Holmgren's theorem. Since $v$ is as smooth as $v^i\in H^2_{\mrm{loc}}(\R^d)$, one has $v=\partial v/\partial n=0$ on the boundary $\partial B$. Moreover, because $v^s$ satisfies the Helmholtz equation in $B$, we have
\begin{align}
\Delta v + k^2 v = \Delta v^i + k^2 v^i = -g \quad \text{in $B$}.
\end{align}
We conclude the proof by taking the $L^2(B)$-scalar product with $g\in W(B)$.
\end{proof}

The operator $\overline{A}:H^{1/2}(\partial D)\to W(B)$, which maps the boundary values $g=w|_{\partial D}$ of absorbing solutions $w\in H^1_{\mrm{loc}}(\R^d\setminus D)$ of the Helmholtz equation onto the near-field pattern $\overline{A}g=w|_{B}$, shares the same properties as $A$.

\begin{theorem}[Factorization]\label{thm:factorization}
The near-field operator $N$ may be factored as $N=-AV$. Similarly, the operators $\overline{N}$, $\overline{V}$, and $\overline{A}$ are related via the factorization $\overline{N}=-\overline{A}\,\overline{V}$.
\end{theorem}

\begin{proof}
Since $v^s|_B=Ng$ represents the near-field pattern of the scattered field corresponding to the incident wave $v^i|_{\partial D}=Vg$, we clearly have that $Ng=A(-Vg)$ using $v^s|_{\partial D}=-v^i|_{\partial D}$. The same proof holds for the conjugate operators.
\end{proof}

Another key ingredient in the analysis of the LSM is the characterization of the range of the operator $A$. We show in the following lemma that the near-field measurements of a point source belong to the range of $A$ if and only if that point source is located inside $D$. Once again, this lemma mirrors that of far-field measurements \cite[Lem.~5.34]{colton2019}.

\begin{lemma}[Range]\label{lem:range}
$\phi|_{B}(\cdot,\bs{z})\in\mrm{range}(A)$ if and only if $\bs{z}\in D$.
\end{lemma}

\begin{proof}
If $\bs{z}\in D$, then $\phi|_{\partial D}(\cdot,\bs{z})\in H^{1/2}(\partial D)$ and $\phi|_{B}(\cdot,\bs{z})=A\phi|_{\partial D}(\cdot,\bs{z})$. If $\bs{z}\notin D$, assume that there exists $f\in H^{1/2}(\partial D)$ such that $Af=\phi|_{B}(\cdot,\bs{z})$. Therefore, by the unique continuation principle in $(\R^d\setminus\overline{D})\setminus\{\bs{z}\}$, the solution $u\in H^1_\mrm{loc}(\R^d\setminus D)$ to the exterior Dirichlet problem with $u|_{\partial D}=f$ must coincide with $\phi(\cdot,\bs{z})$ in $(\R^d\setminus\overline{D})\setminus\{\bs{z}\}$. If $\bs{z}\in\R^d\setminus\overline{D}$, this contradicts the regularity of $u$. If $\bs{z}\in\partial D$, from the boundary condition one has that $\phi|_{\partial D}(\cdot,\bs{z})=f\in H^{1/2}(\partial D)$, which is a contradiction to $\phi(\cdot,\bs{z})\notin H^1(D)$ when $\bs{z}\in\partial D$.
\end{proof}

\section{The imaginary near-field operator}\label{sec:imag-near-field}

In the previous section, we explored the properties of the near-field operator $N$ in a volume $B$, as well as those of its factors $V$ and $A$. The various results we proved, in particular \cref{lem:range}, will be essential to justify the use of the LSM for the imaginary near-field operator $I$, which we introduce next.

\begin{theorem}[Imaginary near-field operator]\label{thm:im-near-field-op} The imaginary near-field operator $I$,
\begin{align}
& I:W(B)\to W(B), \\
& (Ig)(\bs{x}) = \int_{B}\left[u^s(\bs{x},\bs{y}) - \overline{u^s(\bs{x},\bs{y})}\right]g(\bs{y})d\bs{y}, \quad \bs{x}\in B, \nonumber
\end{align}
is injective and has dense range if $k^2$ is not a Dirichlet eigenvalue of $-\Delta$ in $D$.
\end{theorem}

\begin{proof}
We first show that $I$ is injective. Let us define the scattered fields
\begin{align}\label{eq:v^s-w^s}
& v^s(\bs{x}) = \int_{B} u^s(\bs{x},\bs{y})g(\bs{y})d\bs{y}, \quad \bs{x}\in\R^d\setminus\overline{D}, \\
& w^s(\bs{x}) = \int_{B} \overline{u^s(\bs{x},\bs{y})}g(\bs{y})d\bs{y}, \quad \bs{x}\in\R^d\setminus\overline{D}.
\end{align}
Note that $v^s$ is a radiating solution to the Helmholtz equation in $\R^d\setminus\overline{D}$ for the incident wave $v^i|_{\partial D}=Vg$ with near-field pattern $v^s|_{B}=Ng$, while $w^s$ is an absorbing solution to the Helmholtz equation in $\R^d\setminus\overline{D}$ for the incident wave $w^i|_{\partial D}=\overline{V}g$ with conjugate near-field pattern $w^s|_{B}=\overline{N}g$. Suppose that $Ig=0$, which yields $v^s|_{B}=w^s|_{B}$. Since $v^s$ and $w^s$ are both solutions to the Helmholtz equation in $\R^d\setminus\overline{D}$, this implies $v^s=w^s$ in $\R^d\setminus\overline{D}$ by the unique continuation principle. Finally, $v^s$ is radiating while $w^s$ is absorbing, hence both $v^s=0$ and $w^s=0$ in $\R^d\setminus\overline{D}$, and in particular in $B$; $g=0$ follows from \cref{thm:near-field-op}.

To show that $I$ has dense range, we observe that the dual operator $I^*$,
\begin{align}
& I^*:W(B)\to W(B), \\
& (I^*g)(\bs{x}) = \int_{B}\left[\overline{u^s(\bs{x},\bs{y})} - u^s(\bs{x},\bs{y})\right]g(\bs{y})d\bs{y}, \quad \bs{x}\in B, \nonumber
\end{align}
verifies $I^*g=-Ig$ and, hence, is injective.
\end{proof}

A very last result is needed before we can prove our main result---it concerns the density of the image of the so-called product volume operator.

\begin{theorem}[Product volume operator]\label{thm:prod-vol-op}
The product volume operator $\mathcal{V}$,
\begin{align}
& \mathcal{V}:W(B)\to H^{1/2}(\partial D)\times H^{1/2}(\partial D), \\
& \mathcal{V}g = (Vg, \overline{V}g), \nonumber
\end{align}
is injective and has dense range if $k^2$ is not a Dirichlet eigenvalue of $-\Delta$ in $D$.
\end{theorem}

\begin{proof}
The operator $\mathcal{V}$ is trivially injective via \cref{thm:vol-op} for $V$ and $\overline{V}$. To show that $\mathcal{V}$ has dense range, we prove that the dual operator $\mathcal{V}^*$
\begin{align}
& \mathcal{V}^*:H^{-1/2}(\partial D)\times H^{-1/2}(\partial D)\to W(B) \\
& \mathcal{V}^*(\varphi,\psi) = V^*\varphi + \overline{V}^*\psi, \nonumber
\end{align}
is injective. Let $\varphi,\psi\in H^{-1/2}(\partial D)$ and consider the conjugate single-layer potential $v^i$ and the single-layer potential $w^i$ defined by
\begin{align}
& v^i(\bs{y}) = \int_{\partial D}\overline{\phi(\bs{x},\bs{y})}\varphi(\bs{x})dS(\bs{x}), \quad \bs{y}\in \R^d, \\
& w^i(\bs{y}) = \int_{\partial D}\phi(\bs{x},\bs{y})\psi(\bs{x})dS(\bs{x}), \quad \bs{y}\in \R^d.
\end{align}
Note that $v^i$ and $w^i$ are solutions to the Helmholtz equation in $\R^d\setminus\overline{D}$ and $D$, with near-field patterns $v^i|_{B}=V^*\varphi$ and $w^i|_{B}=V^*\psi$. Suppose that $\mathcal{V}^*(\varphi,\psi)=0$. This implies that $v^i+w^i=0$ in $\R^d\setminus\overline{D}$ (unique continuation principle). The regularity of $v^i+w^i\in H^1_{\mrm{loc}}(\R^d\setminus D)$ yields $v^i+w^i=0$ on $\partial D$, and hence $v^i+w^i=0$ in $D$ (since $k^2$ is not a Dirichlet eigenvalue for $-\Delta$ in $D$). Therefore, $v^i+w^i=0$ in $\R^d$. Since $v^i$ is absorbing while $w^i$ is radiating, we further have that $v^i=w^i=0$ in $\R^d$. Finally, the jump relations of the single-layer potential through $\partial D$ imply that $\varphi=\psi=0$.
\end{proof}

We have collected all the necessary ingredients to prove the main result of our paper, which justifies the use of the LSM with the imaginary near-field operator $I$. The following theorem echoes that of the far-field operator \cite[Thm.~5.35]{colton2019}.

\begin{theorem}[Linear sampling method for $I$]\label{thm:LSM-I}
Let us assume that $k^2$ is not a Dirichlet eigenvalue of $-\Delta$ in $D$. For any $\bs{z}\in D$ and $\epsilon>0$, there exists $g_{\bs{z}}^\epsilon\in W(B)$ such that
\begin{align}\label{eq:LSM-ineq}
\Vert Ig_{\bs{z}}^\epsilon - \phi|_{B}(\cdot,\bs{z})\Vert_{L^2(B)} < \epsilon.
\end{align}
Moreover, the volume potential $Vg_{\bs{z}}^\epsilon$ and the conjugate volume potential $\overline{V}g_{\bs{z}}^\epsilon$ remain bounded in the $H^{1/2}(\partial D)$-norm as $\epsilon\to0$.

For any $\bs{z}\notin D$, every $g_{\bs{z}}^\epsilon\in W(B)$ that satisfies \cref{eq:LSM-ineq} for any $\epsilon>0$ is such that
\begin{align}
\lim_{\epsilon\to0}\Vert Vg_{\bs{z}}^\epsilon\Vert_{H^{1/2}(\partial D)} = \infty \quad \text{or} \quad \lim_{\epsilon\to0}\Vert \overline{V}g_{\bs{z}}^\epsilon\Vert_{H^{1/2}(\partial D)} = \infty.
\end{align}
\end{theorem}

\begin{proof}
Let $\bs{z}\in D$, which implies that $\phi|_{\partial D}(\cdot,\bs{z})\in H^{1/2}(\partial D)$. Under our assumption on $k$, from \cref{thm:prod-vol-op}, given any $\epsilon>0$, there exists a sequence $g_{\bs{z}}^\epsilon\in W(B)$ such that
\begin{align}
& \Vert Vg_{\bs{z}}^\epsilon + \phi|_{\partial D}(\cdot,\bs{z})\Vert_{H^{1/2}(\partial D)} < \frac{\epsilon}{2\max(\Vert A\Vert,\Vert\overline{A}\Vert)}, \\
& \Vert \overline{V}g_{\bs{z}}^\epsilon\Vert_{H^{1/2}(\partial D)} < \frac{\epsilon}{2\max(\Vert A\Vert,\Vert\overline{A}\Vert)}.
\end{align}
The factorizations $N=-AV$ and $\overline{N}=-\overline{A}\,\overline{V}$, combined with the triangle inequality, yield
\begin{align}
\Vert Ig_{\bs{z}}^n - \phi|_{B}(\cdot,\bs{z})\Vert_{L^2(B)} \leq \epsilon.
\end{align}

Suppose now that $\bs{z}\notin D$ and consider $g_{\bs{z}}^\epsilon\in W(B)$ that satisfies \cref{eq:LSM-ineq} such that
\begin{align}
\limsup_{\epsilon\to0}\Vert Vg_{\bs{z}}^\epsilon\Vert_{H^{1/2}(\partial D)} < \infty \quad \text{and} \quad \limsup_{\epsilon\to0}\Vert \overline{V}g_{\bs{z}}^\epsilon\Vert_{H^{1/2}(\partial D)} < \infty.
\end{align}
Without loss of generality, we assume that $Vg_{\bs{z}}^\epsilon$ and $\overline{V}g_{\bs{z}}^\epsilon$ weakly converge to some functions $f$ and $g$ in $H^{1/2}(\partial D)$ as $\epsilon\to0$. Let $v^s\in H^1_{\mrm{loc}}(\R^d\setminus D)$ be the radiating solution to the Helmholtz equation with $v^s=f$ on $\partial D$, and let $v^s|_{B}$ be its near-field pattern. Similarly, let $w^s\in H^1_{\mrm{loc}}(\R^d\setminus D)$ be the absorbing solution to the Helmholtz equation with $w^s=g$ on $\partial D$, and let $w^s|_{B}$ be its near-field pattern. Since $Ng_{\bs{z}}^\epsilon$ and $\overline{N}g_{\bs{z}}^\epsilon$ are the near-field patterns for the incident fields $-Vg_{\bs{z}}^\epsilon$ and $-\overline{V}g_{\bs{z}}^\epsilon$, we conclude that $\phi|_{B}(\cdot,\bs{z})=-v^s|_{B}-w^s|_{B}$ at the limit $\epsilon\to0$. Using the unique continuation principle, and equating radiating and absorbing solutions, this yields $\phi(\cdot,\bs{z})=-v^s$ and $w^s=0$ in $\R^d\setminus\overline{D}$, and in particular in $B$. This means that $\phi|_{B}(\cdot,\bs{z})\in\mrm{range}(A)$, which contradicts \cref{lem:range}.
\end{proof}

\cref{thm:LSM-I} legitimizes the utilization of the imaginary near-field matrix \cref{eq:imag-near-field-mat} for the LSM in active imaging. As a byproduct, it also supports the LSM with the cross-correlation matrix \cref{eq:cross-cor-mat} in passive imaging. Note that it is possible to prove the same theorem with $\alpha\phi|_{B}(\cdot,\bs{z}) + \beta\overline{\phi|_{B}(\cdot,\bs{z})}$ for some $\alpha,\beta\in\C$ instead of $\phi|_{B}(\cdot,\bs{z})$ in \cref{eq:LSM-ineq}.

We would like to emphasize that \cref{thm:LSM-I} does not fully justify the numerical algorithms of \cref{sec:numerics}, in which the approximate solution is built using Tikhonov regularization and Morozov's discrepancy principle. This is a well-known shortcoming of the LSM, which motivated the introduction of factorization methods \cite{kirsch2007} and the GLSM \cite{audibert2014}. It would be indeed interesting to generalize the latter methods to random sources and cross-correlations. In the case of point scatterers, however, one can provide a rigorous justification of the LSM; see, e.g., \cite{haddar2012}. We provide such a proof for the imaginary near-field operator in \cref{sec:appendix}.

\section{Numerical experiments}\label{sec:numerics}

We mentioned in \cref{sec:HK-id} that the solution to the sound-soft inverse acoustic scattering problem with the LSM consists of two steps. (We will focus on the cross-correlation matrix, but what we will describe below applies to the near-field and imaginary near-field matrices, too.) First, the cross-correlation matrix $C$ is filled out in the data acquisition step (direct problem). Second, we probe the medium by solving the linear system $Cg_{\bs{z}}=\phi_{\bs{z}}$, for various $\bs{z}\in\R^d$, in the data processing step (inverse problem). The boundary $\partial D$ of the defect $D$ coincides with those points $\bs{z}$ for which $\Vert g_{\bs{z}}\Vert_2$ is large.

\paragraph{Solving the direct problem} To fill out the matrix $C$ in \cref{eq:cross-cor-mat}, we must solve the exterior Dirichlet problem \cref{eq:ext-Dirichlet} for $L$ random sources $\bs{z}_\ell$, and then evaluate the total field at measurement points $\bs{x}_j$. Our implementations in MATLAB employ \href{https://github.com/matthieuaussal/gypsilab}{gypsilab}, an open-source MATLAB toolbox for fast numerical computation with finite and boundary elements in 2D and 3D. We utilize the single-layer formulation of the exterior Dirichlet problem---weakly singular and near-singular integrals may be computed with the method described in \cite{montanelli2022}. (It is preferable, in general, to utilize a combined integral equation approach, with both the single- and double-layer potentials, which is coercive when the wavenumber is large \cite{spence2015}. For our experiments, computations with the single-layer potential are completely fine.)

\paragraph{Solving the inverse problem} Once the matrix $C$ has been filled out, we add some random noise to it to generate a matrix $C_\delta$ (this simulates noisy measurements). To solve $C_\delta g_{\bs{z}} = \phi_{\bs{z}}$, we compute the SVD of the matrix $C_\delta$, $C_\delta= U_\delta S_\delta V_\delta^*$, and apply Tikhonov regularization with parameter $\alpha>0$. We arrive at the following equation for each component $1\leq j\leq J$,
\begin{align}\label{eq:LSM}
(V_\delta^*g_{\bs{z}})_j = \frac{\sigma_j}{\alpha + \sigma_j^2}\left(U_\delta^*\phi_{\bs{z}}\right)_j,
\end{align}
where the $\sigma_j$'s are the singular values. To choose the regularization parameter $\alpha$, we use Morozov's discrepancy principle, which enforces
\begin{align}\label{eq:Morozov}
\Vert C_\delta g_{\bs{z}} - \phi_{\bs{z}}\Vert^2_2 = \delta^2\Vert g_{\bs{z}}\Vert^2_2, 
\end{align}
where $\delta=\Vert C_\delta - C\Vert_2$. Since $\Vert C_\delta g_{\bs{z}} - \phi_{\bs{z}}\Vert_2 = \Vert U_\delta^*\left(C_\delta g_{\bs{z}}  - \phi_{\bs{z}}\right)\Vert_2$ and $\Vert g_{\bs{z}}\Vert_2 = \Vert V_\delta^*g_{\bs{z}} \Vert_2$, by combining \cref{eq:LSM} with \cref{eq:Morozov}, we end up with the following equation to solve for $\alpha$,
\begin{align}\label{eq:alpha}
\sum_{j=1}^J\frac{\alpha^2 - \delta^2\sigma_j^2}{(\alpha + \sigma_j^2)^2}\vert(U_\delta^*\phi)_j\vert^2 = 0.
\end{align}
Once $\alpha$ has been computed, the norm of $g_{\bs{z}}$ is computed via
\begin{align}
\Vert g_{\bs{z}}\Vert_2 = \Vert V_\delta^*g_{\bs{z}} \Vert_2 = \Vert S_\alpha U_\delta^*\phi_{\bs{z}}\Vert_2, 
\end{align}
where $S_\alpha$ is the diagonal matrix with entries $\sigma_j/(\alpha + \sigma_j^2)$. 

\paragraph{Full-aperture measurements} We consider the scattering of points sources by an ellipse and a kite of size $\lambda/2$ centered at $-2\lambda-2\lambda i$ and $2\lambda + 2\lambda i$ for $k=2\pi$ (wavelength $\lambda =1$). The ellipse has axes $a=1.5$ and $b=1$, while the kite is that of \cite[sect.~3.6]{colton2019}. We compare the results obtained for the near-field matrix $N$ of \cref{eq:near-field-mat}, the imaginary near-field matrix $I$ of \cref{eq:imag-near-field-mat}, and the cross-correlation matrix $C$ of \cref{eq:cross-cor-mat}. For $N$ and $I$, we take $J=80$ equispaced co-located sources and receivers on the circle of radius $5\lambda$, 
\begin{align}
\bs{x}_j = 5\lambda e^{i\theta_j},  \quad \theta_j = \frac{2\pi}{J}(j - 1), \quad  1\leq j\leq J.
\label{eq:posrecnum}
\end{align}
For the matrix $C$, the $L=80$ random sources are located on the circle of radius $50\lambda$,
\begin{align}\label{eq:random-sources}
\bs{z}_\ell = 50\lambda e^{i\theta_\ell}, \quad \theta_\ell = \frac{2\pi}{L}(\ell -1+ \beta_\ell), \quad 1\leq \ell\leq L,
\end{align}
where $\beta_\ell$ is drawn from the uniform distribution on $[0,\beta]$ with $\beta=0.1$, and we measure at the $J=80$ equispaced points $\bs{x}_j$ defined in \cref{eq:posrecnum}. Finally, we add some white noise with amplitude $5\times 10^{-2}$ to each matrix, and probe the medium on a $100\times100$ uniform grid on $[-6\lambda,6\lambda]\times[-6\lambda,6\lambda]$. The results are shown in \cref{fig:elli} and \cref{fig:kite} for the ellipse and the kite. The defect is well identified by our novel sampling method, based on cross-correlations and random sources. This illustrates that the LSM can be utilized in passive imaging.

\begin{figure}
\centering
\def\scl{0.18}
\includegraphics[scale=\scl]{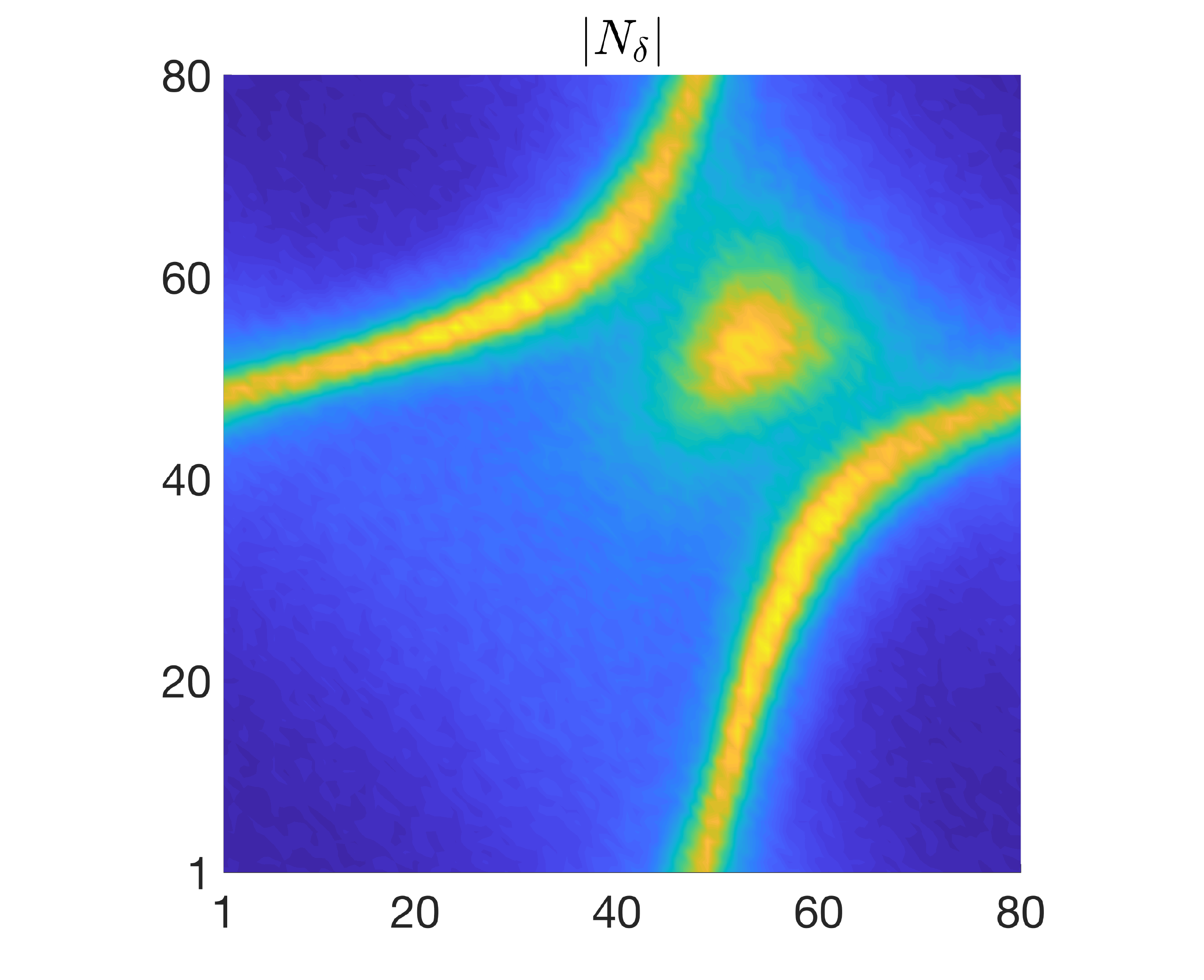} 
\includegraphics[scale=\scl]{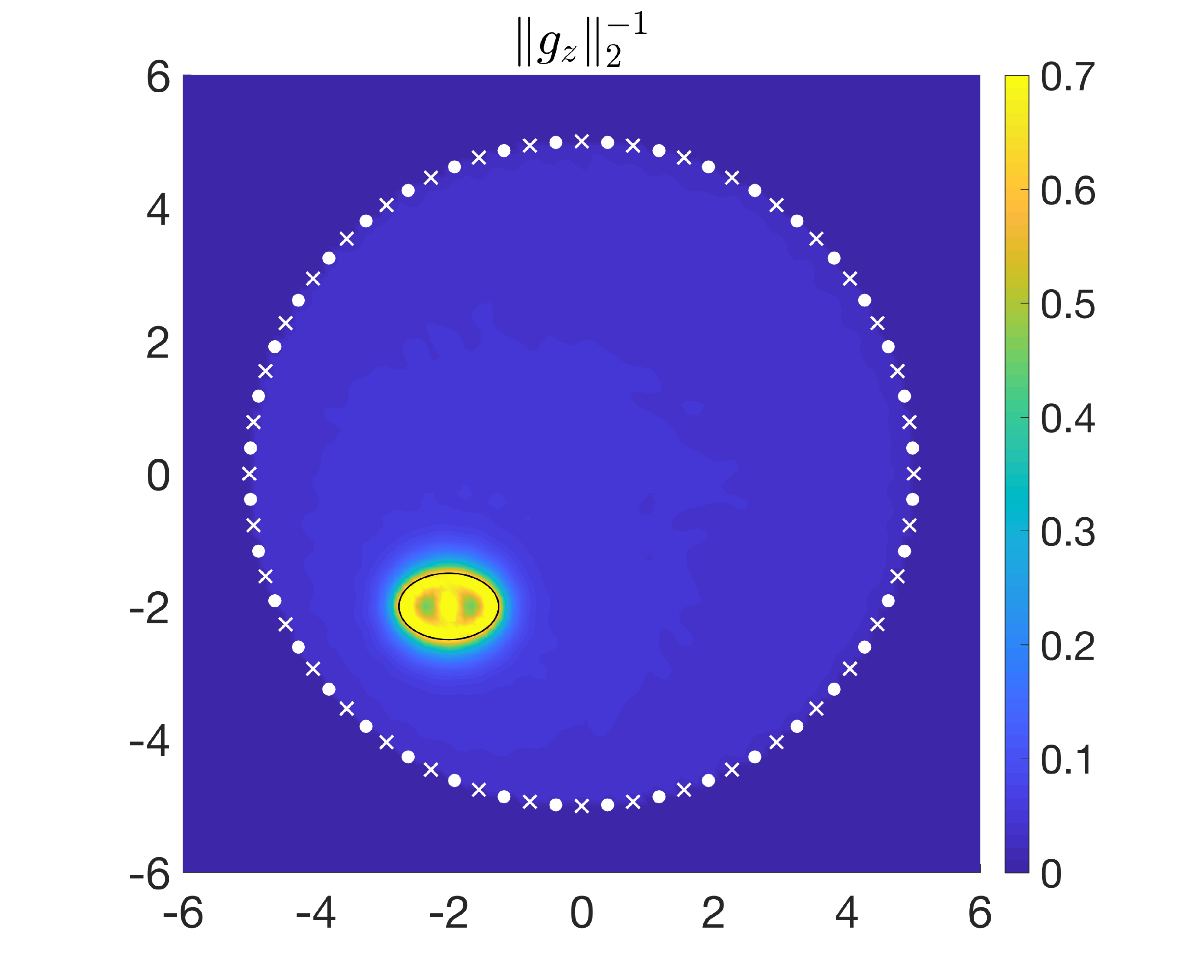} \\
\includegraphics[scale=\scl]{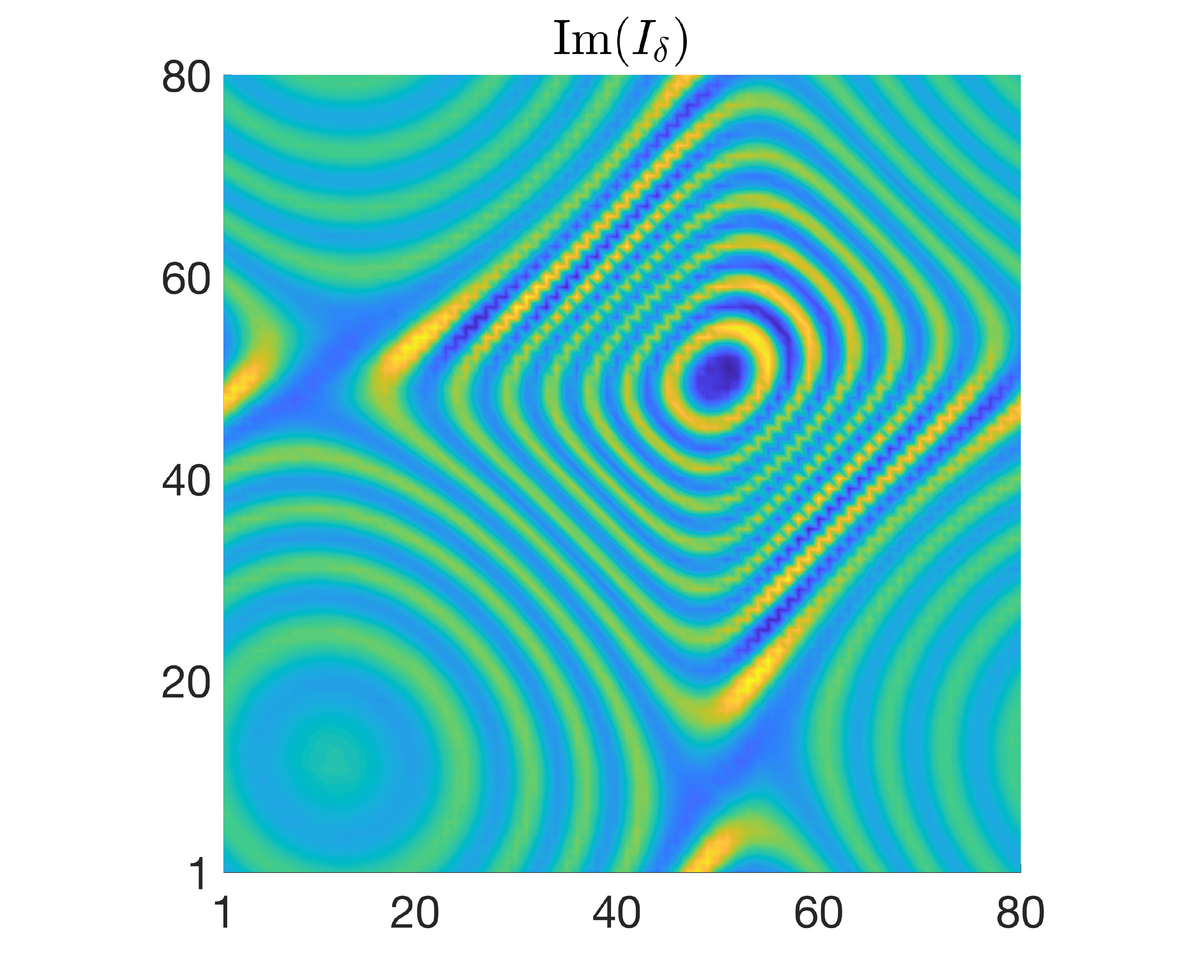} 
\includegraphics[scale=\scl]{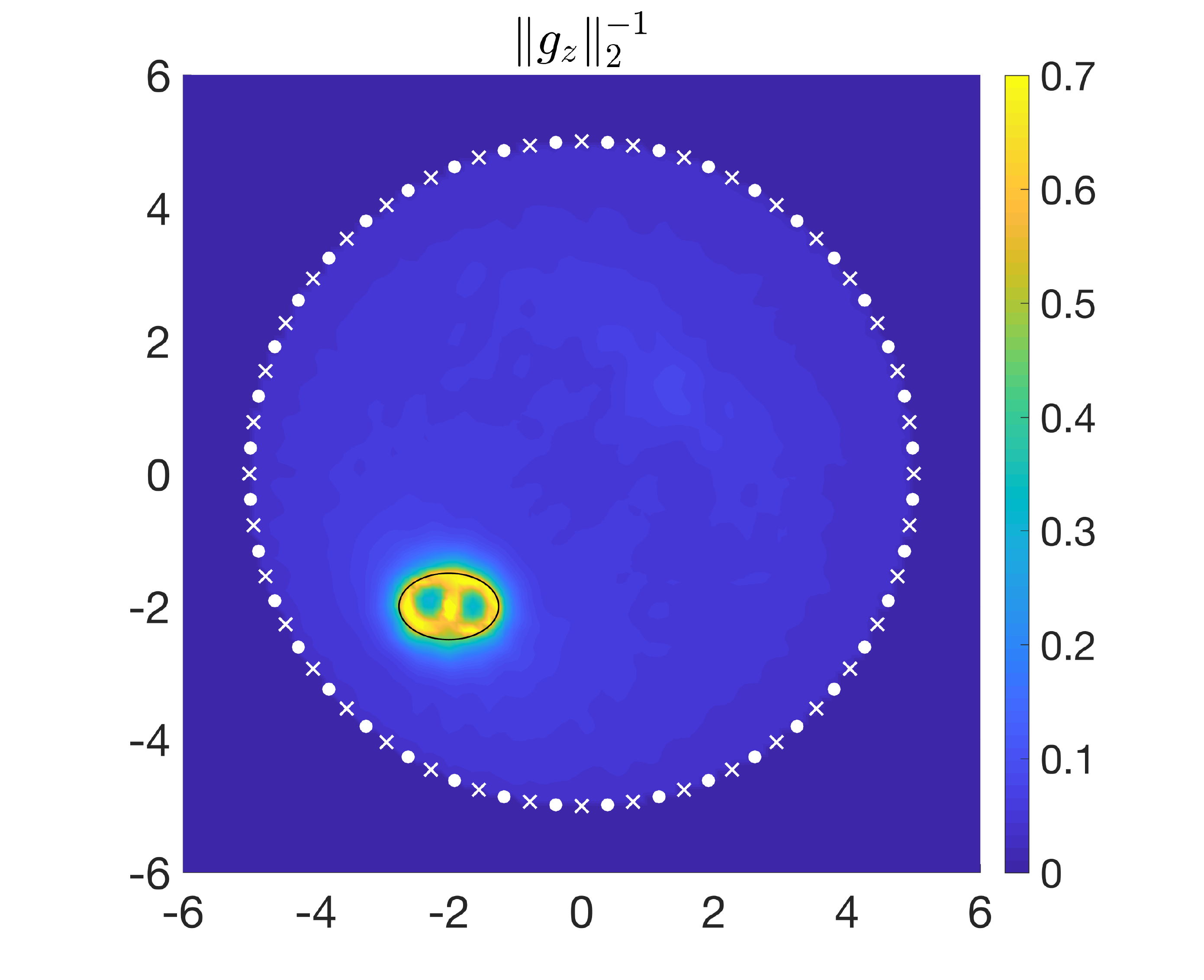}  \\
\includegraphics[scale=\scl]{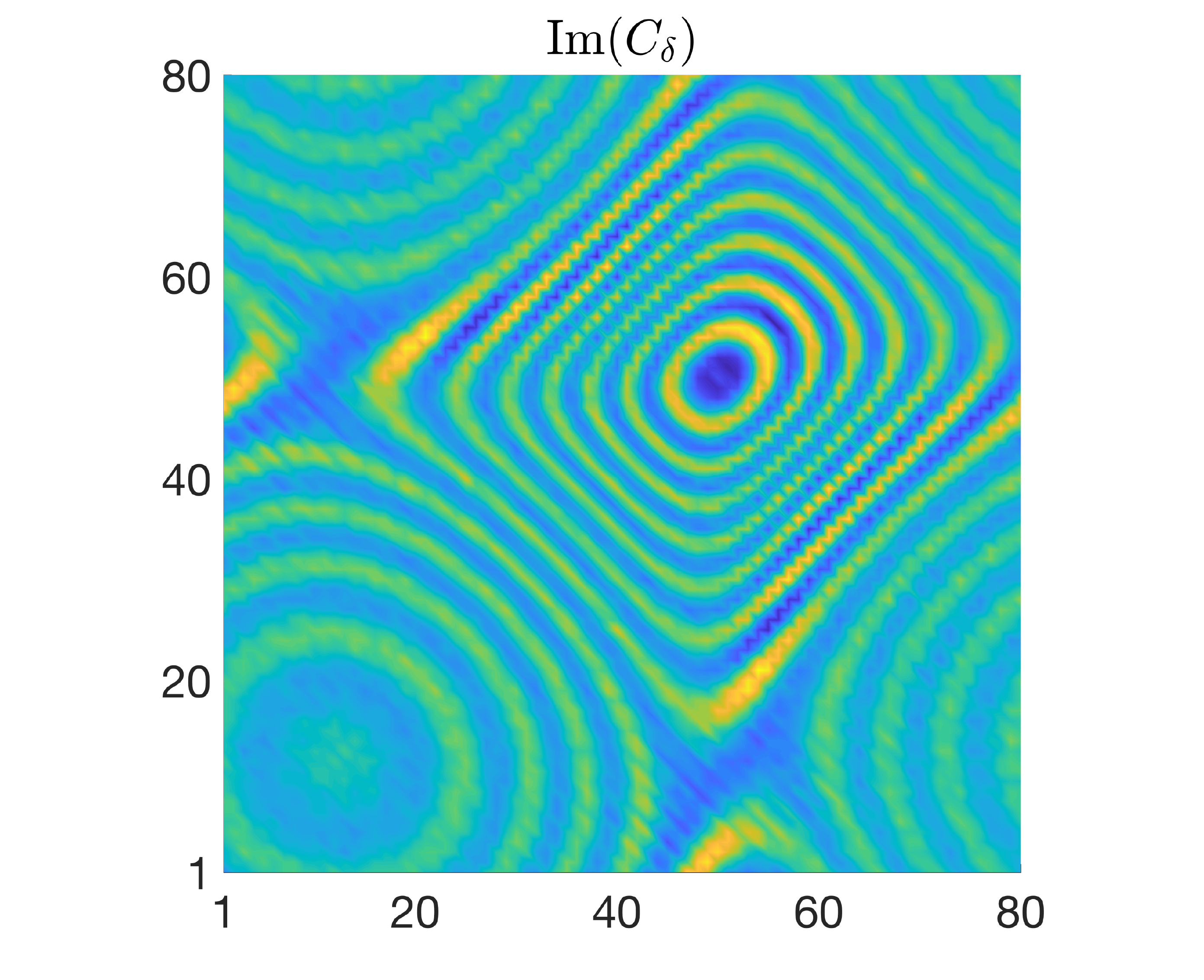} 
\includegraphics[scale=\scl]{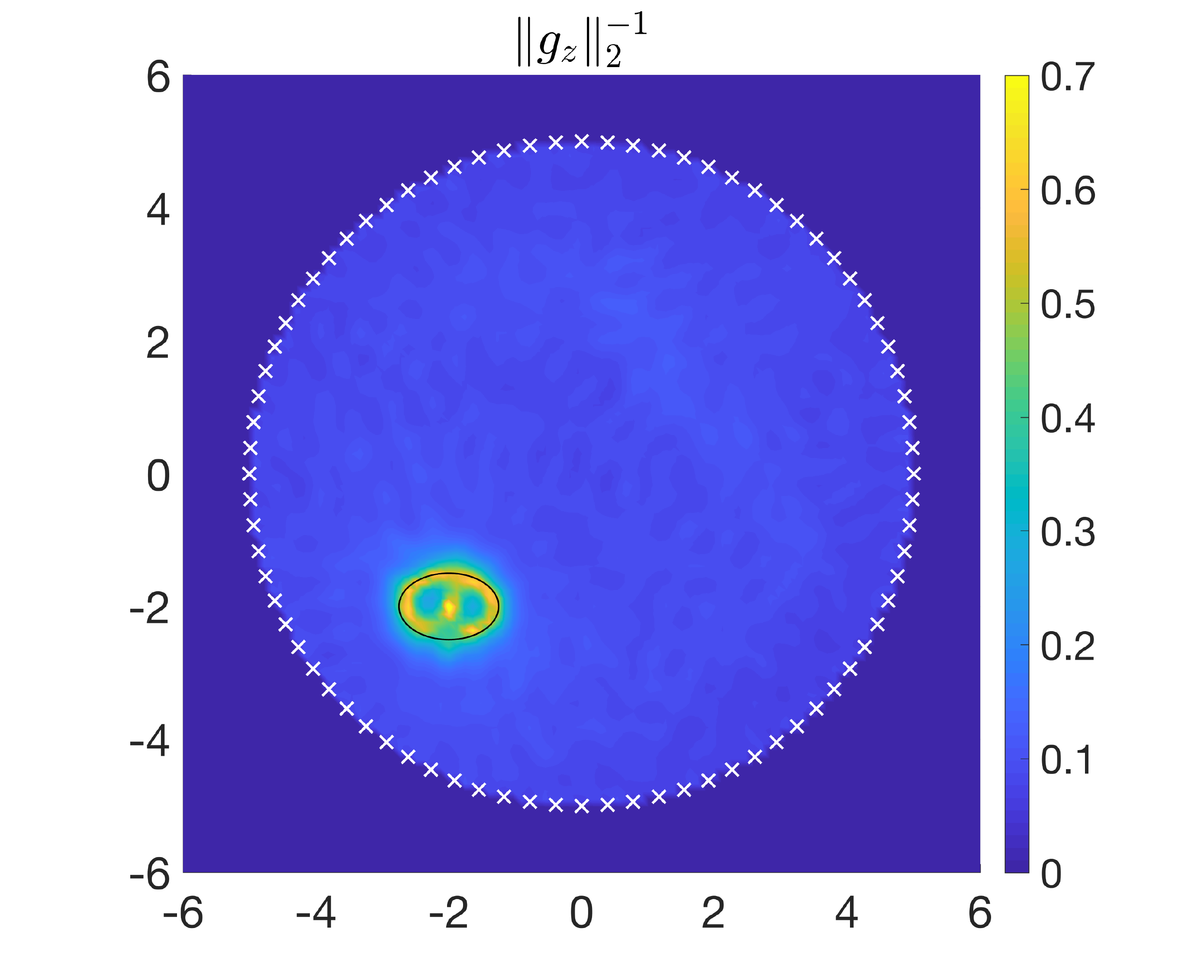}
\caption{The ellipse is well reconstructed by all three methods---the LSM with the standard near-field matrix \cref{eq:near-field-mat} (first row), with the imaginary near-field matrix \cref{eq:imag-near-field-mat} (second row), and with the cross-correlation matrix \cref{eq:cross-cor-mat} (third row). The first column displays the corresponding matrices and the second column the indicator function (values outside of the circle of radius $5\lambda$ were zeroed out). The sources are represented by dots, and the measurement points by crosses (for $N$ and $I$, we plotted every other source/measurement point; for $C$, the random sources are outside of the plot). Our novel sampling method, based on cross-correlations and random sources, shows similar results to the LSM with deterministic sources.}
\label{fig:elli}
\end{figure}

\begin{figure}
\centering
\def\scl{0.18}
\includegraphics[scale=\scl]{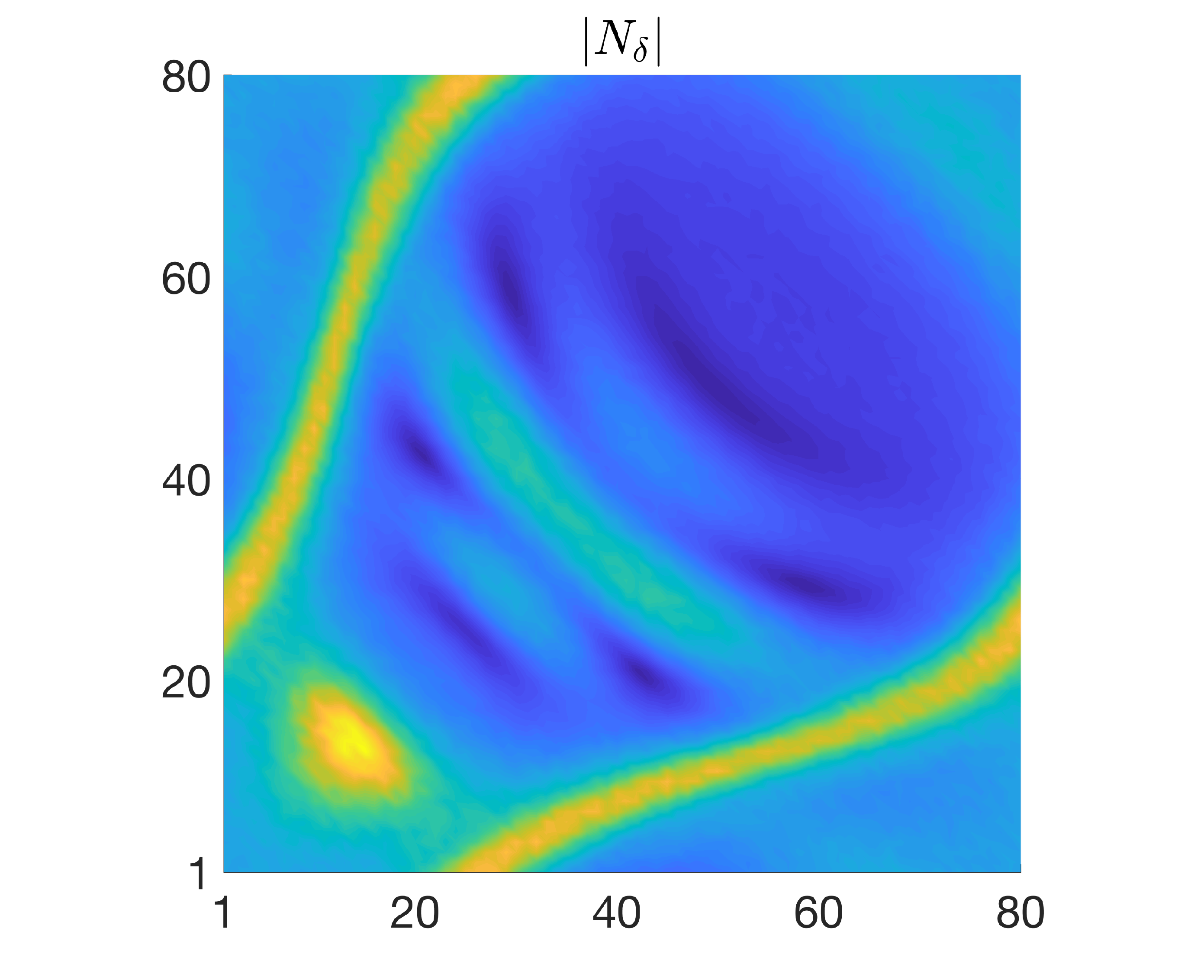} 
\includegraphics[scale=\scl]{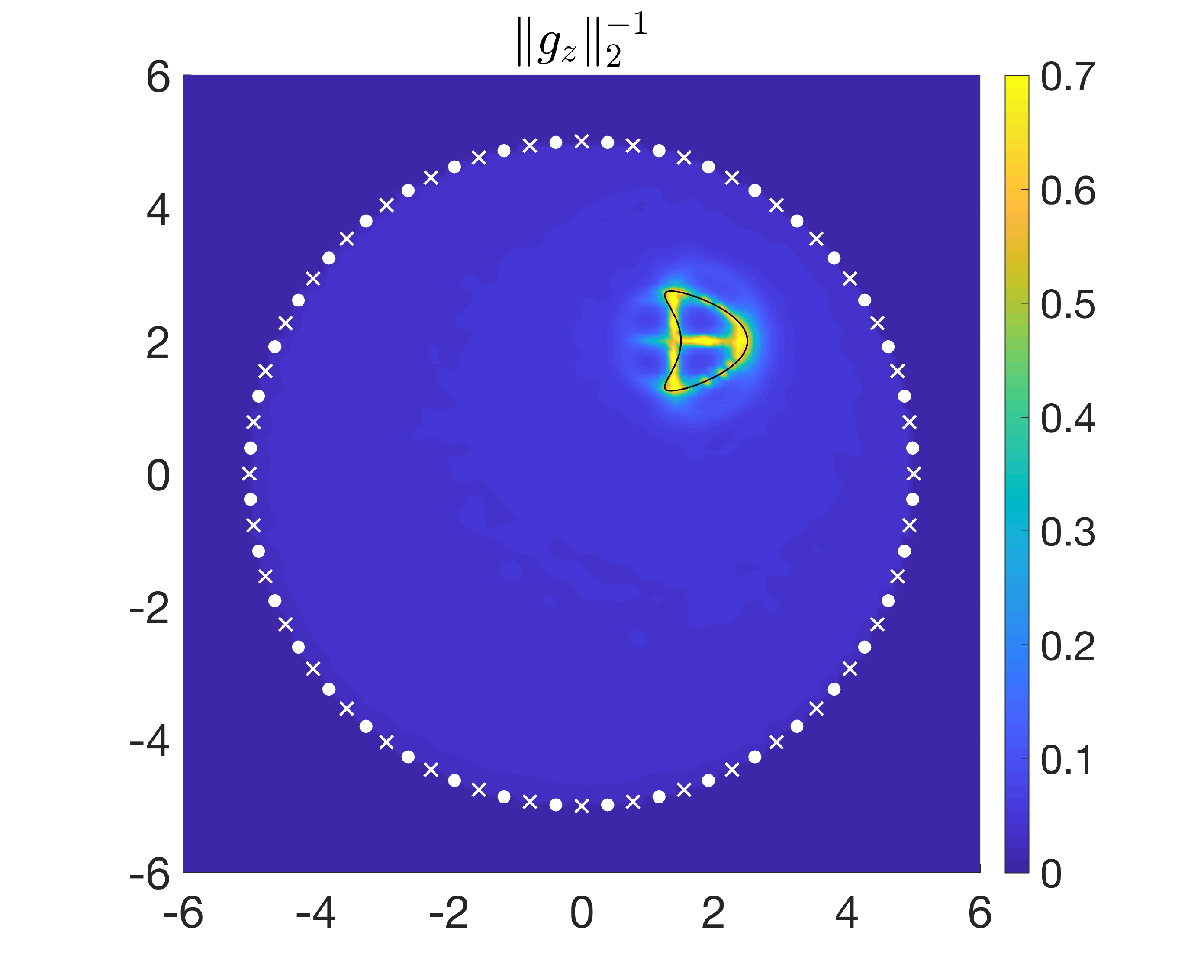} \\
\includegraphics[scale=\scl]{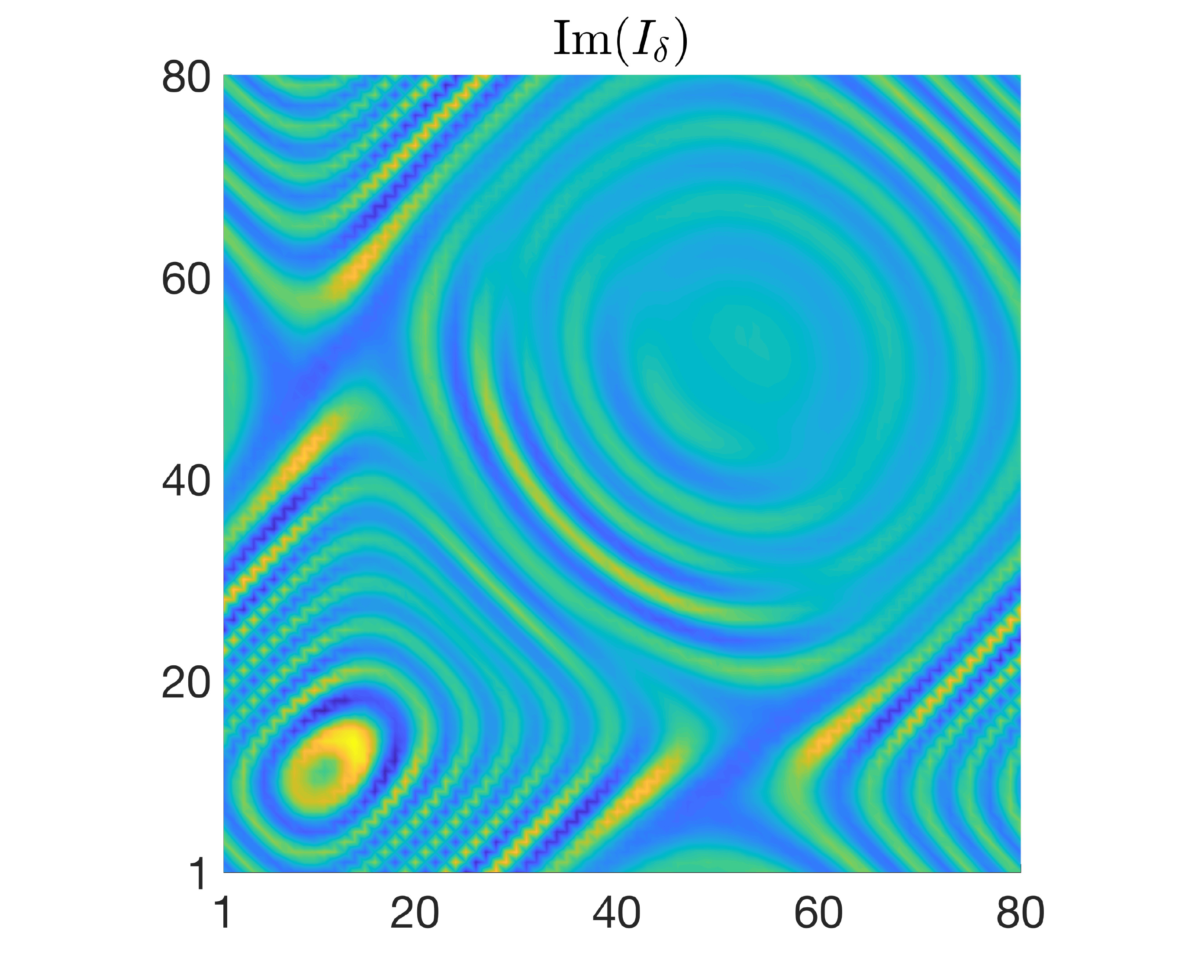} 
\includegraphics[scale=\scl]{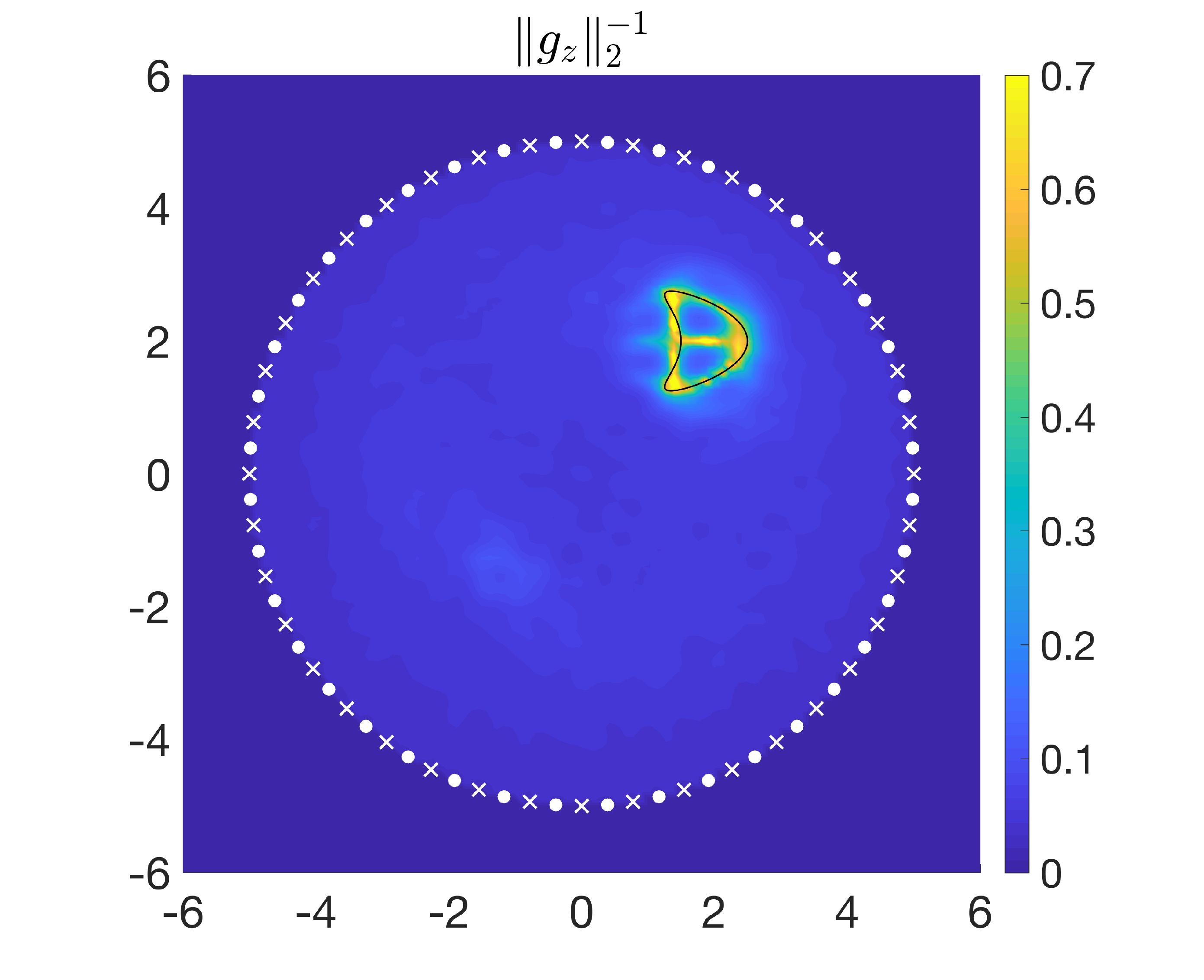}  \\
\includegraphics[scale=\scl]{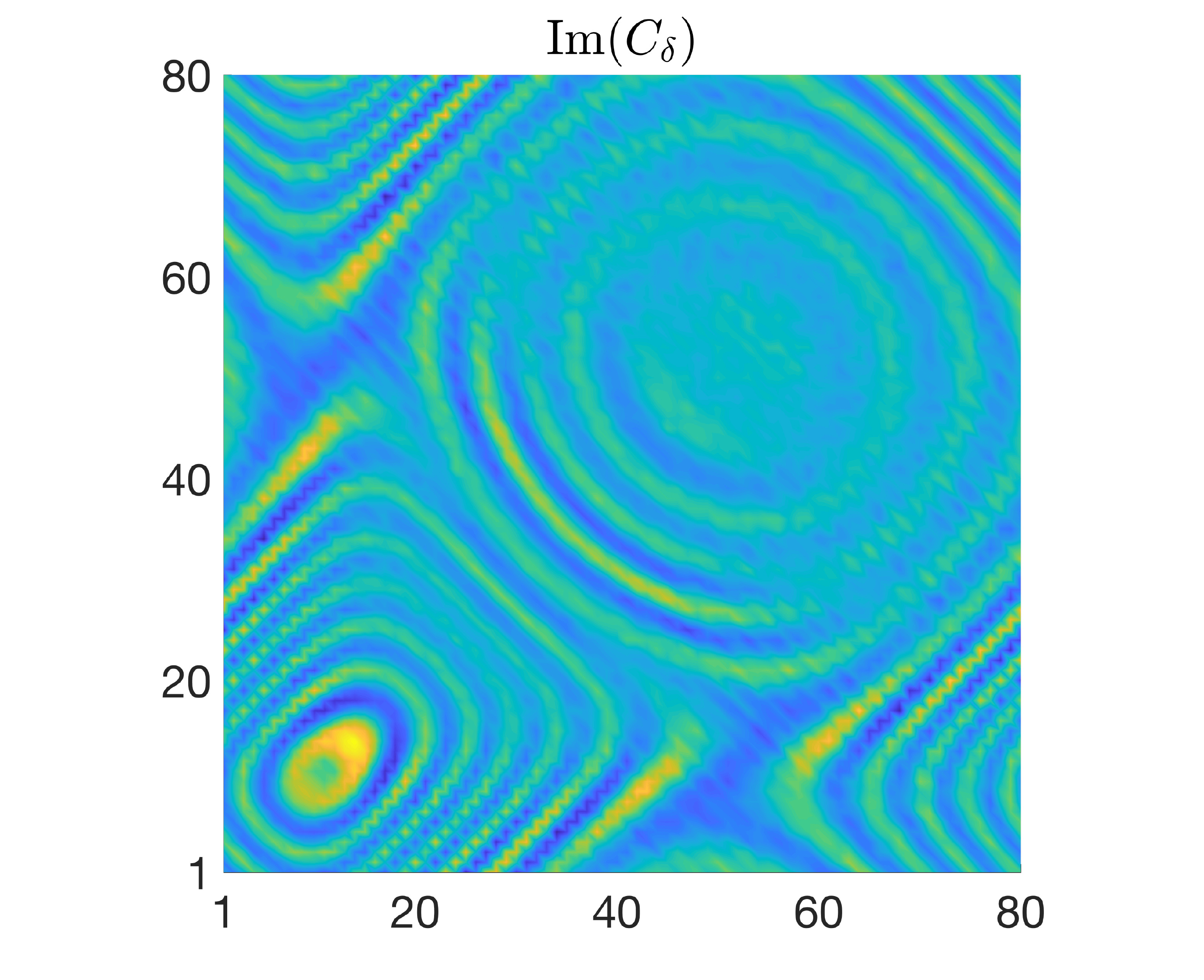} 
\includegraphics[scale=\scl]{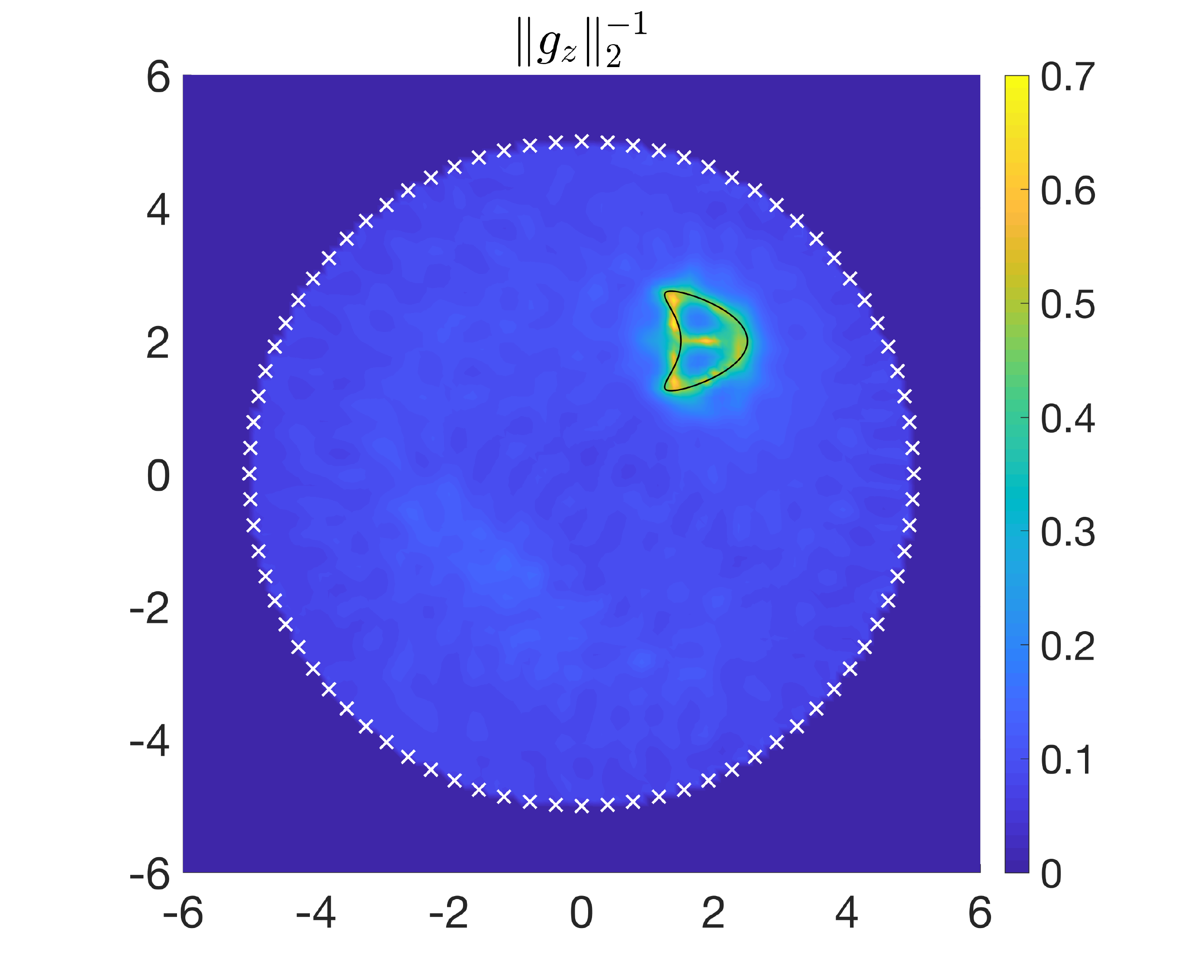}
\caption{In this experiment, too, the defect (a kite) is well identified by the LSM with the standard near-field matrix \cref{eq:near-field-mat} (first row), with the imaginary near-field matrix \cref{eq:imag-near-field-mat}, and with the cross-correlation matrix \cref{eq:cross-cor-mat} (third row). This demonstrates that the LSM can be utilized in passive imaging with random sources using the cross-correlation of the measurements.}
\label{fig:kite}
\end{figure}

\paragraph{Influence of the perturbation $\beta$} We investigate the influence of $\beta$ in \cref{eq:random-sources} in \cref{fig:kite-beta}. We perform the experiment of the previous paragraph for the cross-correlation matrix $C$ with the kite of \cref{fig:kite} for $\beta=0.3$, $\beta=0.6$, and $\beta=0.9$. We observe there that the quality of the reconstruction of the shape deteriorates when the value of $\beta$ increases. This is expected since $C$ is the trapezoidal rule approximation to $I$. Therefore, when the distribution of sources in \cref{eq:random-sources} deviates from the equispaced distribution, the accuracy in computing $C$ decreases---using a larger number of random sources improves the reconstruction process. (The trapezoidal rule is exponentially convergent for analytic functions and equispaced points \cite{trefethen2014}; for functions with $\sigma$ derivatives, it converges at the rate $\OO(N^\sigma)$ \cite[Thm.~4.3]{montanelli2015b}. For $\beta$-perturbed points, algebraic convergence for differentiable functions has only been proven for $\beta<1/2$, at a slower rate $\OO(N^{\sigma-4\beta})$ \cite[Thm.~1]{austin2017b}. For details about computations with trigonometric interpolants, we refer to \cite{austin2017a, montanelli2017phd, montanelli2020b, trefethen2019}.)

\begin{figure}
\centering
\def\scl{0.17}
\includegraphics[scale=\scl]{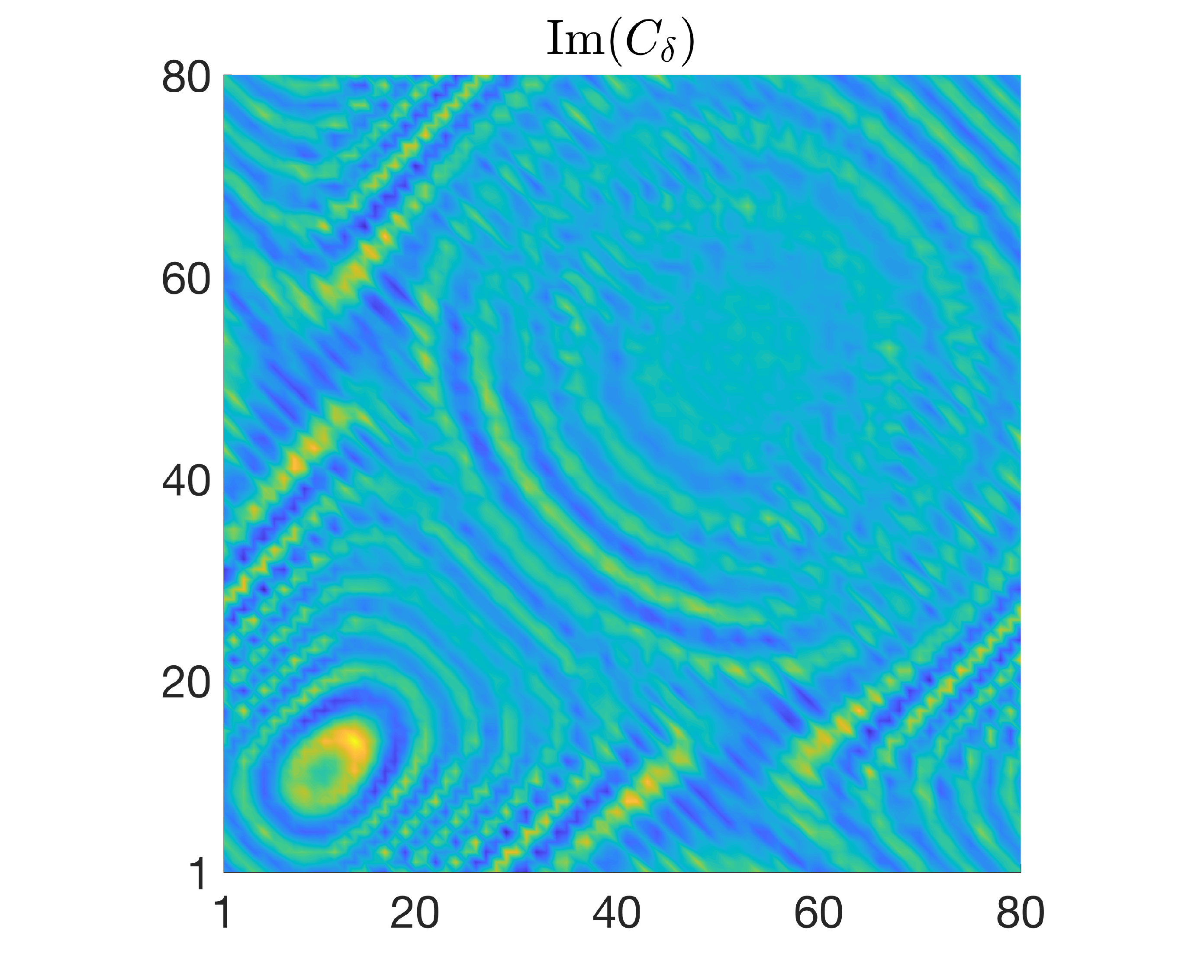} 
\includegraphics[scale=\scl]{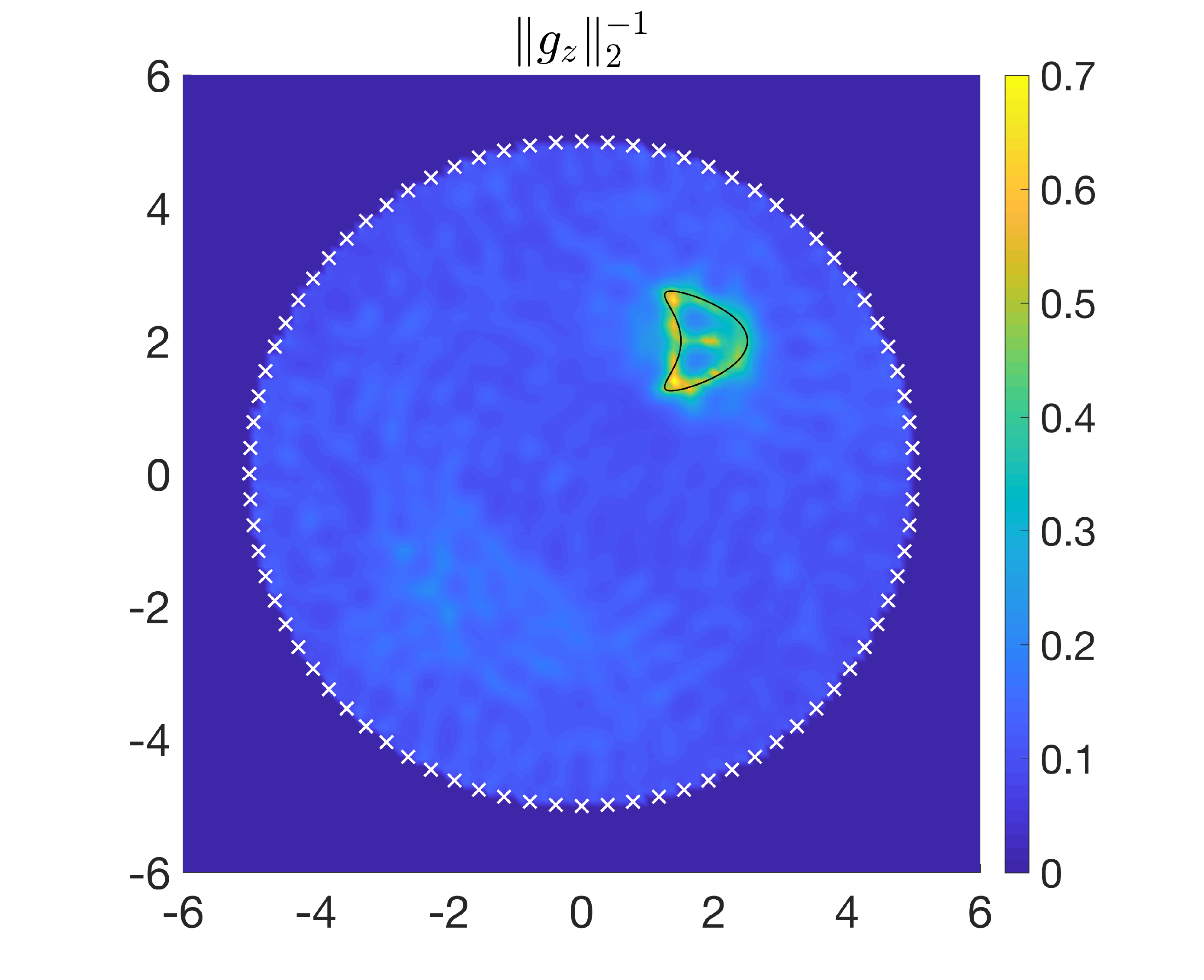}  
\includegraphics[scale=\scl]{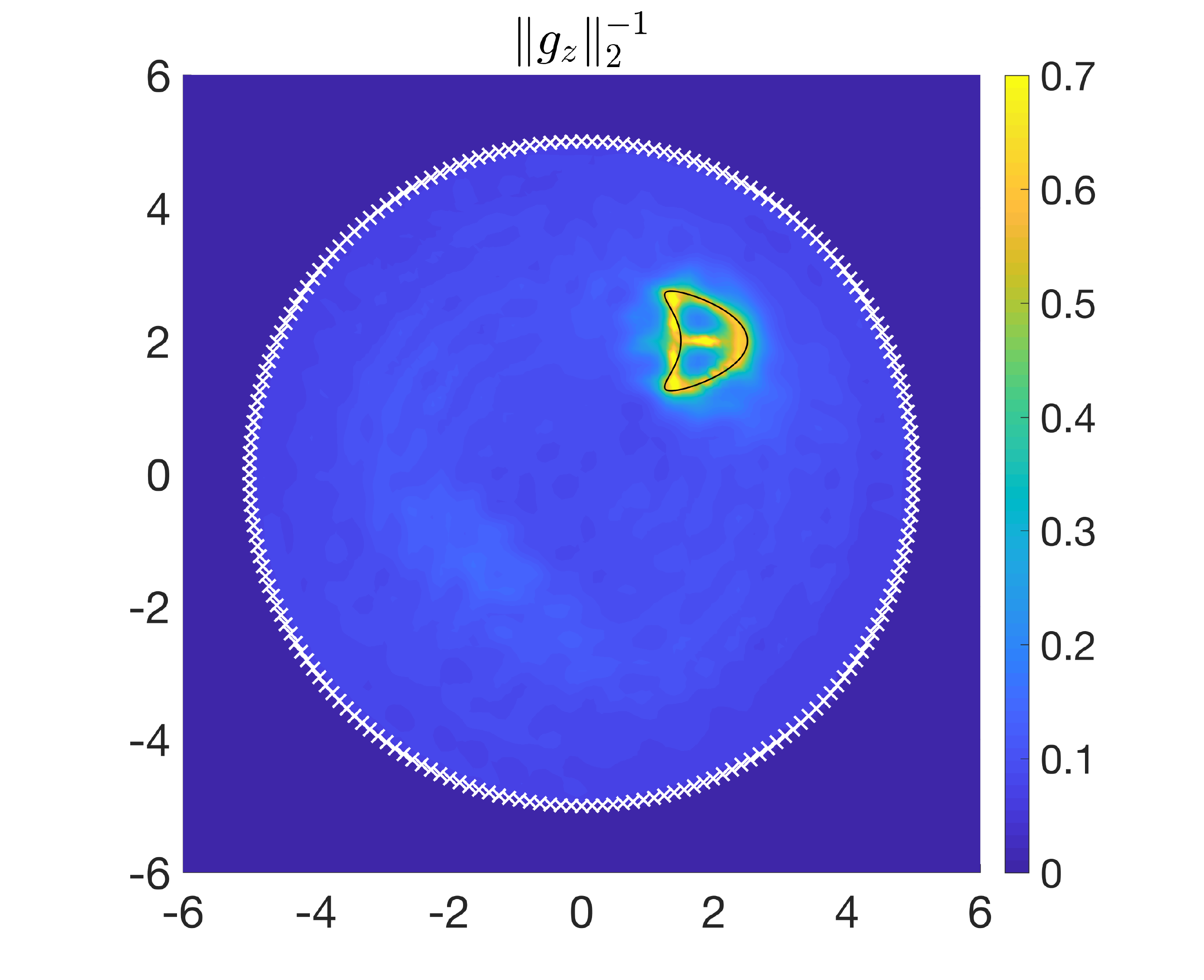} \\
\includegraphics[scale=\scl]{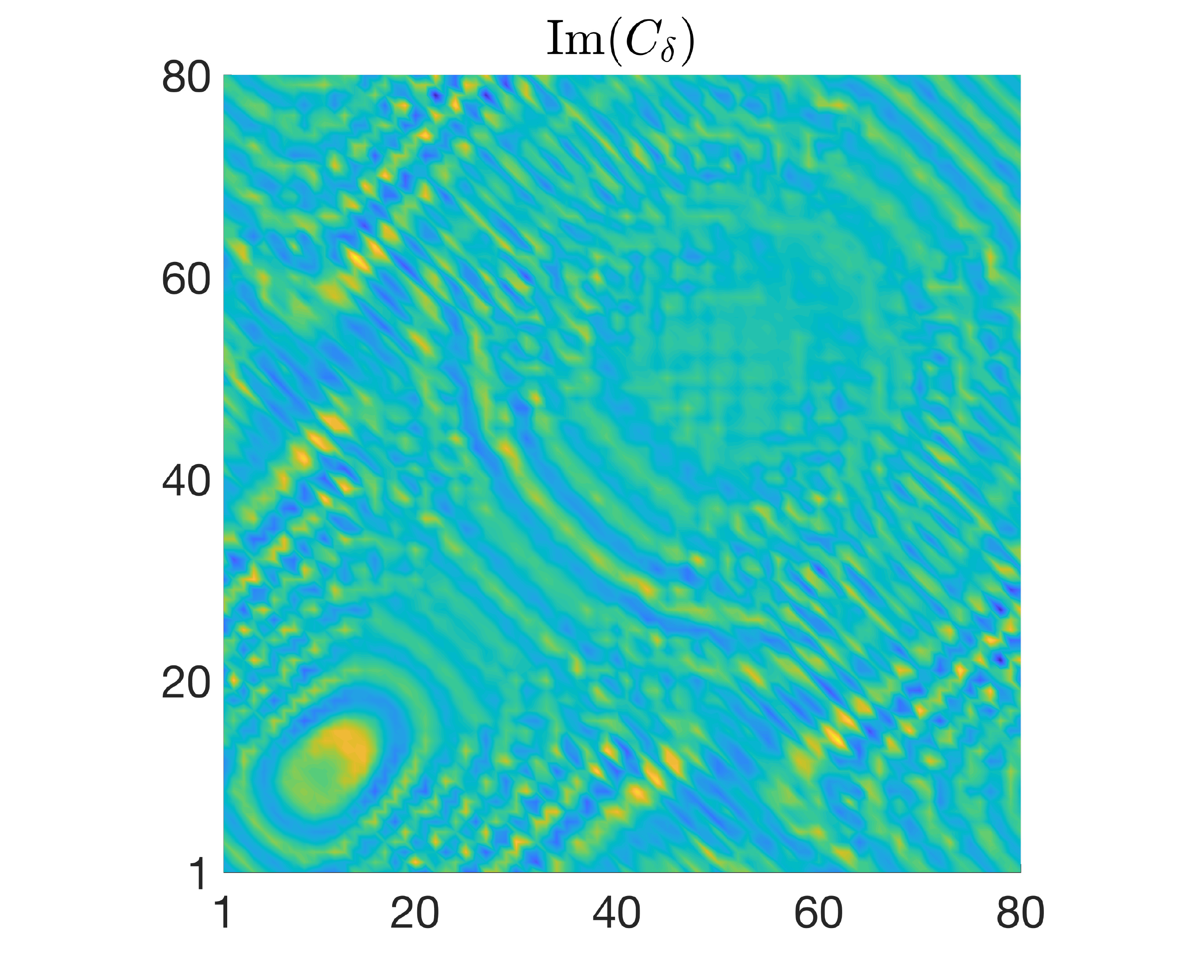} 
\includegraphics[scale=\scl]{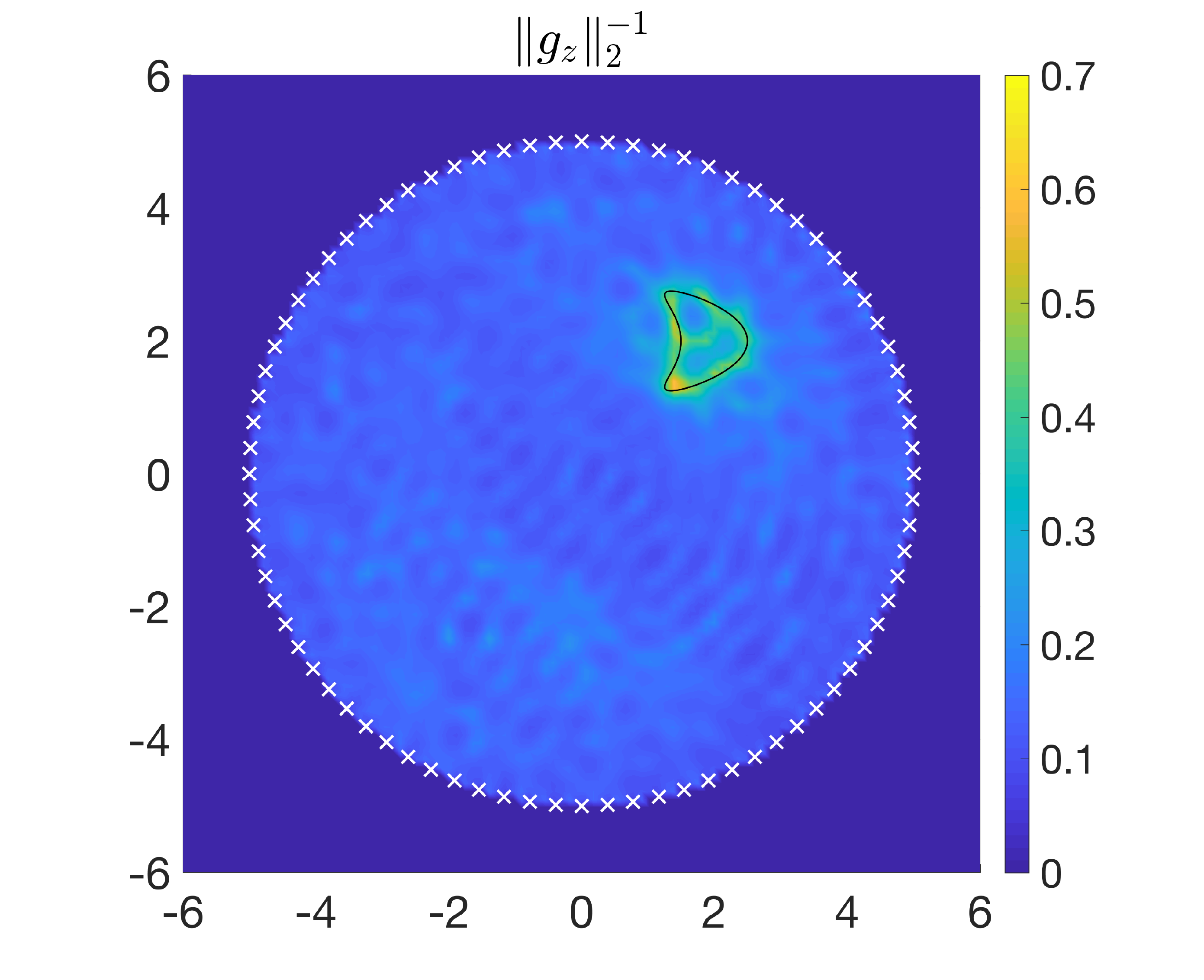}
\includegraphics[scale=\scl]{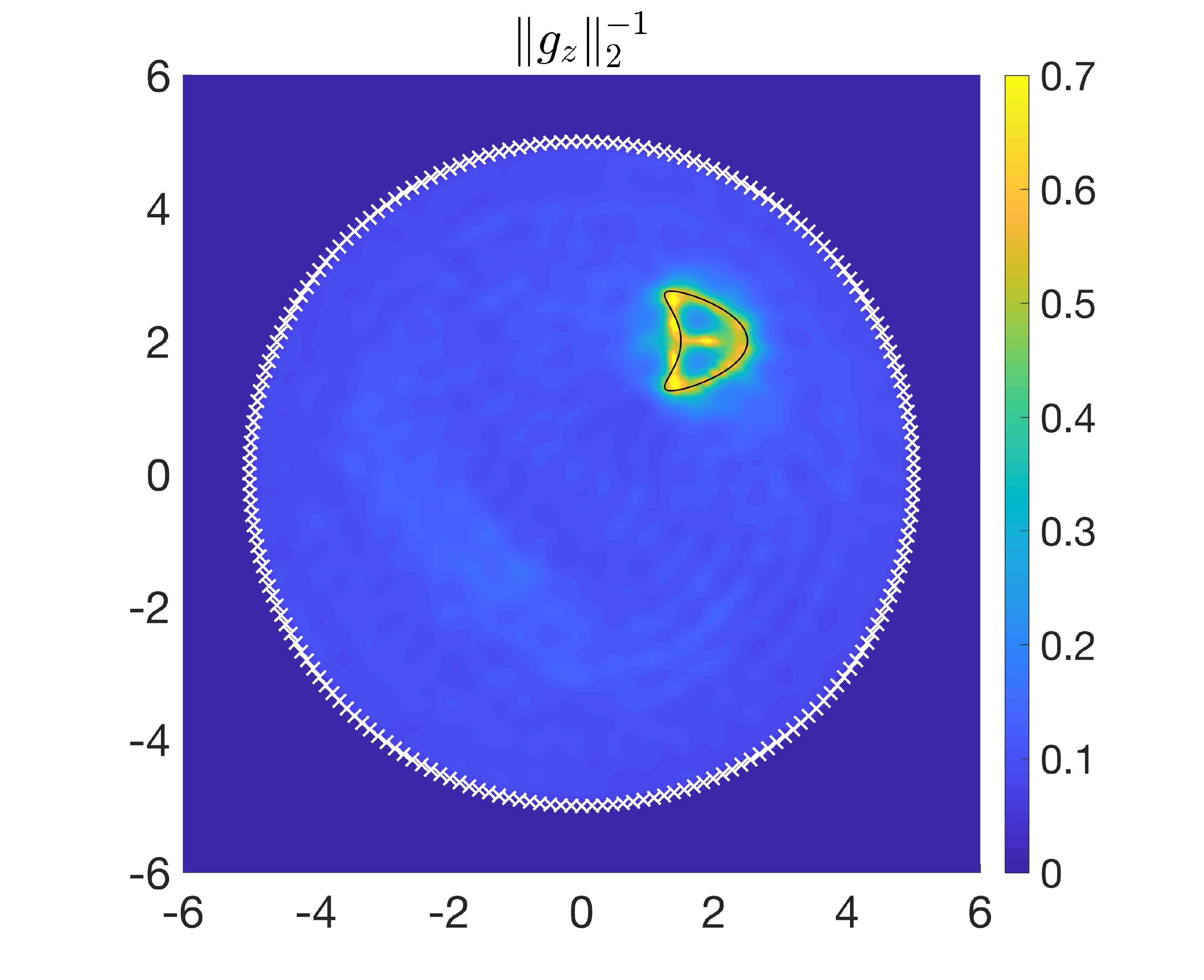} \\
\includegraphics[scale=\scl]{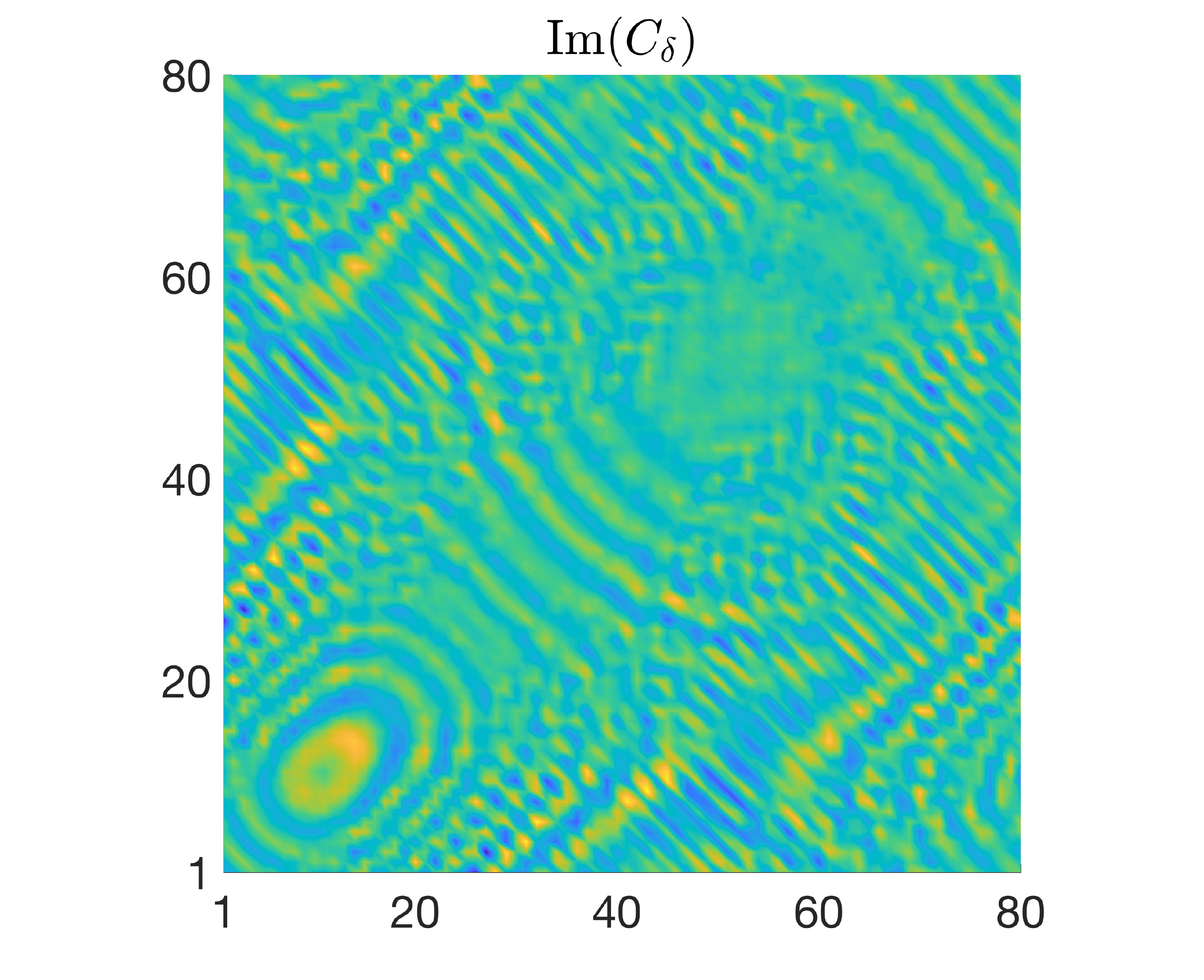} 
\includegraphics[scale=\scl]{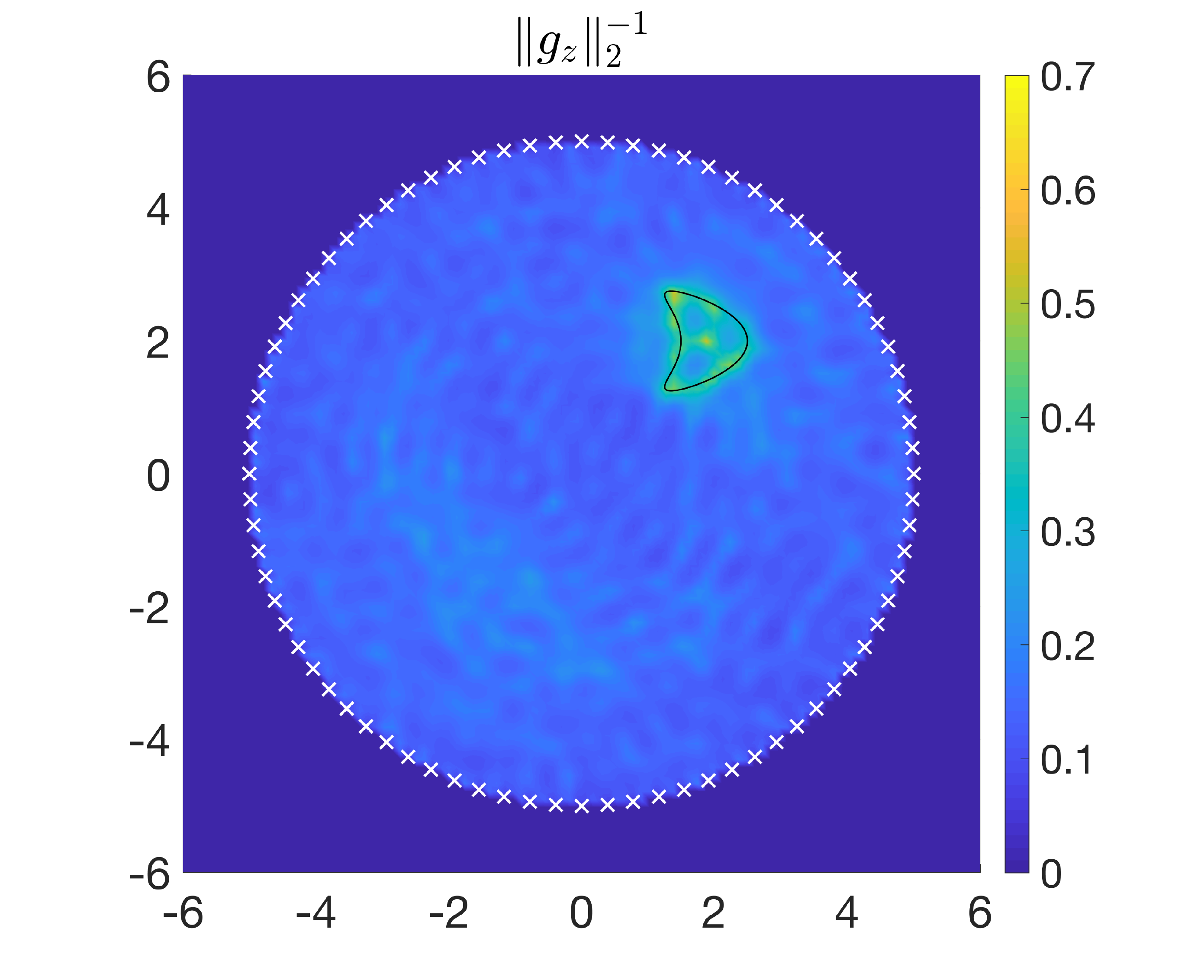}
\includegraphics[scale=\scl]{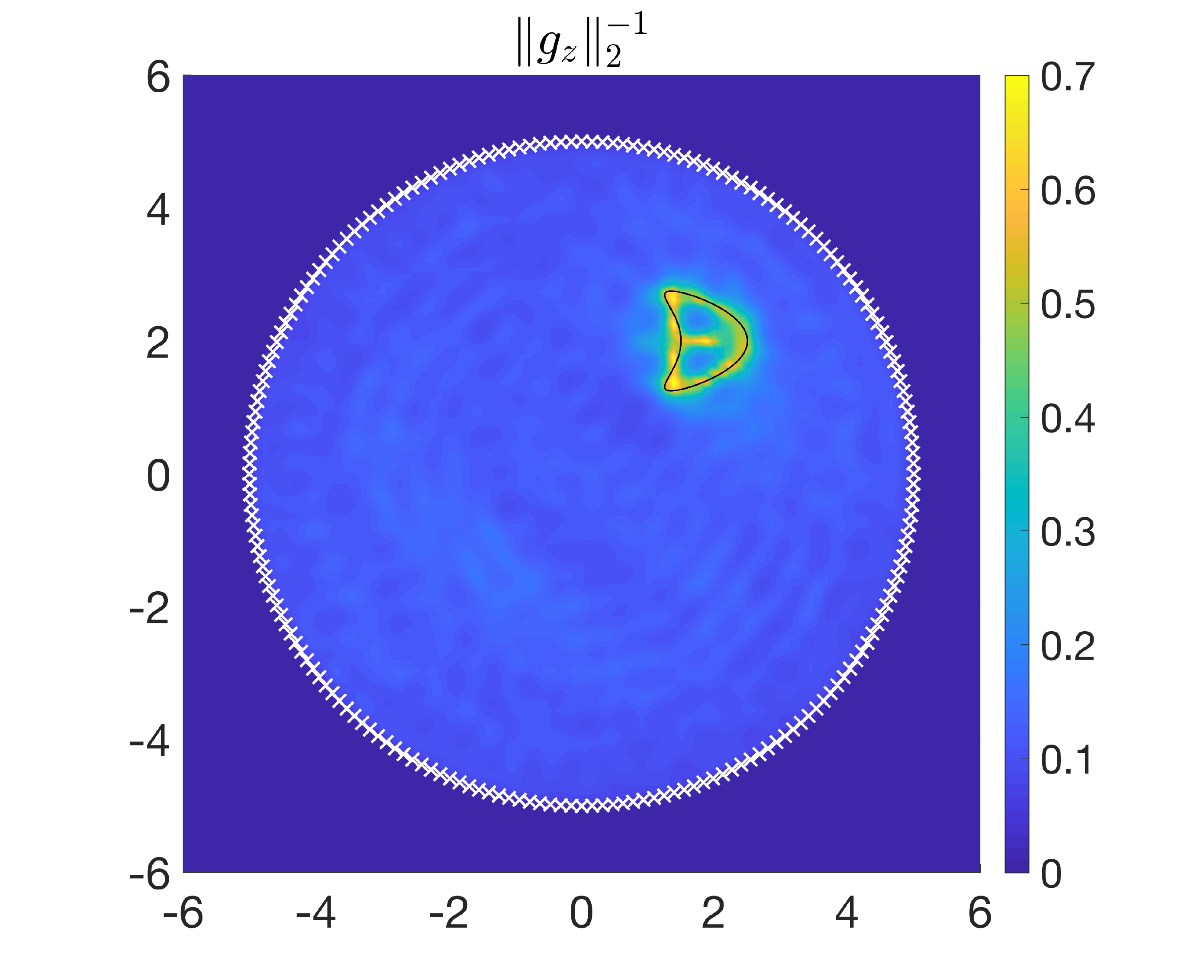}
\caption{Numerical results for a fixed number $L=80$ of random sources (first two columns) deteriorate as $\beta$ grows from $\beta=0.3$ (first row) to $\beta=0.6$ and $\beta =0.9$ (second/third rows). This is because the distribution of random sources is less and less equispaced as $\beta$ increases. Consequently, the entries \cref{eq:cross-cor-mat} of the matrix $C$, which are computed with the trapezoidal rule, become a less and less accurate approximation to the entries \cref{eq:imag-near-field-mat} of $I$ as $\beta$ grows. To improve accuracy, one may increase the number of sources---this is what we did in the third column, using $L=200$ random sources for all values of the parameter $\beta$.}
\label{fig:kite-beta}
\end{figure}

\paragraph{Influence of the shape of $\Sigma$ and the wavenumber $k$} The shape of $\Sigma$ where the random sources are positioned has little impact on the reconstructions. However, accurately computing its area is crucial to appropriately scale the matrix $C$ in \cref{eq:cross-cor-mat}. Regarding the wavenumber $k$, our numerical experiments revealed that doubling the value of $k$ necessitates doubling the number of point sources and measurement points; see \cref{fig:wavenumber}.

\begin{figure}
\centering
\def\scl{0.18}
\includegraphics[scale=\scl]{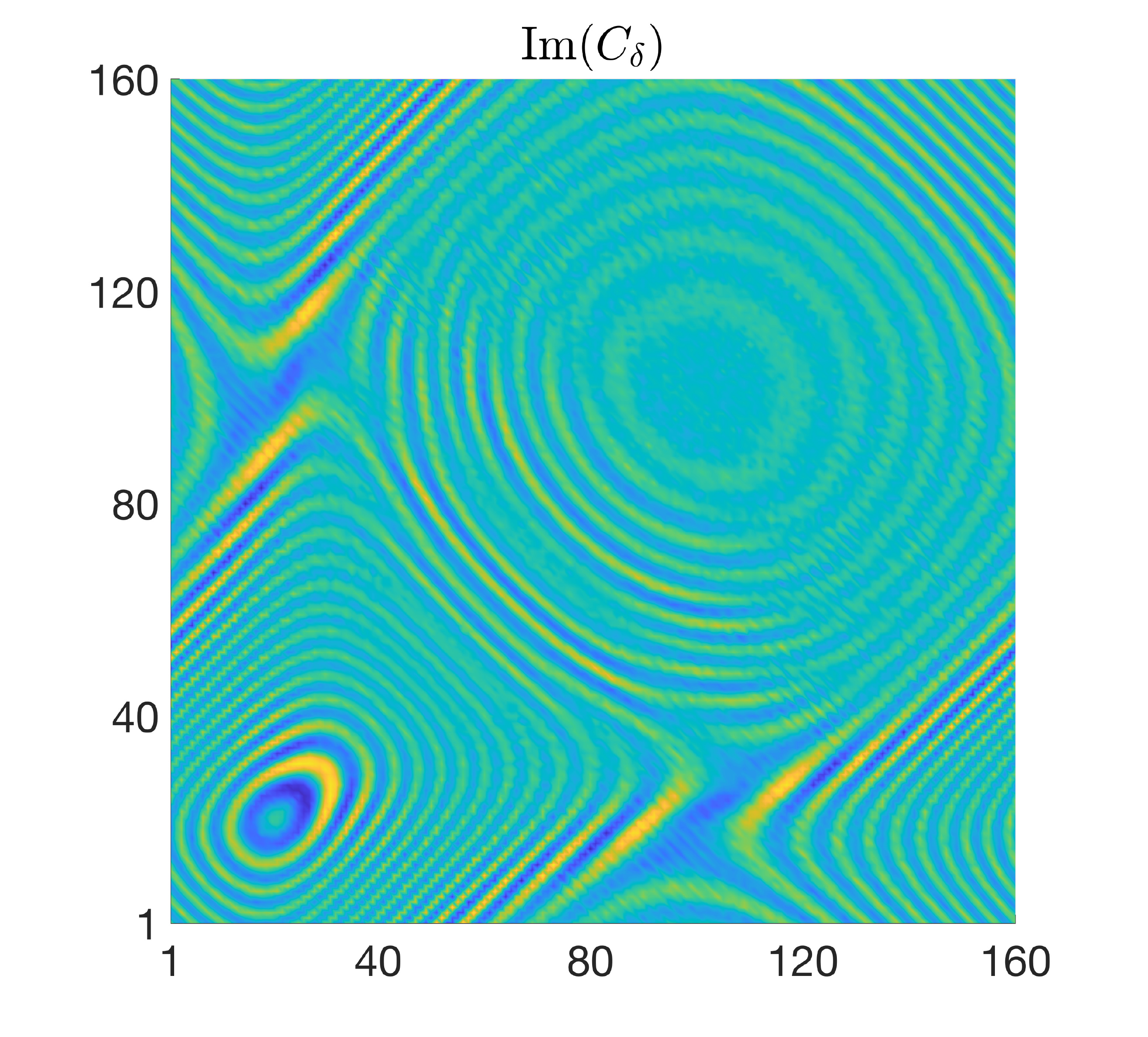} 
\includegraphics[scale=\scl]{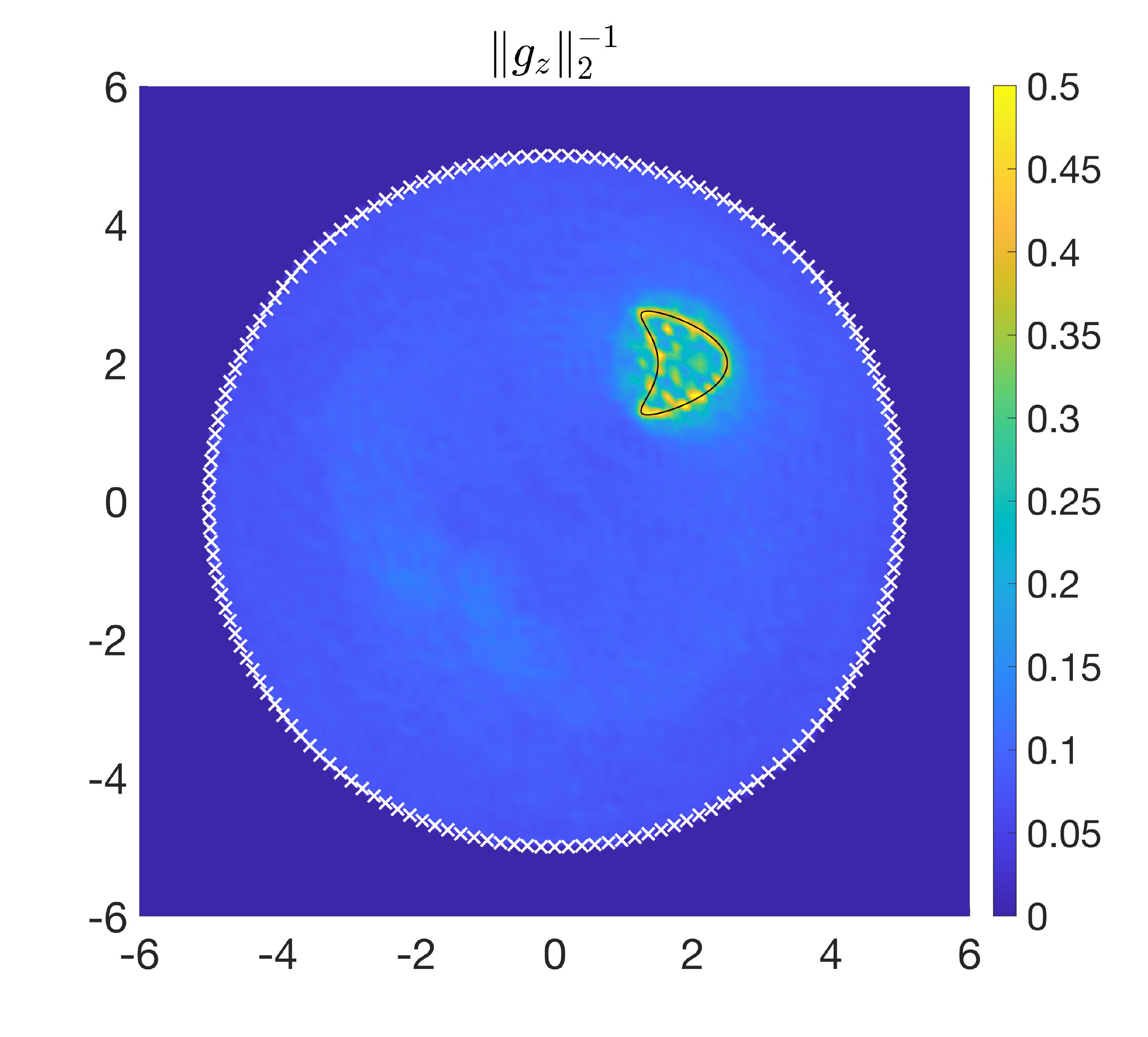} \\
\includegraphics[scale=\scl]{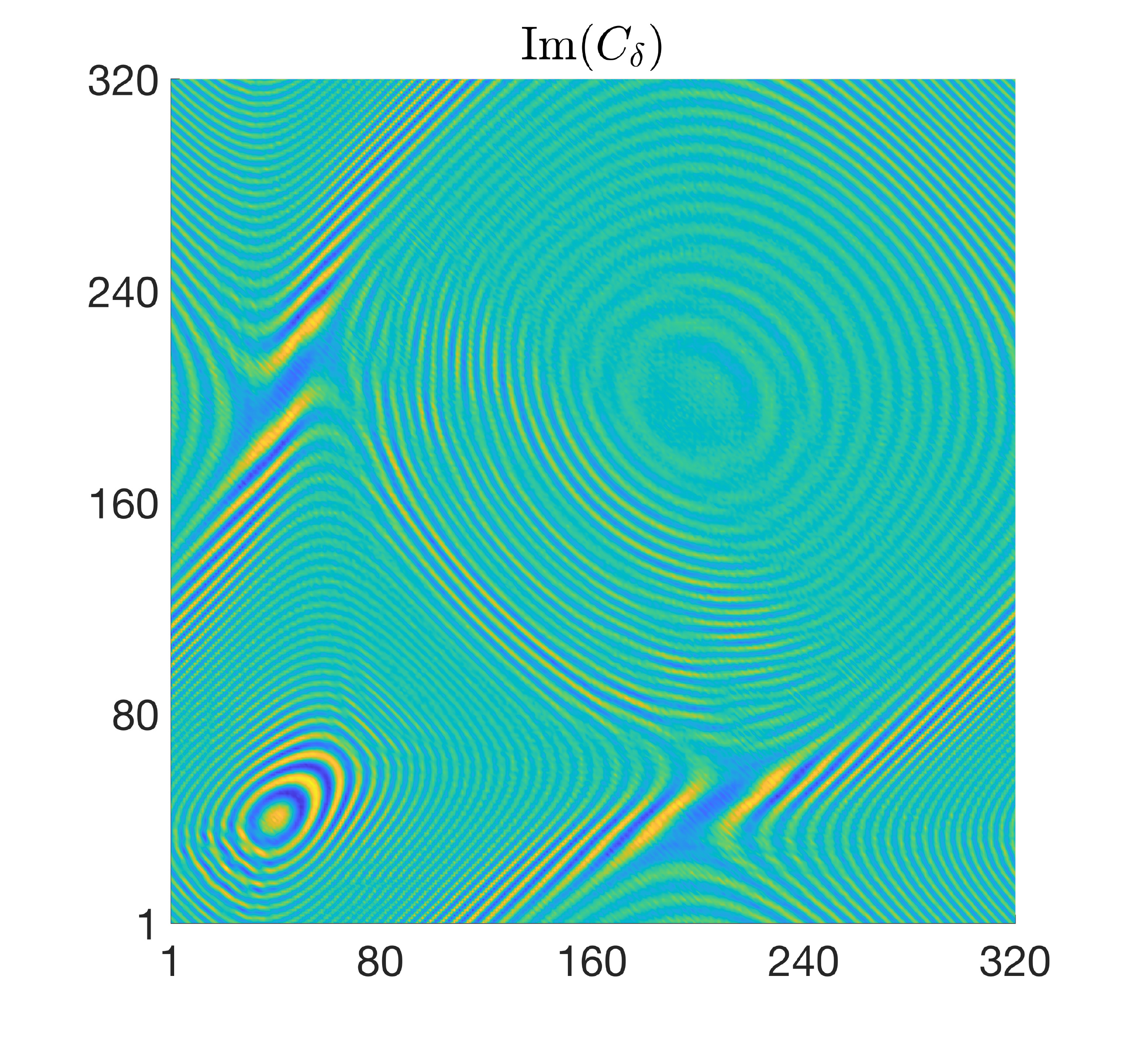} 
\includegraphics[scale=\scl]{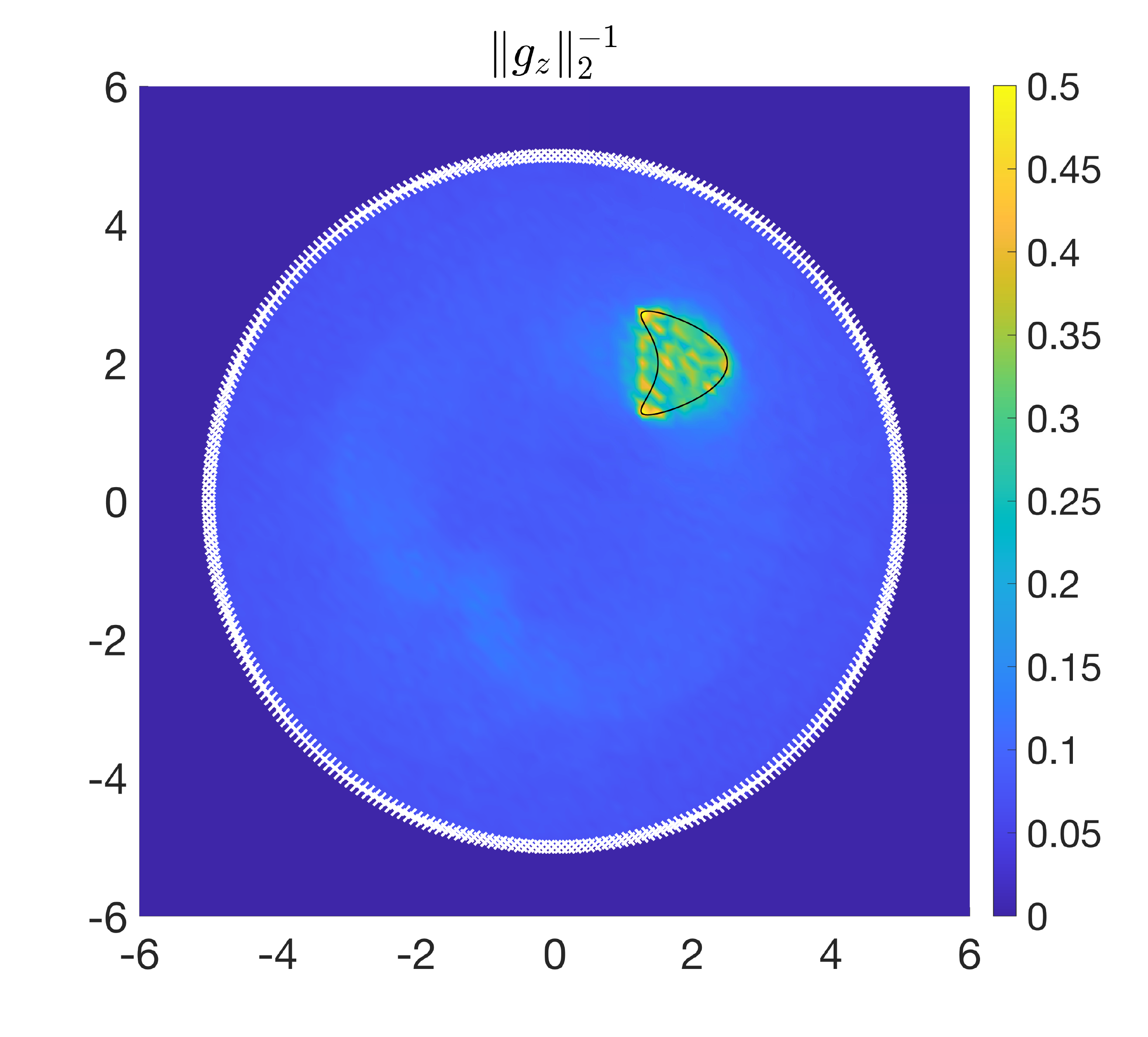}
\caption{In order to achieve accurate reconstructions at higher wavenumbers $k=4\pi$ (first row) and $k=8\pi$ (second row), we have found that it is essential to increase the number of point sources and measurement points to $160$ (first row) and $320$ (second row).}
\label{fig:wavenumber}
\end{figure}

\paragraph{Second setup} In the last experiment for full-aperture measurements, we test the second passive-image setup of \cref{sec:HK-id}. We consider the same measurement points $\bs{x}_j$ as before but this time we select deterministic sources $\bs{z}_\ell$ (by taking $\beta_\ell=0$ in \cref{eq:random-sources}), and compute the total fields $u(\bs{x}_j,\bs{z}_\ell)$. We then discretize the random process $\mathcal{U}(\bs{x})$ as follows:
\begin{align}
\mathcal{U}(\bs{x}) = \int_{\Sigma}u(\bs{x},\bs{z})n(\bs{z})dS(\bs{z}) \approx \sum_{\ell=1}^Lu(\bs{x},\bs{z}_\ell)n(\bs{z}_\ell).
\end{align}
To generate a realization of $\mathcal{U}(\bs{x})$, we draw $2L$ independent samples $(u_1,\ldots,u_L,v_1,\ldots,v_L)$ from the normal distribution $\mathcal{N}(0,\vert \Sigma\vert/2)$ and set $n(\bs{z}_\ell)=u_\ell+iv_\ell$, $1\leq\ell\leq L$. (The scaling of the distribution ensures that the sources verify \cref{eq:noise-distribution} so that \cref{eq:covariance} holds.) Finally, the statistical average is computed via $M$ realizations $\mathcal{U}^{(r)}(\bs{x})$ evaluated at $\bs{x}_j$ and $\bs{x}_m$:
\begin{align}\label{eq:covariance-approx}
\left<\mathcal{U}(\bs{x}_j)\overline{\mathcal{U}(\bs{x}_m)}\right> \approx \frac{1}{M}\sum_{r=1}^M\mathcal{U}^{(r)}(\bs{x}_j)\overline{\mathcal{U}^{(r)}(\bs{x}_m)}, \quad \quad 
1\leq j,m \leq J.
\end{align}
We show the results in \cref{fig:kite-setup2} for $J=L=M=200$---the results are not as good as for the first setup. The reason is that the empirical average in the right-hand side of \cref{eq:covariance-approx} is a rather poor approximation to the statistical average in the left-hand side (the relative error is proportional to $1/\sqrt{M}$). As we can see in \cref{fig:kite-setup2}  (left), the entries of the matrix are very noisy approximations to the entries of the matrix of \cref{fig:kite} (second row, left).

\begin{figure}
\centering
\def\scl{0.17}
\includegraphics[scale=\scl]{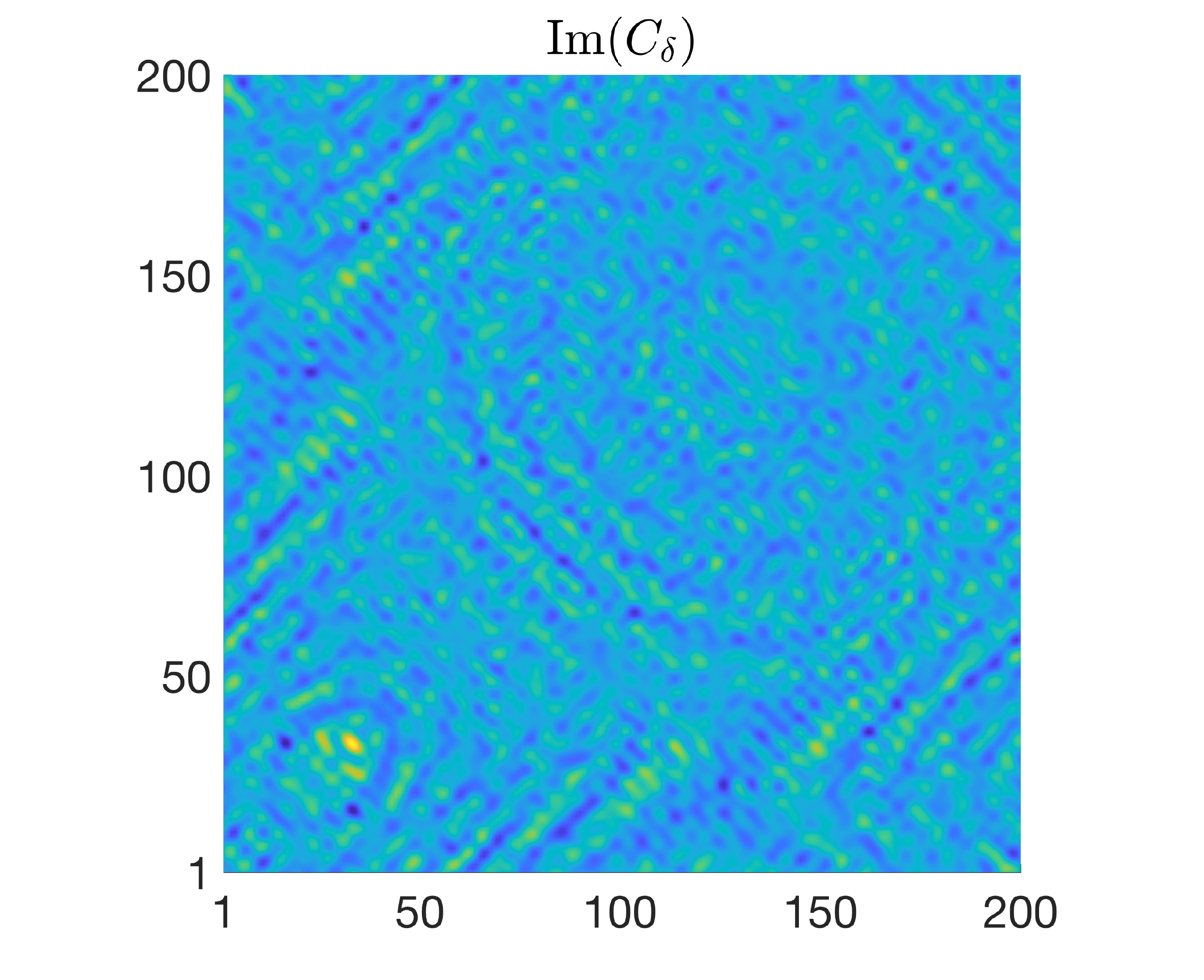} 
\includegraphics[scale=\scl]{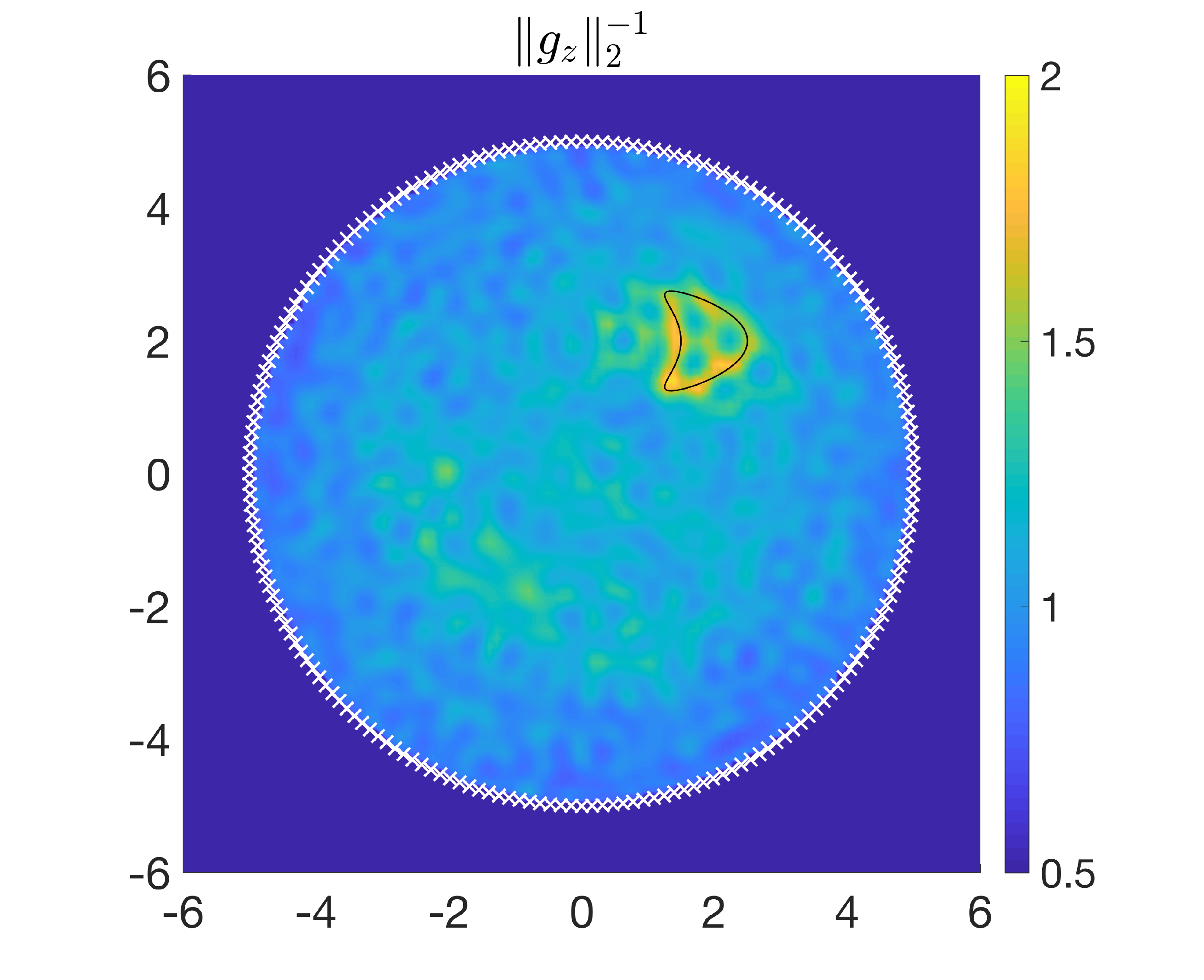}  
\includegraphics[scale=\scl]{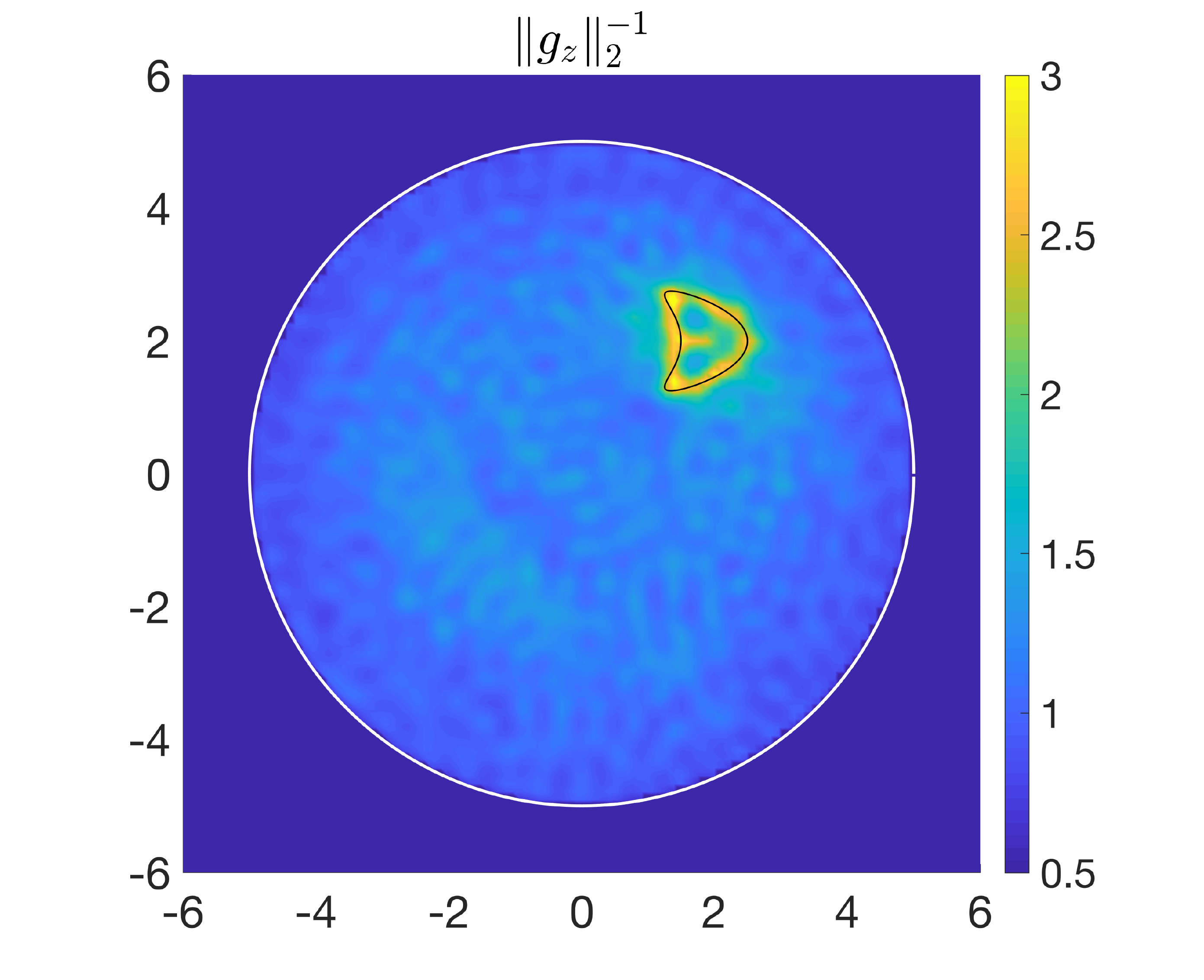}  
\caption{In the second setup, the entries of the cross-correlation matrix $C$ are computed via an empirical average from $M=200$ realizations, which generates an error $\OO(1/\sqrt{M})$ (left). The corresponding image is not as crisp as before (middle). Increasing the value of $M$ to $800$ improves the resolution (right).}
\label{fig:kite-setup2}
\end{figure}

\paragraph{Limited-aperture measurements} In this experiment, we consider limited-aperture measurements, as in \cite{audibert2017}. We show the results in \cref{fig:limi}. This setup yields poorer reconstructions for all three methods (it is a much harder problem), but our method based on the cross-correlation matrix $C$ gives comparable results to the method with $N$. The code that was used to generate the figures is available on the third author's \href{https://github.com/Hadrien-Montanelli/lsmlab}{GitHub} page.

\begin{figure}
\centering
\def\scl{0.18}
\includegraphics[scale=\scl]{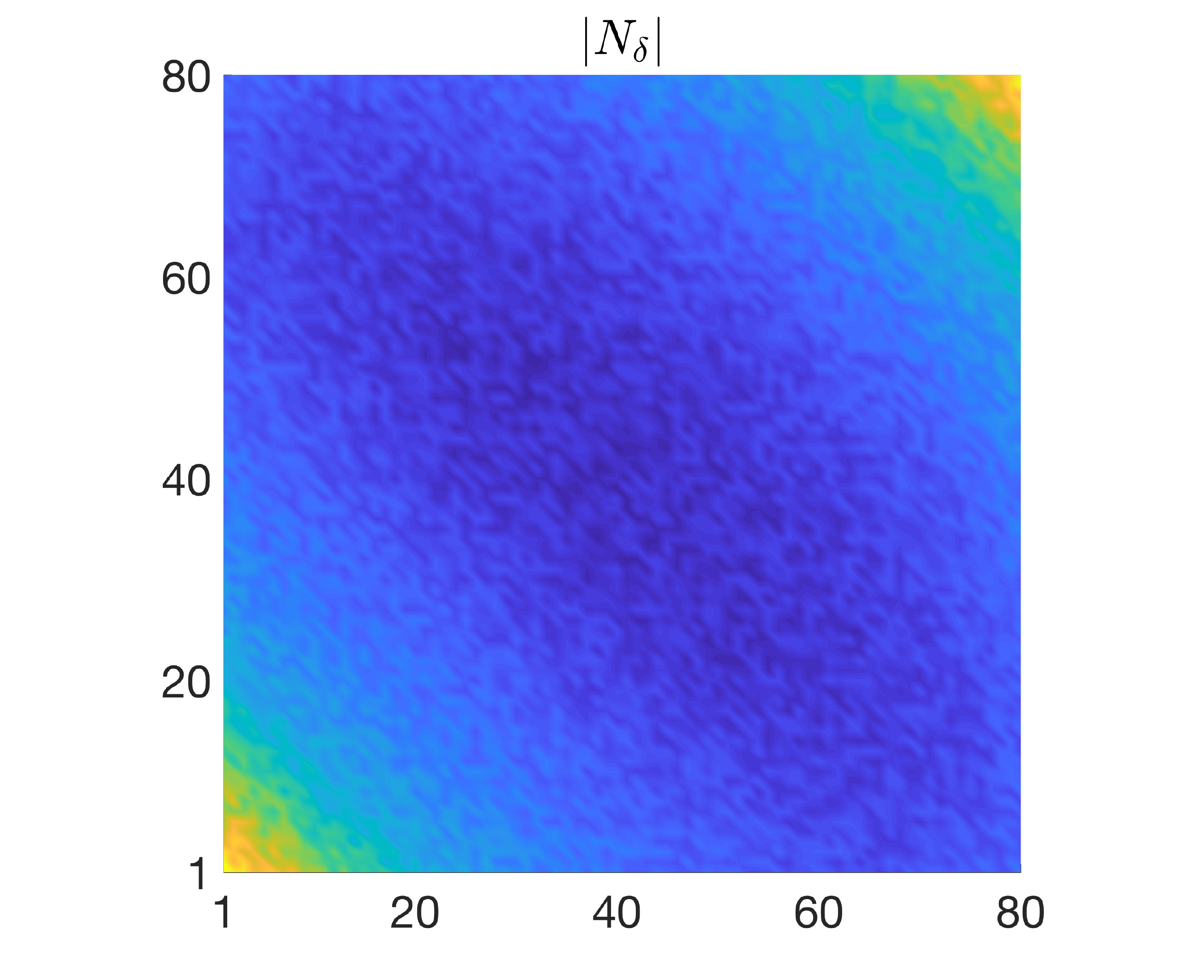} 
\includegraphics[scale=\scl]{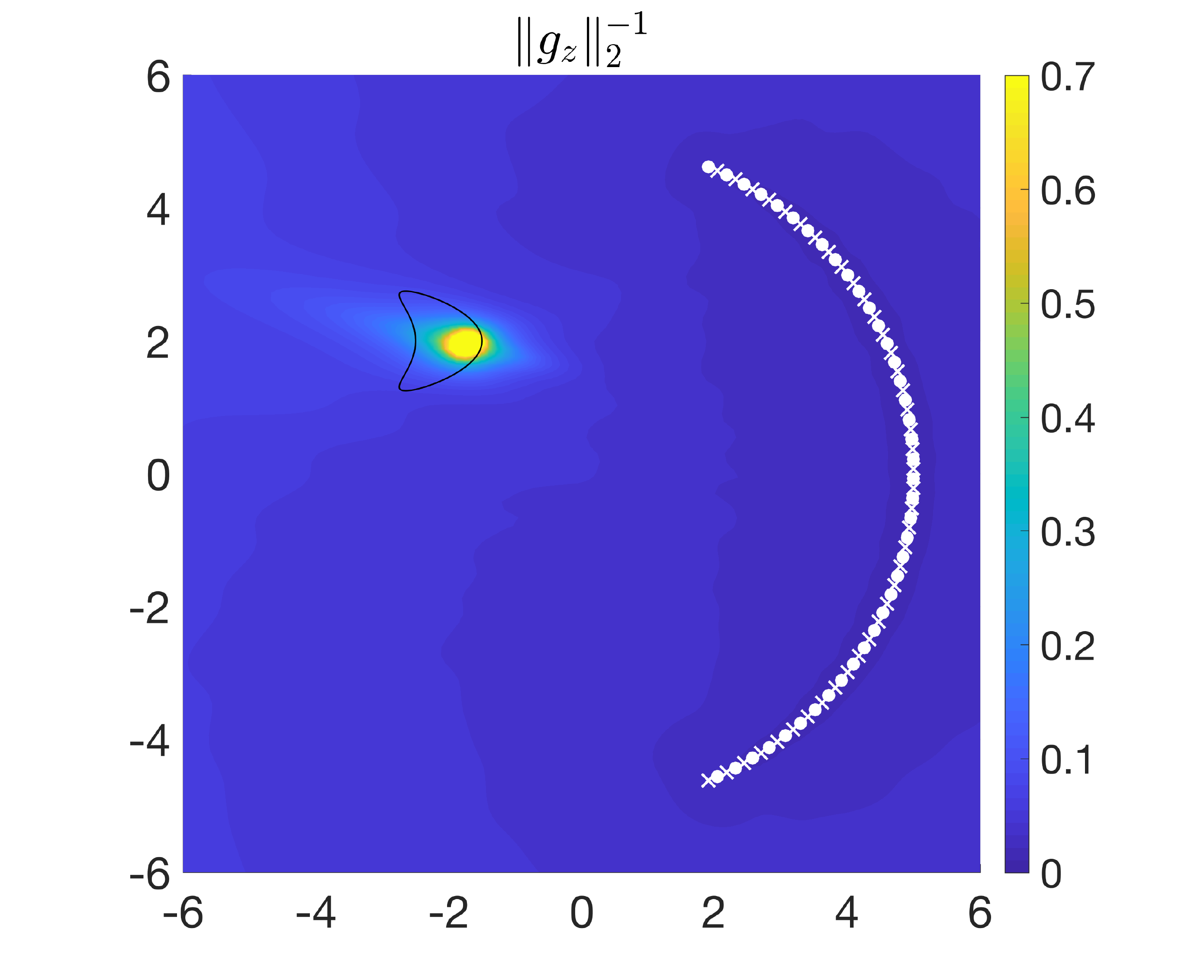} \\
\includegraphics[scale=\scl]{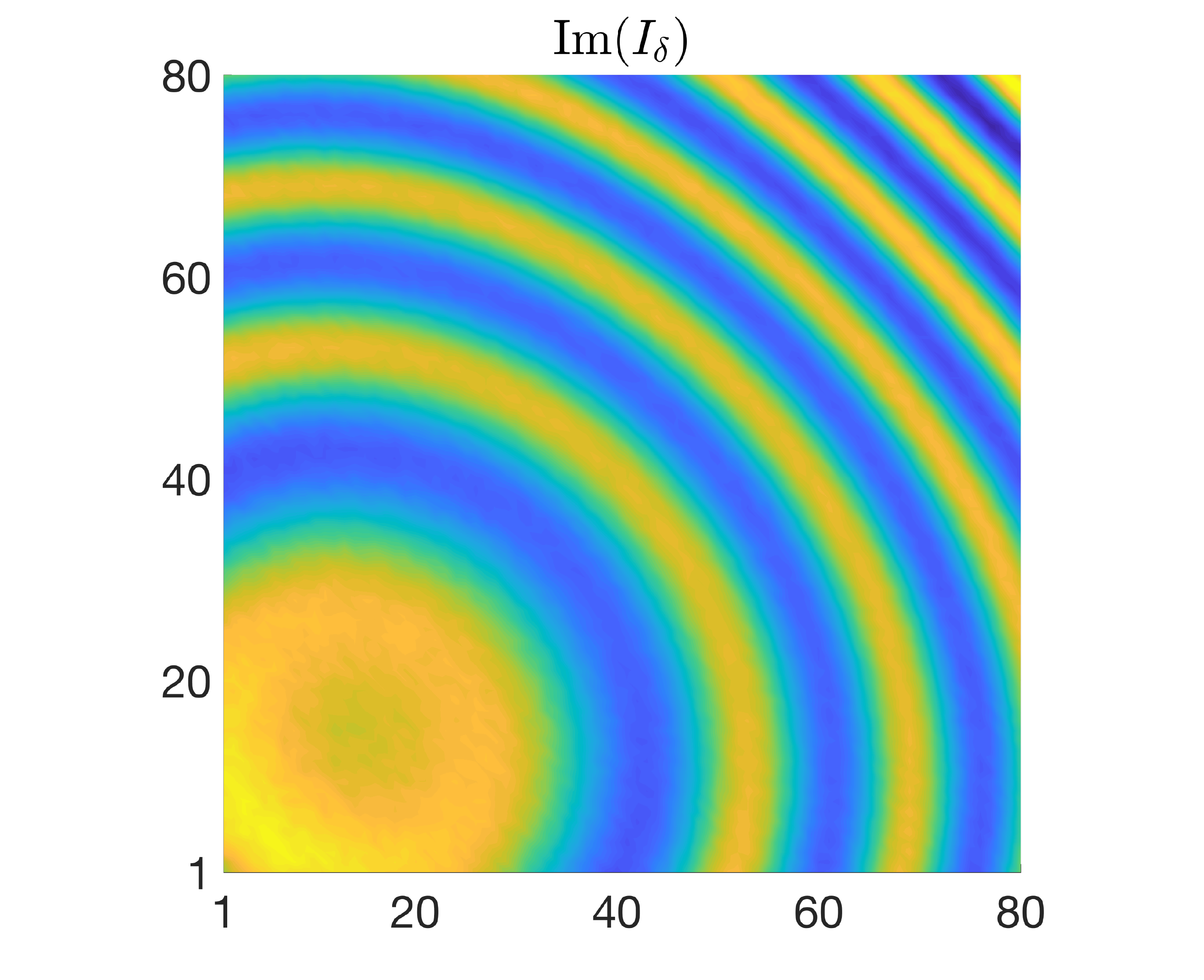} 
\includegraphics[scale=\scl]{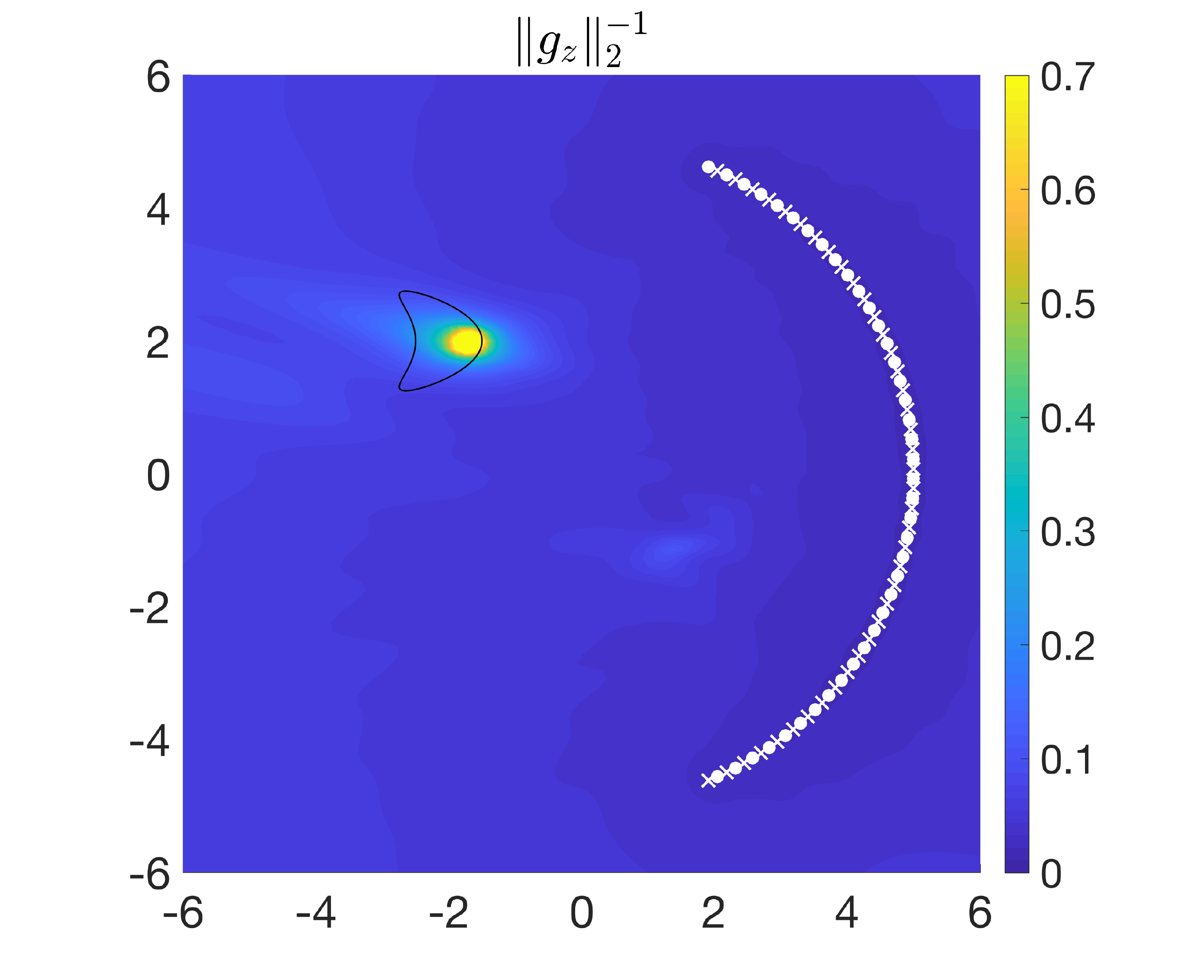}  \\
\includegraphics[scale=\scl]{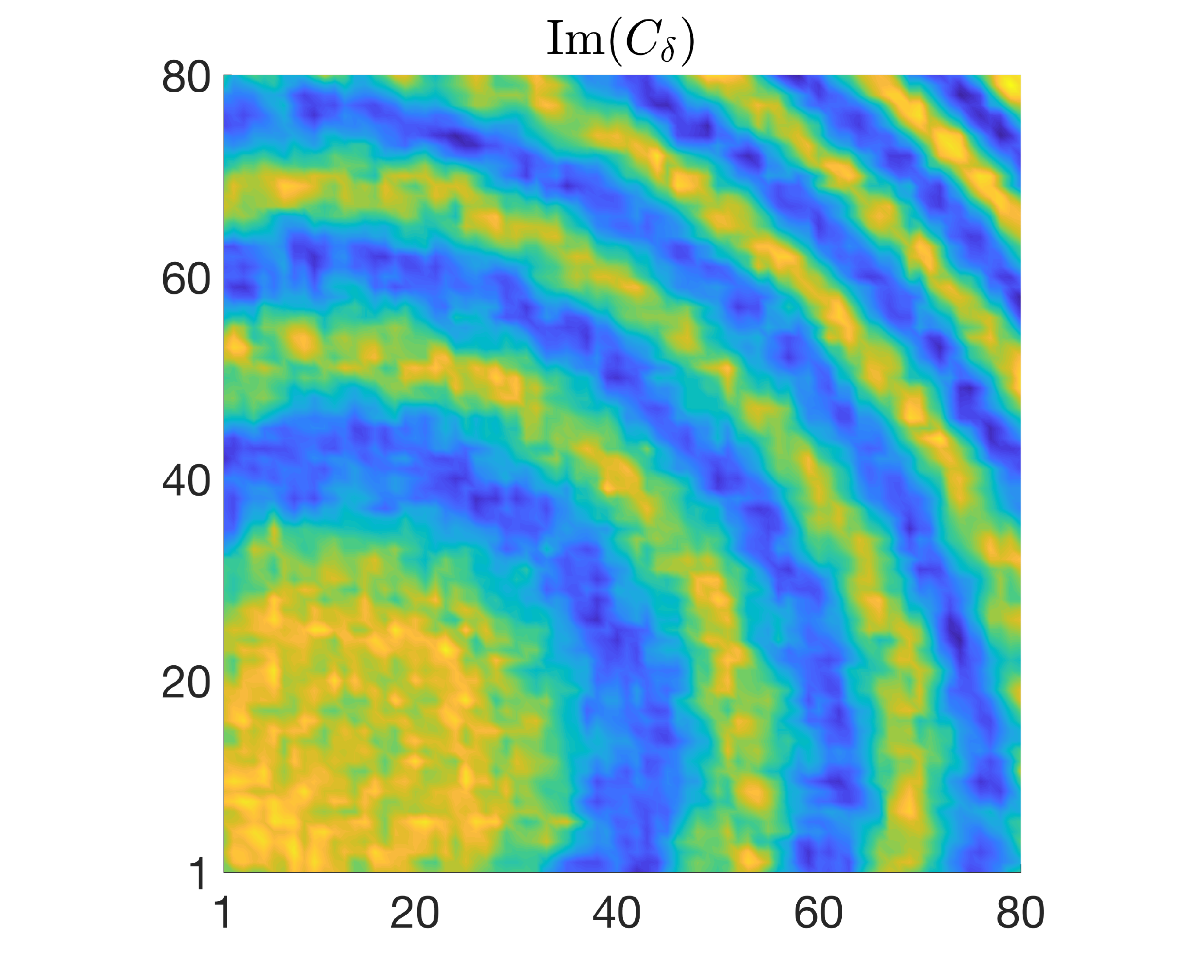} 
\includegraphics[scale=\scl]{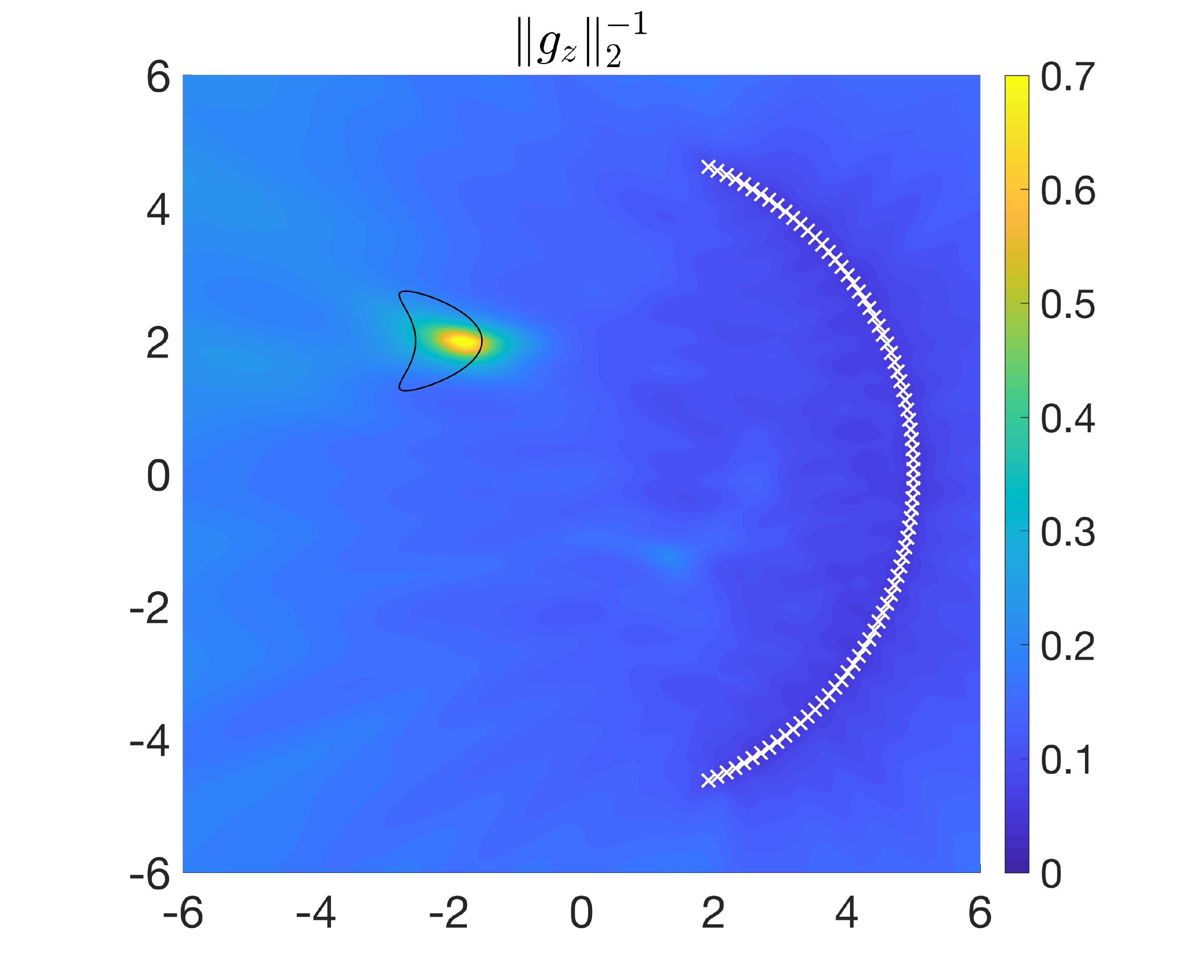}
\caption{For limited-aperture measurements, a perfect reconstruction of the shape should not be expected. However, our sampling method based on the first setup with cross-correlations (third row) gives comparable results to the LSM with the near-field and imaginary near-field matrices (first and second rows).}
\label{fig:limi}
\end{figure}

\section{Conclusions}

We have presented in this article an extension of the linear sampling method for solving the sound-soft inverse acoustic scattering problem with random sources. To prove our main theoretical result (\cref{thm:LSM-I}), we have followed the standard recipe, including factoring the near-field operator (\cref{thm:factorization}) and characterizing the range of the operator that maps boundary data to measurements (\cref{lem:range}). We have demonstrated the robustness and accuracy of our algorithms in \cref{sec:numerics} by considering both full- and limited-aperture measurements for different shapes.

How far are we from real-life applications? The primary hurdle is the assumption that the background is homogeneous and known, which may not be true in practical scenarios. This leads to noise in the data, the accurate modeling of which is nontrivial as it is not a simple additive noise \cite[Chap.~12]{garnier2016}. Moreover, the idealistic placement of sources needs to be customized for every experiment. Despite these limitations, the method is robust and presents compelling alternatives to existing techniques.

There are many ways in which this work could be profitably continued. For instance, one could look at the sound-hard inverse scattering problem---this entails using Neumann boundary conditions in \cref{eq:ext-Dirichlet}---as well as penetrable objects. One could also try and extend our procedure to configurations with deterministic sources and randomly distributed small scatterers (simulating a random medium) illuminating a defect.

\section*{Acknowledgments}
We thank the Interdisciplinary Centre for Defence and Security of the Polytechnic Institute of Paris for funding this work (PRODIPO project). We also thank the members of the Inria Idefix research team, in particular Lorenzo Audibert and Fabien Pourre, for fruitful discussions about the LSM. Finally, we are thankful to Julie Tran from Western University for her significant contribution to figure clarity.

\appendix
\section{Proof for the asymptotic model of small obstacles}\label{sec:appendix} Consider $L$ spheres of radii $r_\ell$ centered at points $\bs{c}_\ell\in\R^d$, $1\leq\ell\leq L$. For small radii $r_\ell$'s, the scattered field $u^s(\cdot,\bs{y})$ generated by a source point $\phi(\cdot,\bs{y})$ located at $\bs{y}$ may be approximated by 
\begin{align}
u^s(\bs{x},\bs{y}) \approx \sum_{\ell=1}^L\lambda_\ell\phi(\bs{c}_\ell,\bs{y})\phi(\bs{x},\bs{c}_\ell),
\end{align}
with reflection coefficients
\begin{align}
\lambda_\ell = \frac{4i}{H_0^{(1)}(kr_\ell)}\;(d=2), \quad \lambda_\ell = -4\pi r_\ell e^{-ik r_\ell}\;(d=3), \quad 1\leq\ell\leq L.
\end{align}
We refer to \cite{cassier2013} for details. In this case, the imaginary near-field operator $I$ has the form
\begin{align}
(Ig)(\bs{x}) \approx \sum_{\ell=1}^L\left[\Lambda_\ell(g)\phi(\bs{x},\bs{c}_\ell) - \overline{\Lambda_\ell(\overline{g})}\,\overline{\phi(\bs{x},\bs{c}_\ell)}\right],
\end{align}
with coefficients
\begin{align}
\Lambda_\ell(g) = \lambda_\ell\int_{B}\phi(\bs{c}_\ell,\bs{y})g(\bs{y})d\bs{y}, \quad 1\leq\ell\leq L.
\end{align}
Here the imaginary near-field equation $Ig_{\bs{z}}=\phi|_{B}(\cdot,\bs{z})$ reads
\begin{align}
\sum_{\ell=1}^L\left[\Lambda_\ell(g_{\bs{z}})\phi(\bs{x},\bs{c}_\ell) - \overline{\Lambda_\ell(\overline{g_{\bs{z}}})}\,\overline{\phi(\bs{x},\bs{c}_\ell)}\right] = \phi(\bs{x},\bs{z}), \quad \bs{x}\in B.
\label{eq:im-near-field}
\end{align}
Since both sides of \eqref{eq:im-near-field} satisfy the Helmholtz equation in $B$, it follows from the unique continuation principle that the equation \eqref{eq:im-near-field} is also valid for $\bs{x}\in\R^d\setminus\cup\{\bs{c}_\ell,\bs{z}\}$.

\begin{proposition} The imaginary near-field equation \eqref{eq:im-near-field} has a solution if and only if $\bs{z}\in\{\bs{c}_\ell,\,\ell=1,\ldots,L\}$. If $\bs{z}=\bs{c}_{\ell_0}$ for some $1\leq\ell_0\leq L$ and $g_{\bs{z}}$ is a solution of \eqref{eq:im-near-field}, then 
\begin{align}
& \int_{B}\phi(\bs{c}_{\ell_0},\bs{y})g_{\bs{z}}(\bs{y})d\bs{y} = \frac{1}{\lambda_{\ell_0}}, && \ell=\ell_0, \\
& \int_{B}\phi(\bs{c}_\ell,\bs{y})g_{\bs{z}}(\bs{y})d\bs{y} = 0, && \ell\neq\ell_0, \\
& \int_{B}\phi(\bs{c}_\ell,\bs{y})\overline{g_{\bs{z}}(\bs{y})}d\bs{y} = 0, && \forall\ell.
\end{align}
\end{proposition} 

\begin{proof} The proof is similar to that of \cite[Thm.~2]{haddar2012}. Firstly, we note that if $\bs{z}\neq\bs{c}_\ell$ for all $1\leq\ell\leq L$, then we cannot obtain a solution of \eqref{eq:im-near-field} since the left-hand side remains bounded as $\bs{x}$ approaches $\bs{z}$, while the right-hand side is singular. Secondly, we observe that the $\phi(\cdot,\bs{c}_\ell)$'s and $\overline{\phi(\cdot,\bs{c}_\ell)}$'s are linearly independent functions. Lastly, if $\bs{z}=\bs{c}_{\ell_0}$ for some $\ell_0$ and $g_{\bs{z}}$ solves \eqref{eq:im-near-field}, then the independence yields
\begin{align}
& \Lambda_{\ell_0}(g_{\bs{z}}) = \lambda_{\ell_0}\int_{B}\phi(\bs{c}_{\ell_0},\bs{y})g_{\bs{z}}(\bs{y})d\bs{y} = 1, && \ell=\ell_0, \\
& \Lambda_{\ell}(g_{\bs{z}}) = \lambda_\ell\int_{B}\phi(\bs{c}_\ell,\bs{y})g_{\bs{z}}(\bs{y})d\bs{y} = 0, && \ell\neq\ell_0, \\
& \overline{\Lambda_{\ell}(\overline{g_{\bs{z}}})} = \overline{\lambda_\ell}\int_{B}\overline{\phi(\bs{c}_\ell,\bs{y})}g_{\bs{z}}(\bs{y})d\bs{y} = 0, && \forall\ell.
\end{align}
For the existence of solutions, we may take a function $g_{\bs{z}}$ that is a linear combination of the $\phi(\cdot,\bs{c}_\ell)$'s and $\overline{\phi(\cdot,\bs{c}_\ell)}$'s such that the previous system of equations is satisfied.
\end{proof}

\bibliographystyle{siamplain}
\bibliography{references.bib}

\begin{thebibliography}{10}

\bibitem{ammari2012b}
{\sc H.~Ammari, J.~Garnier, V.~Jugnon, and H.~Kang}, {\em Stability and
  resolution analysis for a topological derivative based imaging functional},
  SIAM J. Control Optim., 50 (2012), pp.~48--76.

\bibitem{ammari2012a}
{\sc H.~Ammari, J.~Garnier, H.~Kang, M.~Lim, and K.~S{\o}lna}, {\em Multistatic
  imaging of extended targets}, SIAM J. Imaging Sci., 5 (2012), pp.~564--600.

\bibitem{audibert2014}
{\sc L.~Audibert and H.~Haddar}, {\em A generalized formulation of the linear
  sampling method with exact characterization of targets in terms of farfield
  measurements}, Inverse Probl., 30 (2014), p.~035011.

\bibitem{audibert2017}
{\sc L.~Audibert and H.~Haddar}, {\em The generalized linear sampling method
  for limited aperture measurements}, SIAM J. Imaging Sci., 10 (2017),
  pp.~845--870.

\bibitem{austin2017b}
{\sc A.~P. Austin and L.~N. Trefethen}, {\em Trigonometric interpolation and
  quadrature in perturbed points}, SIAM J. Numer. Anal., 55 (2017),
  pp.~2113--2122.

\bibitem{austin2017a}
{\sc A.~P. Austin and K.~Xu}, {\em On the numerical stability of the second
  barycentric formula for trigonometric interpolation in shifted equispaced
  points}, IMA J. Numer. Anal., 37 (2017), pp.~1355--1374.

\bibitem{bao2007}
{\sc G.~Bao, S.~Hou, and P.~Li}, {\em Inverse scattering by a continuation
  method with initial guesses from a direct imaging algorithm}, J. Comput.
  Phys., 227 (2007), pp.~755--762.

\bibitem{bourgeois2012}
{\sc L.~Bourgeois, N.~Chaulet, and H.~Haddar}, {\em On simultaneous
  identification of the shape and generalized impendance boundary conditions in
  obstacle scattering}, SIAM J. Sci. Comput., 34 (2012), pp.~A1824--A1848.

\bibitem{cakoni2014}
{\sc F.~Cakoni and D.~Colton}, {\em A Qualitative Approach to Inverse
  Scattering Theory}, Applied Mathematical Sciences, Springer, New York, 2014.

\bibitem{cakoni2016}
{\sc F.~Cakoni, D.~Colton, and H.~Haddar}, {\em Inverse Scattering Theory and
  Transmission Eigenvalues}, CBMS-NSF Regional Conference Series on
  Mathematics, SIAM, Philadelphia, 2016.

\bibitem{cassier2013}
{\sc M.~Cassier and C.~Hazard}, {\em Multiple scattering of acoustic waves by
  small sound-soft obstacles in two dimensions: {Mathematical justification of
  the Foldy--Lax model}}, Wave Motion, 50 (2013), pp.~18--28.

\bibitem{colton2003}
{\sc D.~Colton, H.~Haddar, and M.~Piana}, {\em The linear sampling method in
  inverse electromagnetic scattering theory}, Inverse Probl., 19 (2003),
  pp.~S105--S137.

\bibitem{colton1996}
{\sc D.~Colton and A.~Kirsch}, {\em A simple method for solving inverse
  scattering problems in the resonance region}, Inverse Probl., 12 (1996),
  pp.~383--393.

\bibitem{colton2018}
{\sc D.~Colton and R.~Kress}, {\em Looking back on inverse scattering theory},
  SIAM Rev., 60 (2018), pp.~779--807.

\bibitem{colton2019}
{\sc D.~Colton and R.~Kress}, {\em Inverse Acoustic and Electromagnetic
  Scattering Theory}, Springer, New York, 4th~ed., 2019.

\bibitem{colton1997}
{\sc D.~Colton, M.~Piana, and R.~Potthast}, {\em A simple method using
  {Morozov's} discrepancy principle for solving inverse scattering problems},
  Inverse Probl., 13 (1997), pp.~1477--1493.

\bibitem{curtis2006}
{\sc A.~Curtis, P.~Gerstoft, H.~Sato, R.~Snieder, and K.~Wapenaar}, {\em
  Seismic interferometry turning noise into signal}, Lead. Edge, 25 (2006),
  pp.~1082--1092.

\bibitem{dorn2006}
{\sc O.~Dorn and D.~Lesselier}, {\em Level set methods for inverse scattering},
  Inverse Probl., 22 (2006), pp.~R67--R131.

\bibitem{duroux2010}
{\sc A.~Duroux, K.~Sabra, J.~Ayers, and M.~Ruzzene}, {\em Using
  cross-correlations of elastic diffuse fields for attenuation tomography of
  structural damage}, J. Acoust. Soc. Am., 127 (2010), pp.~3311--3314.

\bibitem{gallot2011}
{\sc T.~Gallot, S.~Catheline, P.~Roux, J.~Brum, N.~Benech, and C.~Negreira},
  {\em Passive elastography: shear-wave tomography from physiological-noise
  correlation in soft tissues}, IEEE Trans. Ultrason. Ferroelectr. Freq.
  Control, 58 (2011), pp.~1122--1126.

\bibitem{garnier2009}
{\sc J.~Garnier and G.~Papanicolaou}, {\em Passive sensor imaging using cross
  correlations of noisy signals in a scattering medium}, SIAM J. Imaging Sci.,
  2 (2009), pp.~396--437.

\bibitem{garnier2010}
{\sc J.~Garnier and G.~Papanicolaou}, {\em Resolution analysis for imaging with
  noise}, Inverse Probl., 26 (2010), p.~074001.

\bibitem{garnier2016}
{\sc J.~Garnier and G.~Papanicolaou}, {\em Passive Imaging with Ambient Noise},
  Cambridge University Press, Cambridge, 2016.

\bibitem{godin2010}
{\sc O.~A. Godin, N.~A. Zabotin, and V.~V. Goncharov}, {\em Ocean tomography
  with acoustic daylight}, Geophys. Res. Lett., 37 (2010), p.~L13605.

\bibitem{gouedard2008}
{\sc P.~Gou\'edard, L.~Stehly, F.~Brenguier, M.~Campillo, Y.~Colin~de
  Verdi\`ere, E.~Larose, L.~Margerin, P.~Roux, F.~J. Sanchez-Sesma, N.~M.
  Shapiro, and R.~L. Weaver}, {\em Cross-correlation of random fields:
  {M}athematical approach and applications}, Geophys. Prospect., 56 (2008),
  pp.~375--393.

\bibitem{haddar2012}
{\sc H.~Haddar and R.~Mdimagh}, {\em Identification of small inclusions from
  multistatic data using the reciprocity gap concept}, Inverse Probl., 28
  (2012), p.~045011.

\bibitem{hohage1998}
{\sc T.~Hohage}, {\em Convergence rates of a regularized {Newton} method in
  sound-hard inverse scattering}, SIAM J. Numer. Anal., 36 (1998),
  pp.~125--142.

\bibitem{kirsch1993}
{\sc A.~Kirsch}, {\em The domain derivative and two applications in inverse
  scattering theory}, Inverse Probl., 9 (1993), pp.~81--96.

\bibitem{kirsch2007}
{\sc A.~Kirsch and N.~Grinberg}, {\em The Factorization Method for Inverse
  Problems}, Oxford Lecture Series in Mathematics and its Applications, Oxford
  University Press, Oxford, 2007.

\bibitem{koulakov2014}
{\sc I.~Koulakov and N.~Shapiro}, {\em Seismic Tomography of Volcanoes},
  Springer, Berlin, 2014, pp.~1--18.

\bibitem{montanelli2017phd}
{\sc H.~Montanelli}, {\em Numerical Algorithms for Differential Equations with
  Periodicity}, PhD thesis, University of Oxford, 2017.

\bibitem{montanelli2022}
{\sc H.~Montanelli, M.~Aussal, and H.~Haddar}, {\em Computing weakly singular
  and near-singular integrals over curved boundary elements}, SIAM J. Sci.
  Comput., 44 (2022), pp.~A3728--A3753.

\bibitem{montanelli2020b}
{\sc H.~Montanelli and N.~Bootland}, {\em {Solving periodic semilinear stiff
  PDEs in $1\mrm{D}$, $2\mrm{D}$ and $3\mrm{D}$ with exponential integrators}},
  Math. Comput. Simul., 178 (2020), pp.~307--327.

\bibitem{sabra2011}
{\sc K.~G. Sabra and S.~Huston}, {\em Passive structural health monitoring of a
  high-speed naval ship from ambient vibrations}, J. Acoust. Soc. Am., 129
  (2011), pp.~2991--2999.

\bibitem{shapiro2005}
{\sc N.~Shapiro, M.~Campillo, L.~Stehly, and M.~H. Ritzwoller}, {\em
  High-resolution surface-wave tomography from ambient seismic noise}, Science,
  307 (2005), pp.~1615--1618.

\bibitem{siderius2010}
{\sc M.~Siderius, H.~Song, P.~Gerstoft, W.~S. Hodgkiss, P.~Hursky, and C.~H.
  Harrison}, {\em Adaptive passive fathometer processing}, J. Acoust. Soc. Am.,
  127 (2010), pp.~2193--2200.

\bibitem{spence2015}
{\sc E.~A. Spence, I.~V. Kamotski, and V.~P. Smyshlyaev}, {\em Coercivity of
  combined boundary integral equations in high-frequency scattering}, Comm.
  Pure Appl. Math., 68 (2015).

\bibitem{trefethen2019}
{\sc L.~N. Trefethen}, {\em Approximation Theory and Approximation Practice},
  SIAM, Philadelphia, extended~ed., 2019.

\bibitem{trefethen2014}
{\sc L.~N. Trefethen and J.~A.~C. Weideman}, {\em The exponentially convergent
  trapezoidal rule}, SIAM Rev., 56 (2014), pp.~385--458.

\bibitem{woolfe2015}
{\sc K.~F. Woolfe, S.~Lani, K.~G. Sabra, and W.~A. Kuperman}, {\em Monitoring
  deep-ocean temperatures using acoustic ambient noise}, Geophys. Res. Lett.,
  42 (2015), pp.~2878--2884.

\bibitem{montanelli2015b}
{\sc G.~B. Wright, M.~Javed, H.~Montanelli, and L.~N. Trefethen}, {\em
  Extension of {C}hebfun to periodic functions}, SIAM J. Sci. Comput., 37
  (2015), pp.~C554--C573.

\end{thebibliography}

\end{document}